\documentclass[a4paper]{amsart}

\usepackage{a4wide}

\usepackage{jk-macros}

\title[Higher Nakayama algebras I]{Higher Nakayama algebras I: Construction}

\author[G.~Jasso]{Gustavo Jasso}
\address[Jasso]{Mathematisches Institut\\
  Universit\"at Bonn\\
  Endenicher Allee 60\\
  D-53115 Bonn\\
  Germany}%
\email{gjasso@math.uni-bonn.de}%
\urladdr{http://gustavo.jasso.info}

\author[J.~K\"ulshammer]{Julian K\"ulshammer}
\address[K\"ulshammer]{Department of Mathematics\\
  Uppsala University\\
  Box 480\\
  751 06 Uppsala\\
  Sweden}%
\email{julian.kuelshammer@math.uu.se}%
\urladdr{https://katalog.uu.se/profile/?id=N18-1115}%

\address[Psaroudakis]{Department of Mathematics\\
  Aristotle University of Thessaloniki\\
  54124 Thessaloniki\\
  Greece}%
\email{chpsaroud@math.auth.gr}%
\urladdr{https://sites.google.com/view/chrysostomos-psaroudakis/}%

\address[Kvamme]{Sondre Kvamme\\
Laboratoire de math\'ematiques d'Orsay\\
Universit\'e Paris-Sud, CNRS\\
Universit\'e Paris-Saclay\\
91405 Orsay\\
France
}%
\email{sondre.kvamme@u-psud.fr}%
\urladdr{https://www.math.u-psud.fr/~kvamme}%

\makeatletter
\let\@wraptoccontribs\wraptoccontribs
\makeatother

\contrib[with an appendix by]{Julian K\"ulshammer and Chrysostomos Psaroudakis}
\contrib[and an appendix by]{Sondre Kvamme}

\begin{document}

\date{\today}

\begin{abstract}
  We introduce higher dimensional analogues of the Nakayama algebras from the
viewpoint of Iyama's higher Auslander--Reiten theory. More precisely, for each
Na\-ka\-ya\-ma algebra $A$ and each positive integer $d$, we construct a finite
dimensional algebra $A^{(d)}$ having a distinguished $d$-cluster-tilting
$A^{(d)}$-module whose endomorphism algebra is a higher dimensional analogue of
the Auslander algebra of $A$. We also construct higher dimensional analogues of
the mesh category of type $\mathbb{ZA}_\infty$ and the tubes.


\end{abstract}



\thanks{The authors would like to thank Erik Darp{\"o} and Osamu Iyama for
  sharing a preliminary version of their preprint \cite{DI17}. We also thank
  Martin Herschend for useful conversations surrounding the combinatorics used
  in this article. Finally, the authors thank the anonymous referee for their
  numerous suggestions and corrections which improved the presentation of the article.}

\maketitle

\setcounter{tocdepth}{1}
\tableofcontents


\section*{Introduction}

Let $A$ be a finite dimensional algebra over some field, which for simplicity we
assume to be algebraically closed. The most basic problem in representation
theory is to classify the indecomposable finite dimensional (right) $A$-modules.
Since the general problem is known to be hopeless, special attention has been
paid to finite dimensional algebras of finite representation type. These are the
algebras for which there exist only finitely many isomorphism classes of
indecomposable finite dimensional $A$-modules. Classifying such algebras is
still a daunting problem. This lead to the study of algebras satisfying
additional homological constraints. An important class of such algebras is that
of hereditary algebras of finite representation type, which were classified by
Gabriel in \cite{Gab72}. Another significant class is that of selfinjective
algebras, which includes the group algebras of finite groups. Selfinjective
algebras of finite representation type were classified by Riedtmann in
\cite{Rie80}.

Yet another important class of algebras of finite representation type is that of
Nakayama algebras, see for example Chapter IV.2 in \cite{ARS97} and Chapter V in
\cite{ASS06}. In relation to the above, it is worth noting that symmetric
Nakayama algebras are stably equivalent to those Brauer tree algebras which
arise from blocks of group algebras of finite representation type.

The classification of the Nakayama algebras over an algebraically closed field
up to Morita equivalence is straightforward. For each non-negative integer $n$,
the path algebra of the linearly oriented quiver of type $\mathbb{A}_n$ is the
unique (basic, connected) \emph{hereditary} Nakayama algebra having $n$ simple
modules. Similarly, for each positive integer $n$ and each $\ell\geq2$ the
quotient of the path algebra of the circular quiver with $n$ vertices and $n$
arrows modulo the ideal generated by all paths of length $\ell$ is the unique
(basic, connected) \emph{selfinjective} Nakayama algebra having $n$ simple
modules whose projective covers have Loewy length $\ell$. The (basic,
connected) Nakayama algebras are the \emph{admissible quotients} of the
hereditary and the selfinjective Nakayama algebras.

The purpose of this article is to construct, for each (basic, connected)
Nakayama algebra $A$ and each positive integer $d\geq2$, a finite dimensional
algebra $A^{(d)}$ which we regard as a `$d$-dimensional analogue' of $A$. Our
construction is similar in spirit to that of the classical Nakayama algebras. We
consider finite dimensional algebras $A_n^{(d)}$ and
$\widetilde{A}_{n-1,\ell}^{(d)}$ which we think of as higher dimensional
analogues of the hereditary and of the selfinjective Nakayama algebras (see
\th\ref{def:An,def:A-tilde:selfinjective}). The first family of algebras was
introduced by Iyama in \cite{Iya11} while the second one is constructed using
recent results from Darp\"o and Iyama \cite{DI17}. Every other higher Nakayama
algebra is then constructed as an \emph{idempotent quotient} of a higher
Nakayama algebra of the form $A_n^{(d)}$ or $\widetilde{A}_{n-1,\ell}^{(d)}$
(see \th\ref{def:An:l,def:A-tilde}). The idempotent quotients that arise are
determined by the combinatorial data classifying the classical Nakayama
algebras, their so-called Kupisch series.

The analogy between the classical Nakayama algebras and the higher Nakayama
algebras is better appreciated in the context of Iyama's higher
Auslander--Reiten theory \cite{Iya07,Iya07a}, which we review in Section
\ref{subsec:higher_AR_theory}. Instead of algebras of finite representation type
one considers algebras $A$ having a so-called $d$-cluster tilting $A$-module
$M$. From the viewpoint of this theory, the additive closure of $M$ in $\mmod
A$, denoted by $\add M$, is considered to be a replacement of the module
category. This is justified by the existence of a `$(d+1)$-dimensional'
Auslander--Reiten theory \emph{inside} $\add M$ (see Subsection
\ref{subsec:higher_AR_theory} for definitions and precise statements).
Constructing algebras having a $d$-cluster-tilting module is a difficult task,
nonetheless higher Auslander--Reiten theory has found connections and
applications in subjects such as algebraic geometry, combinatorics, and higher
category theory \cite{OT12,IT13,IW14,DJW18}. Previous work on this direction has
been mostly restricted to algebras of global dimension $d$ and to selfinjective
algebras (see \cite{IO11,HI11} and \cite{DI17} respectively). To our knowledge,
the higher Nakayama algebras are the first family of algebras having a
$d$-cluster-tilting module and which are neither selfinjective nor have global
dimension precisely $d$.

Our higher Nakayama algebras are equipped with distinguished $d$-cluster-tilting
modules whose additive closures are, by design, higher dimensional analogues of
the module categories of their classical counterparts (see
\th\ref{thm:An,thm:A-tilde}). The explicit combinatorial nature of Nakayama
algebras makes them particularly suitable for computations. As such, we expect
the higher Nakayama algebras to become fertile ground for testing conjectures in
higher Auslander--Reiten theory.

Along the way we also construct higher dimensional analogues of the mesh
category of type $\mathbb{ZA}_\infty$ and of the tubes. The latter categories,
ubiquitous in representation theory \cite{Rin78,Web82,Cra88,Erd95}, can be
thought of as infinite limits of module categories of Nakayama algebras. Their
higher analogues are constructed in a similar vein (see
\th\ref{def:mesh,def:tubes}). To the best of our knowledge, these categories
provide the first non-trivial examples of abelian categories having no non-zero
projective objects and no non-zero injective objects and having a
cluster-tilting subcategory (see \th\ref{thm:mesh,thm:tubes}).

Further properties of the higher Nakayama algebras and the higher analogues of
the mesh category of type $\mathbb{ZA}_\infty$ and the tubes are investigated in
separate works \cite{JKa,JKb}.

The article is structured as follows. In Section \ref{sec:preliminaries} we
collect a few preliminaries which are needed in the construction of the higher
Nakayama algebras. This includes an `idempotent reduction' lemma
(\th\ref{lemma:d-CT:idempotent_reduction}) which is the main technical tool in
our constructions. In Section \ref{sec:An} we recall Iyama's construction of the
higher Auslander algebras of type $\mathbb{A}$ and use it to introduce the
higher Nakayama algebras of type $\mathbb{A}$. In Section \ref{sec:A-infty} we
introduce the higher dimensional analogues of the mesh category of type
$\mathbb{ZA}_\infty$ as well as infinite versions of the higher Nakayama
algebras of type $\mathbb{A}$. The latter categories are needed in the
construction of the higher Nakayama algebras of type $\widetilde{\mathbb{A}}$ in
Section \ref{sec:A-tilde}. The higher dimensional analogues of the tubes are
introduced at the end of Section \ref{sec:A-tilde}. For the convenience of the
reader we include an index of symbols at the end of the article.


\section*{Conventions}

Unless noted otherwise, throughout the article $d$ denotes an arbitrary positive
integer. We fix an arbitrary field
$\mathbbm{k}$\nomenclature[00]{$\mathbbm{k}$}{the ground field} and write
$D:=\Hom_{\mathbbm{k}}(-,\mathbbm{k})$\nomenclature[01]{$D$}{the usual duality
  on the category of finite dimensional vector spaces} for the usual duality on
the category of finite dimensional vector spaces. By `algebra' we mean unital
associative $\mathbbm{k}$-algebra and by `module' we mean right module. By
`category' we mean $\mathbbm{k}$-category, that is category enriched in vector
spaces. Thus, for two objects $x$ and $y$ in a category $\mathcal{C}$ there is a
vector space of morphisms $\mathcal{C}(x,y)$ and the composition law in
$\mathcal{C}$ is bilinear. We compose morphisms in an abstract category as
functions: if $f\colon x\to y$ and $g\colon y\to z$ are morphisms in some
category, then we denote their composite by $g\circ f\colon x\to z$. Given an
object $X$ in an additive category $\mathcal{C}$, we denote by $\add X$ the
smallest additive subcategory of $\mathcal{C}$ containing $X$ and which is
closed under direct summands. Recall that the Loewy length of a finite
dimensional module $M$ is the length of its radical filtration; we denote this
number by $\len(M)$\nomenclature[03]{$\len(M)$}{the Loewy length of a module}.


\section{Preliminaries}
\label{sec:preliminaries}

In this section we recall some terminology as well as results which are needed
in the remainder of the article. Our construction of the higher Nakayama
algebras requires us to consider not only finite dimensional algebras but, more
generally, locally bounded categories. Of particular interest for us are certain
locally bounded categories constructed from partially ordered sets. Therefore we
begin with a brief review of basic aspects of the representation theory of
locally bounded categories and partially ordered sets. We also include an
overview of higher Auslander--Reiten theory in which we prove a technical
result, \th\ref{lemma:d-CT:idempotent_reduction}, which is the key ingredient in
our construction of the higher Nakayama algebras.

\subsection{Representation theory of locally bounded categories}
\label{subsec:locally_bounded_categories}

Let $\mathcal{A}$ be a small category\footnote{We remind the reader of our
  conventions regarding categories.}. By definition, a \emph{(right)
  $\mathcal{A}$-module} is a linear functor
$M\colon\mathcal{A}^\op\to\Mod\mathbbm{k}$, that is a functor such that for all
$x,y\in\mathcal{A}$ the induced function
\[
  M_{yx}\colon\mathcal{A}(x,y)\to\Hom_{\mathbbm{k}}(M_y,M_x)
\]
is a linear map. We denote the abelian category of $\mathcal{A}$-modules and
natural transformations between them by $\Mod\mathcal{A}$.
\nomenclature[04]{$\Mod\mathcal{A}$}{the category of $\mathcal{A}$-modules}By
Yoneda's lemma, for every object $x\in\mathcal{A}$ the representable functor
\[
  \mathcal{A}(-,x)\colon\mathcal{A}^\op\to\Mod\mathbbm{k}
\]
is a projective $\mathcal{A}$-module. The category of \emph{finitely generated
  projective $\mathcal{A}$-modules} is the subcategory
\[
  \proj\mathcal{A}:=\add\setP{\mathcal{A}(-,x)}{x\in\mathcal{A}}\subset\Mod\mathcal{A}.
\]\nomenclature[05]{$\proj\mathcal{A}$}{the category of finitely generated
  projective $\mathcal{A}$-modules}
By definition, $\proj\mathcal{A}$ is an idempotent complete additive category.

\begin{remark}
  Let $\mathcal{A}$ be a category with finitely many objects. There is a
  canonical isomorphism between $\Mod\mathcal{A}$ and the category of right
  modules over the unital algebra
  \[
    \bigoplus_{x,y\in\mathcal{A}}\mathcal{A}(x,y),
  \]
  whose multiplication is induced by the composition law in $\mathcal{A}$ (the
  unit is the tuple $\tupleP{1_x}{x\in\mathcal{A}}$). In view of this
  observation, throughout the article we identify categories with finitely many
  objects and algebras.
\end{remark}

A small category $\mathcal{A}$ is \emph{locally finite dimensional} if for every
pair of objects $x,y\in\mathcal{A}$ the vector space of morphisms
$\mathcal{A}(x,y)$ is finite dimensional; note that this property implies that
the Krull--Schmidt theorem holds in $\proj\mathcal{A}$, see for example
Corollary 4.4 in \cite{Kra15}. Let $\mathcal{A}$ be a locally finite dimensional
category. In order to simplify the exposition we make the additional assumptions
that all objects in $\mathcal{A}$ have local endomorphism ring and that objects
in $\mathcal{A}$ are pairwise non-isomorphic. These assumptions are equivalent
to requiring that the Yoneda embedding
$\mathcal{A}\hookrightarrow\Mod\mathcal{A}$ identifies $\mathcal{A}$ with a
complete set of representatives of the isomorphism classes of indecomposable
finitely generated projective $\mathcal{A}$-modules. Morita theory guarantees
that these assumptions do not result in a loss of generality.

Throughout the article we are mostly concerned with the following class of
$\mathcal{A}$-modules. An $\mathcal{A}$-module $M$ is \emph{finite dimensional}
if $\bigoplus_{x\in\mathcal{A}}M_x$ is a finite dimensional vector space. Thus,
$M$ is finite dimensional if for every $x\in\mathcal{A}$ the vector space $M_x$
is finite dimensional and $M_x=0$ for all but finitely many objects
$x\in\mathcal{A}$. The finite dimensional $\mathcal{A}$-modules form an abelian
subcategory $\mmod\mathcal{A}$ of
$\Mod\mathcal{A}$.\nomenclature[06]{$\mmod\mathcal{A}$}{the category of finite
  dimensional $\mathcal{A}$-modules}

A locally finite dimensional category $\mathcal{A}$ is \emph{locally bounded} if
for every object $x\in\mathcal{A}$ there are only finitely many objects
$y\in\mathcal{A}$ such that $\mathcal{A}(x,y)\neq0$ and only finitely many
objects $w\in\mathcal{A}$ such that $\mathcal{A}(w,x)\neq0$. Equivalently,
$\mathcal{A}$ is locally bounded if for each object $x\in\mathcal{A}$ the
projective $\mathcal{A}$-module $\mathcal{A}(-,x)$ and the projective
$\mathcal{A}^\op$-module $\mathcal{A}(x,-)$ are finite dimensional. This readily
implies that, if $\mathcal{A}$ is a locally bounded category, then the abelian
category $\mmod\mathcal{A}$ has enough projectives and a classical argument
shows that every finite dimensional $\mathcal{A}$-module has a projective cover
in $\mmod\mathcal{A}$, see for example Proposition 4.1 in \cite{Kra15}. As a
consequence of the duality
\begin{align*}
  D\colon\mmod\mathcal{A}&\to\mmod(\mathcal{A}^\op),\\
  M&\mapsto D\circ M
\end{align*}
the abelian category $\mmod\mathcal{A}$ also has enough injectives and every
finite dimensional $\mathcal{A}$-module has an injective envelope in
$\mmod\mathcal{A}$. We denote the category of finitely generated injective
$\mathcal{A}$-modules by\nomenclature[07]{$\inj\mathcal{A}$}{the category of
  finitely generated injective $\mathcal{A}$-modules}
\[
  \inj\mathcal{A}:=D(\proj(\mathcal{A}^\op)).
\]
Given a finite dimensional $\mathcal{A}$-module $M$, its Auslander--Bridger
transpose $\operatorname{Tr}M$ and its Auslander--Reiten translate
\[
  \tau(M):=D\operatorname{Tr}(M)
\]
\nomenclature[08]{$\tau(M)$}{the Auslander--Reiten translate of $M$}can be
defined in the usual way, see for example pages 337 and 338 in \cite{AR74}. It
is also a classical fact that the abelian category $\mmod\mathcal{A}$ has
almost-split sequences, see for example page 343 in \cite{AR74}.

Let $\mathcal{A}$ be an essentially small locally finite dimensional category
(with local endomorphism rings) and $\mathcal{X}$ a full subcategory of
$\mathcal{A}$. The \emph{idempotent ideal of $\mathcal{A}$ generated by
  $\mathcal{X}$}, denoted by $[\mathcal{X}]$, is the ideal of morphisms in
$\mathcal{A}$ which factor through an object in $\mathcal{X}$. The category
$\underline{\mathcal{A}}_{\mathcal{X}}$ has the same objects as $\mathcal{A}$
but has vector spaces of morphisms
\[
  \underline{\mathcal{A}}_{\mathcal{X}}(x,y):=\frac{\mathcal{A}(x,y)}{[\mathcal{X}](x,y)}.
\]
An object $x\in\mathcal{A}$ becomes a zero object in
$\underline{\mathcal{A}}_{\mathcal{X}}$ if and only if it belongs to
$\mathcal{X}$. It is well known that the canonical functor
$\pi\colon\mathcal{A}\to\underline{\mathcal{A}}_{\mathcal{X}}$ induces a fully
faithful exact functor
$\pi^*\colon\Mod\underline{\mathcal{A}}_{\mathcal{X}}\to\Mod\mathcal{A}$ which
restricts to a fully faithful exact functor
\[
  \pi^*\colon\mmod\underline{\mathcal{A}}_{\mathcal{X}}\to\mmod\mathcal{A},
\]
see for example \cite{Ste71} in the case of rings. Note that $\pi^*$ identifies
$\Mod\underline{\mathcal{A}}_{\mathcal{X}}$ with the full subcategory of
$\Mod\mathcal{A}$ consisting of those $\mathcal{A}$-modules $M$ such that
$M_x=0$ for all $x\in\mathcal{X}$. To alleviate the notation we leave this
identification implicit in the sequel. We also use the standard notations
\[
  \underline{\mmod}\,\mathcal{A}:=\underline{\mmod\mathcal{A}}_{\proj\mathcal{A}}
  \qquad\text{and}\qquad
  \overline{\mmod}\,\mathcal{A}:=\underline{\mmod\mathcal{A}}_{\inj\mathcal{A}}.
\]

A locally bounded category $\mathcal{A}$ is \emph{selfinjective} if
$\proj\mathcal{A}=\inj\mathcal{A}$. In this case $\mmod\mathcal{A}$ is a
Frobenius abelian category and
$\underline{\mmod}\,\mathcal{A}=\overline{\mmod}\,\mathcal{A}$ has the structure
of a triangulated category for which the suspension automorphism is given by
Heller's cosyzygy functor
$\Omega^-\colon\underline{\mmod}\,\mathcal{A}\to\underline{\mmod}\,\mathcal{A},$
see Section 2 in \cite{Hap88} for details. In this case the Auslander--Reiten
translation induces an exact equivalence
$\tau\colon\underline{\mmod}\,\mathcal{A}\to\underline{\mmod}\,\mathcal{A}$. In
particular, there is a natural isomorphism
$\tau\circ\Omega^-\cong\Omega^-\circ\tau$. Moreover, the classical
Auslander--Reiten formula implies the existence of bifunctorial isomorphisms
\[
  D\underline{\Hom}_{\mathcal{A}}(M,N)\cong\Ext_{\mathcal{A}}^1(N,\tau(M))\cong\underline{\Hom}_{\mathcal{A}}(N,(\tau\circ\Omega^-)(M))
\]
for all $M,N\in\mmod\mathcal{A}$, where the rightmost isomorphism is obtained by
the usual dimension-shifting argument.

Let $\mathcal{T}$ be a triangulated category with finite dimensional spaces of
morphisms. Recall from \cite{BK89} that an exact autoequivalence
$\mathbb{S}\colon\mathcal{T}\to\mathcal{T}$ is a \emph{Serre functor} if there
are bifunctorial isomorphisms
$D\mathcal{T}(X,Y)\cong\mathcal{T}(Y,\mathbb{S}(X))$ for all
$X,Y\in\mathcal{T}$. Note that Yoneda's embedding implies that the Serre
functor, if it exists, is unique up to canonical isomorphism. In particular, if
$\mathcal{A}$ is a selfinjective locally bounded category, then the
Auslander--Reiten formula shows that
\[
  \mathbb{S}:=(\tau\circ\Omega^-)\colon\underline{\mmod}\,\mathcal{A}\to\underline{\mmod}\,\mathcal{A}
\]
is a Serre functor.

Let $\mathcal{A}$ be a locally finite dimensional category. Although the abelian
category $\mmod\mathcal{A}$ does not necessarily have enough projectives nor
enough injectives, we can always define the extension spaces using equivalence
classes of extensions in the sense of Yoneda \cite{Yon60}. With this in mind, we
say that the abelian category $\mmod\mathcal{A}$ has \emph{global dimension
  $\delta$} if
\[
  \delta=\sup\setP{i\geq0}{\exists{M,N}\in\mmod\mathcal{A}:\Ext_{\mathcal{A}}^i(M,N)\neq0}.
\]

\subsection{Representation theory of partially ordered sets}
\label{subsec:posets}

Let $(P,\leq)$ be a poset. The \emph{incidence category of $(P,\leq)$} is the
category $\mathcal{P}$ with set of objects $P$ and vector spaces of morphisms
\[
  \mathcal{P}(x,y):=\begin{cases}
    \mathbbm{k}f_{yx}&\text{if }x\leq y,\\
    0&\text{otherwise}
  \end{cases}
\]
(by convention, $f_{yx}:=0$ if $x\not\leq y$). The composition law in
$\mathcal{P}$ is completely determined by requiring the equation
\[
  f_{zx}=f_{zy}\circ f_{yx}
\]
to be satisfied whenever $x\leq y\leq z$. When convenient, we identify
$(P,\leq)$ with its incidence category $\mathcal{P}$, which is clearly a locally
finite dimensional category.

\begin{example}
  \th\label{ex:poset:An} Let $n$ be a positive integer. A basic example of a
  poset, which is of central importance in this article, is the finite linear
  order
  \[
    A_n:=\set{0<1<\cdots<n-1}.
  \]
  When viewed as a locally finite dimensional category, the above poset is
  canonically isomorphic to the path category of the linear quiver
  \[
    \mathbb{A}_n:0\to 1\to\cdots n-1.
  \]
\end{example}

Given $x_1,x_2\in P$, define the \emph{closed interval from $x_1$ to $x_2$} to
be the subset
\[
  [x_1,x_2]:=\setP{x\in P}{x_1\leq x\leq x_2}.
\]
Quotients by idempotent ideals are easy to describe in this setting.

\begin{lemma}
  \th\label{lemma:posets:idempotent_quotient} Let $\mathcal{P}=(P,\leq)$ be a
  poset, $X$ a subset of $P$ and $\mathcal{X}$ the full subcategory of
  $\mathcal{P}$ spanned by $X$. Then,
  \[
    \underline{\mathcal{P}}_{\mathcal{X}}(x_1,x_2)\cong\begin{cases}
      \mathbbm{k}&\text{if }x_1\leq x_2\text{ and }[x_1,x_2]\cap X=\emptyset,\\
      0&\text{otherwise.}
    \end{cases}
  \]
\end{lemma}
\begin{proof}
  The morphism $f_{x_2x_1}\colon x_1\to x_2$ factors through $x\in X$ if and
  only if $x_1\leq x\leq x_2$ if and only if $x\in[x_1,x_2]$. The claim follows.
\end{proof}

Let $\mathcal{P}=(P,\leq)$ be a poset. We are mostly interested in the following
class of $\mathcal{P}$-modules.

\begin{definition}
  \th\label{def:posets:intervals} Let $\mathcal{P}=(P,\leq)$ be a poset and
  $x_1,x_2\in P$. The \emph{interval $\mathcal{P}$-module $M[x_1,x_2]$} is
  defined by associating to $x\in P$ the vector space
  \[
    M[x_1,x_2]_x:=\begin{cases}
      \mathbbm{k}&\text{if }x\in[x_1,x_2],\\
      0&\text{otherwise;}
    \end{cases}
  \]
  \nomenclature[09]{$M[x_1,x_2]$}{the interval module with top $S_{x_2}$ and
    socle $S_{x_1}$ over some poset}and to $x\leq y$ the linear map
  \[
    M[x_1,x_2]_{yx}:=\begin{cases}
      1&\text{if }x,y\in[x_1,x_2],\\
      0&\text{otherwise.}
    \end{cases}
  \]
\end{definition}

The combinatorial nature of interval modules makes them particularly suited to
explicit calculations. As a first example of this, we describe the space of
morphisms between two interval modules in combinatorial terms.

\begin{proposition}
  \th\label{prop:posets:interval_modules:interlacing} Let $\mathcal{P}=(P,\leq)$
  be a poset and $x_1,x_2,y_1,y_2\in P$. Then,
  \[
    \Hom_{\mathcal{P}}(M[x_1,x_2],M[y_1,y_2])\cong\begin{cases}
      \mathbbm{k}&\text{if } x_1\leq y_1\leq x_2\leq y_2,\\
      0&\text{otherwise.}
    \end{cases}
  \]
  Moreover, if $\eta\colon M[x_1,x_2]\to M[y_1,y_2]$ is a non-zero morphism of
  $\mathcal{P}$-modules, then its image is isomorphic to the interval module
  $M[y_1,x_2]$.
\end{proposition}
\begin{proof}
  
  Unwinding of the definitions, one can verify that a natural transformation
  $f\colon M[x_1,x_2]\to M[y_1,y_2]$ is completely determined by its component
  $f_{x_2}$. This component can be non-zero only if $y_1\leq x_2\leq y_2$ in
  which case it can be canonically identified with a non-zero scalar
  $\lambda\in\mathbbm{k}^\times$. We can also verify that the fact that $f$ is a
  natural transformation also implies that $x_1\leq y_1\leq x_2$. We leave the
  details to the reader.
\end{proof}

\begin{corollary}
  Let $\mathcal{P}=(P,\leq)$ be a poset and $x_1,x_2\in P$ such that $x_1\leq
  x_2$. Then, the $\mathcal{P}$-module $M[x_1,x_2]$ is indecomposable.
\end{corollary}
\begin{proof}
  According to \th\ref{prop:posets:interval_modules:interlacing}, the
  endomorphism algebra of $M[x_1,x_2]$ is isomorphic to the ground field whence
  it is local. The claim follows.
\end{proof}

\begin{example}
  \th\label{ex:poset:An:interlacing} Let $A_n$ be the linear order with $n$
  elements, see \th\ref{ex:poset:An}. It is well known that every indecomposable
  $A_n$-module is isomorphic to an interval module $M[i,j]$, usually depicted as
  \[
    0\to\cdots\to0\to\mathbbm{k}\xrightarrow{1}\cdots\xrightarrow{1}\mathbbm{k}\to0\to\cdots0
  \]
  where the rightmost $\mathbbm{k}$ is at position $i$ and the leftmost
  $\mathbbm{k}$ is at position $j$, see for example Section 2.6 in \cite{Bar15}.
  \th\ref{prop:posets:interval_modules:interlacing} can be seen as a
  straightforward generalisation of the corresponding statement for the poset
  $A_n$.
\end{example}

One of the main ingredients in our construction of the higher Nakayama algebras
are Iyama's higher Auslander algebras of type $\mathbb{A}$. These algebras were
introduced in \cite{Iya11}. A combinatorial approach for describing this
algebras was introduced by Oppermann and Thomas in \cite{OT12}. The notation
introduced below, a slight modification of theirs, helps us make the analogy
between the algebras we construct and the classical Nakayama algebras more
transparent.

\begin{notation}
  Let $\mathcal{P}=(P,\leq)$ be a poset and $d$ a positive integer. We endow the
  cartesian product
  \[
    P^d=\underbrace{P\times\cdots\times P}_{d\text{ times}}
  \]
  with the product order: given tuples $\bm{x}=(x_1,\dots,x_d)$ and
  $\bm{y}=(y_1,\dots,y_d)$ in $P^d$, we define $\bm{x}\preceq\bm{y}$ if and only
  if for each $i\in\set{1,\dots,d}$ the relation $x_i\leq y_i$ is satisfied. We
  denote the resulting poset by
  $\mathcal{P}^d:=(P^d,\preceq)$.\nomenclature[10]{$\mathcal{P}^d$}{the $d$-fold
    cartesian product of the poset $\mathcal{P}$ with itself}
\end{notation}

\begin{definition}
  \th\label{def:poset:d-cone} Let $\mathcal{P}=(P,\leq)$ be a poset and $d$ a
  positive integer.
  \begin{enumerate}
  \item Given $\bm{x},\bm{y}\in P^d$, we say that \emph{$\bm{x}$ interlaces
      $\bm{y}$}\nomenclature[11]{$\bm{x}\rightsquigarrow\bm{y}$}{the interlacing
      relation between ordered sequences of length $d$ $\bm{x}$ and $\bm{y}$ in
      some poset} if
    \[
      x_1\leq y_1\leq x_2\leq y_2\leq\cdots\leq x_d\leq y_d.
    \]
    We use the symbol $\bm{x}\rightsquigarrow\bm{y}$ to signify this relation.
    By definition, $\bm{x}$ is an \emph{ordered sequence of length $d$} if it
    interlaces with itself. Thus, $\bm{x}$ is an ordered sequence of length $d$
    if and only if
    \[
      x_1\leq x_2\leq\cdots\leq x_d.
    \]
    We denote the set of ordered sequences of length $d$ in $\mathcal{P}$ by
    $\mathbf{os}^d(\mathcal{P})$.\nomenclature[11]{$\mathbf{os}^d(\mathcal{P})$}{the
      poset of ordered sequences of length $d$ in the poset $\mathcal{P}$}
  \item Let $\mathcal{P}=(P,\leq)$ be a poset and $d$ a positive integer. We
    define the \emph{$d$-cone of
      $\mathcal{P}$}\nomenclature[12]{$\cone(\mathcal{P}^d)$}{the $d$-cone of
      the poset $\mathcal{P}$} to be the idempotent quotient
    \[
      \cone(\mathcal{P}^d):=\mathcal{P}^d/[\mathcal{P}^d\setminus\mathbf{os}^d(\mathcal{P})].
    \]
  \end{enumerate}
\end{definition}

The introduction of the interlacing relation on the set of ordered sequences of
length $d$ in a poset is justified by the following elementary observation.

\begin{proposition}
  \th\label{prop:poset:d-cone:interlacing} Let $\mathcal{P}=(P,\leq)$ be a poset
  and $d$ a positive integer. Then, for every pair of objects $\bm{x},\bm{y}\in
  \mathcal{P}^d$ there is an interlacing $\bm{x}\rightsquigarrow\bm{y}$ if and
  only if $\bm{x}\preceq\bm{y}$ and
  $[\bm{x},\bm{y}]\subset\mathbf{os}^d(\mathcal{P})$. In particular,
  \[
    \cone(\mathcal{P}^d)(\bm{x},\bm{y})\cong\begin{cases}
      \mathbbm{k}&\text{if }\bm{x}\rightsquigarrow\bm{y},\\
      0&\text{otherwise}.
    \end{cases}
  \]
\end{proposition}
\begin{proof}
  When $d=1$ the statement is obvious. Suppose that $d>1$. By
  \th\ref{lemma:posets:idempotent_quotient} it is enough to show that $\bm{x}$
  interlaces $\bm{y}$ if and only if $\bm{x}\preceq\bm{y}$ and
  $[\bm{x},\bm{y}]\subset\mathbf{os}^d(\mathcal{P})$. Suppose that $\bm{x}$
  interlaces $\bm{y}$, which by definition implies $\bm{x}\preceq\bm{y}$. Let
  $\bm{z}\in[\bm{x},\bm{y}]$. Then, for each $i\in\set{1,\dots,d-1}$ the
  inequalities
  \[
    x_i\leq z_i\leq y_i\leq x_{i+1}\leq z_{i+1}\leq y_{i+1}
  \]
  are satisfied (the inequality in the middle uses
  $\bm{x}\rightsquigarrow\bm{y}$ while the others follow from
  $\bm{x}\preceq\bm{z}\preceq\bm{y}$). This shows that
  $\bm{z}\in\mathbf{os}^d(\mathcal{P})$, as required.
  
  Conversely, suppose that $\bm{x}\preceq\bm{y}$ and
  $[\bm{x},\bm{y}]\subset\mathbf{os}^d(\mathcal{P})$. We need to prove that
  $\bm{x}$ interlaces $\bm{y}$. For this, let $j\in\set{1,\dots,d-1}$ and
  consider the auxiliary tuple
  \[
    \bm{z}:=(x_1,\dots,x_j,y_{j+1},\dots,y_d)
  \]
  Given that $\bm{x}\preceq\bm{y}$, the tuple $\bm{z}$ belongs to
  $[\bm{x},\bm{y}]\subset\mathbf{os}^d(\mathcal{P})$ whence it is an ordered
  sequence of length $d$. In particular, the inequality
  \[
    x_j=z_j\leq z_{j+1}=y_{j+1}
  \]
  is satisfied. This shows that $\bm{x}$ interlaces $\bm{y}$.
\end{proof}

The $d$-cone can be constructed inductively in representation-theoretic terms.
This construction extends the inductive construction of the higher Auslander
algebras of type $\mathbb{A}$ given in \cite{Iya11} from finite total orders to
arbitrary posets.

\begin{notation}
  Let $\mathcal{P}=(P,\leq)$ be a poset and $d$ a positive integer. Let
  $\bm{x}\in\mathbf{os}^{d+1}(\mathcal{P})$. Note that there is an obvious
  interlacing
  \[
    (x_1,\dots,x_d)\rightsquigarrow(x_2,\dots,x_{d+1})
  \]
  in $\mathbf{os}^d(\mathcal{P})$. Then, in view of
  \th\ref{prop:poset:d-cone:interlacing}, the $\mathcal{P}^{d}$-module
  \nomenclature[13]{$M(\bm{x})$}{the interval module
    $M[(x_1,\dots,x_d),(x_2,\dots,x_{d+1})]$ associated to the ordered sequence
    $(x_1,\dots,x_d,x_{d+1})$ in some poset}
  \[
    M(\bm{x}):=M[(x_1,\dots,x_d),(x_2,\dots,x_{d+1})]
  \]
  is in fact a module over $\cone(\mathcal{P}^d)$.
\end{notation}

We conclude our general discussion on interval modules with an elementary
observation. It should be compared with Proposition 3.12 in \cite{OT12} which
corresponds to the case of the poset $A_n$ of
\th\ref{ex:poset:An,ex:poset:An:interlacing}.

\begin{proposition}
  \th\label{prop:d-cone:interval_modules:interlacing} Let $\mathcal{P}=(P,\leq)$
  be a poset and $d$ a positive integer. Then, for every
  $\bm{x},\bm{y}\in\mathbf{os}^{d+1}(\mathcal{P})$ there is an isomorphism
  \[
    \Hom_{\cone(\mathcal{P}^d)}(M(\bm{x}),M(\bm{y}))\cong\begin{cases}
      \mathbbm{k}&\text{if }\bm{x}\rightsquigarrow\bm{y},\\
      0&\text{otherwise}.
    \end{cases}
  \]
  Moreover, the image of a non-zero morphism of $\mathcal{P}^{(d+1)}$-modules
  $M(\bm{x})\to M(\bm{y})$ is isomorphic to the interval module
  \[
    M[(y_1,\dots,y_d),(x_2,\dots,x_{d+1})].
  \]
  In particular, there is an equivalence of categories
  \[
    \mathcal{P}^{(d+1)}\cong\setP{M(\bm{x})\in\mmod\cone(\mathcal{P}^d)}{\bm{x}\in\mathbf{os}^{d+1}(\mathcal{P})}\subset\mmod\cone(\mathcal{P}^d).
  \]
\end{proposition}
\begin{proof}
  In view of \th\ref{prop:posets:interval_modules:interlacing}, it is enough to
  show that $\bm{x}$ interlaces $\bm{y}$ if and only if
  \[
    (x_1,\dots,x_d)\preceq(y_1,\dots,y_d)\preceq(x_2,\dots,x_{d+1})\preceq(y_2,\dots,y_{d+1}).
  \]
  But this is tautological.
\end{proof}

\subsection{Higher Auslander--Reiten theory}
\label{subsec:higher_AR_theory}

We conclude this section with a brief overview of Iyama's higher
Auslander--Reiten theory. We begin by recalling the classical concept of a
functorially finite subcategory from \cite{AS80}.

\begin{definition}
  Let $\mathcal{A}$ be a category and $\mathcal{X}$ a full subcategory of
  $\mathcal{A}$. We say that $\mathcal{X}$ is \emph{contravariantly finite in
    $\mathcal{A}$} if for every object $a\in\mathcal{A}$ there exists an object
  $x\in\mathcal{X}$ and a morphism $f\colon x\to a$ such that for every object
  $x'\in\mathcal{X}$ and every morphism $g\colon x'\to a$ there exists a (not
  necessarily unique) morphism $h\colon x'\to x$ such that $g=f\circ h$, that is
  such that the diagram
  \[
    \begin{tikzcd}
      &x'\dar{\forall g}\dlar[dotted,swap]{\exists h}\\
      x\rar{f}&a
    \end{tikzcd}
  \]
  commutes. Such a morphism $f$ is called a \emph{right
    $\mathcal{X}$-approximation of $\mathcal{A}$}. The notions of a covariantly
  finite subcategory of $\mathcal{A}$ and of a left approximation are defined
  dually. We say $\mathcal{X}$ is \emph{functorially finite in $\mathcal{A}$} if
  it is both contravariantly finite and covariantly finite in $\mathcal{A}$.
\end{definition}

One of the central concepts in higher Auslander--Reiten theory is the notion of
a $d$-cluster-tilting subcategory which was introduced by Iyama in \cite{Iya07}.
We recall the definition below as well as the stronger notion of a
$d\mathbb{Z}$-cluster-tilting subcategory which is implicit in \cite{Iya11} and
\cite{GKO13} in the triangulated case and was introduced in \cite{IJ17} in the
abelian case.

\begin{definition}
  \th\label{def:d-CT} Let $\mathcal{A}$ be an abelian or a triangulated category
  and $\mathcal{C}$ a subcategory of $\mathcal{A}$. We call $\mathcal{C}$ a
  \emph{$d$-cluster-tilting subcategory} if the following conditions are
  satisfied.
  \begin{enumerate}
  \item $\mathcal{C}$ is a functorially finite subcategory of $\mathcal{A}$.
  \item If $\mathcal{A}$ is abelian we require $\mathcal{C}$ to be a
    \emph{generating-cogenerating subcategory of $\mathcal{A}$}, that is for
    every object $X\in\mathcal{A}$ there exist objects $C',C''\in\mathcal{C}$,
    an epimorphism $C'\to X$, and a monomorphism $X\to C''$.
  \item\label{it:d-CT-property} There are equalities
    \begin{align*}
      \mathcal{C}&=\setP{X\in\mathcal{A}}{\forall i\in\set{1,\dots,d-1}\colon\Ext_{\mathcal{A}}^i(X,\mathcal{C})=0}\\
                 &=\setP{Y\in\mathcal{A}}{\forall i\in\set{1,\dots,d-1}\colon\Ext_{\mathcal{A}}^i(\mathcal{C},Y)=0}.
    \end{align*}
  \end{enumerate}
  We call $\mathcal{C}$ a \emph{$d\mathbb{Z}$-cluster-tilting subcategory} if
  the following additional condition is satisfied:
  \begin{enumerate}
    \setcounter{enumi}{3}
  \item\label{it:dZ-CT} If $\Ext_{\mathcal{A}}^i(\mathcal{C},\mathcal{C})\neq0$,
    then $i\in d\mathbb{Z}$.
  \end{enumerate}
  
\end{definition}

\begin{remark}
  Let $\mathcal{A}$ be an abelian or triangulated category. We make a few simple
  observations.
  \begin{enumerate}
  \item The category $\mathcal{A}$ itself is the unique $1$-cluster-tilting
    subcategory of $\mathcal{A}$.
  \item If $\mathcal{A}$ is an abelian category with enough projectives and
    enough injectives, then every subcategory of $\mathcal{A}$ satisfying
    condition \eqref{it:d-CT-property} in \th\ref{def:d-CT} is a
    generating-cogenerating subcategory.
  \item If $\mathcal{A}$ is a triangulated category with suspension functor
    $\Sigma$ and $\mathcal{C}$ is a full subcategory of $\mathcal{A}$ satisfying
    condition \eqref{it:d-CT-property} in \th\ref{def:d-CT}, then $\mathcal{C}$
    satisfies condition \eqref{it:dZ-CT} in \th\ref{def:d-CT} if and only if
    $\Sigma^d(\mathcal{C})=\mathcal{C}$.
  \end{enumerate}
\end{remark}

\begin{remark}
  We recall the `higher homological algebra' perspective on cluster-tilting
  subcategories.
  \begin{enumerate}
  \item Let $\mathcal{A}$ be an abelian category and $\mathcal{C}$ a
    $d$-cluster-tilting subcategory of $\mathcal{A}$. Then, $\mathcal{A}$ is a
    $d$-abelian category in the sense of \cite{Jas16}, see Theorem 3.16 therein.
  \item Analogously, if $\mathcal{A}$ is a triangulated category with suspension
    functor $\Sigma$ and $\mathcal{C}$ is a $d\mathbb{Z}$-cluster-tilting
    subcategory of $\mathcal{A}$, then $\mathcal{C}$ is a $(d+2)$-angulated
    category in the sense of \cite{GKO13}, see Theorem 1 therein.
  \end{enumerate}
\end{remark}

\begin{definition}
  Let $\mathcal{A}$ and $\mathcal{B}$ be abelian categories and
  $\iota\colon\mathcal{B}\to\mathcal{A}$ a fully faithful exact functor. Let
  $\mathcal{X}$ be a full subcategory of $\mathcal{B}$.
  \begin{enumerate}
  \item We say that $\iota$ is a \emph{contravariantly $\mathcal{X}$-relative
      $(d-1)$-homological embedding} if for all $B\in\mathcal{B}$ and for all
    $i\in\set{1,\dots,d-1}$ the induced morphism
    \[
      \begin{tikzcd}[column sep=small]
        \Ext_\mathcal{B}^i(-,B)|_\mathcal{X}\rar&\Ext_\mathcal{A}^i(-,B)|_\mathcal{X}
      \end{tikzcd}
    \]
    is an isomorphism.
  \item We say that $\iota$ is a \emph{covariantly $\mathcal{X}$-relative
      $(d-1)$-homological embedding} if for all $B\in\mathcal{B}$ and for all
    $i\in\set{1,\dots,d-1}$ the induced morphism
    \[
      \begin{tikzcd}[column sep=small]
        \Ext_{\mathcal{B}}^i(B,-)|_{\mathcal{X}}\rar&\Ext_{\mathcal{A}}^i(B,-)|_{\mathcal{X}}
      \end{tikzcd}
    \]
    is an isomorphism.
  \item We say that $\iota$ is an \emph{$\mathcal{X}$-relative $(d-1)$-homological
      embedding} if it is a covariantly $\mathcal{X}$-relative
    $(d-1)$-homological embedding and a contravariantly $\mathcal{X}$-relative
    $(d-1)$-homological embedding.
  \end{enumerate}
\end{definition}

The following `idempotent reduction' lemma is one of the key technical tools
used in our construction of the higher Nakayama algebras.

\begin{lemma}
  \th\label{lemma:d-CT:idempotent_reduction} Let $\mathcal{A}$ be a locally
  bounded category and $\mathcal{M}$ a $d$-cluster-tilting subcategory of
  $\mmod\mathcal{A}$. Let $\mathcal{X}$ be a full subcategory of $\mathcal{A}$
  such that the following conditions are satisfied.
  \begin{itemize}
  \item All the projective and all the injective
    $\underline{\mathcal{A}}_{\mathcal{X}}$-modules belong to $\mathcal{M}$.
  \item Every indecomposable $\mathcal{A}$-module $M\in\mathcal{M}$ which does
    not lie in $\mmod\underline{\mathcal{A}}_{\mathcal{X}}$ is
    projective-injective.
  \end{itemize}
  Define
  $\mathcal{M}_{\mathcal{X}}:=\mathcal{M}\cap\mmod\underline{\mathcal{A}}_{\mathcal{X}}$.
  Then, the following statements hold.
  \begin{enumerate}
  \item\label{it:MX-relative} The canonical inclusion
    $\mmod\underline{\mathcal{A}}_{\mathcal{X}}\to\mmod\mathcal{A}$ is an
    $\mathcal{M}_{\mathcal{X}}$-relative $(d-1)$-homological embedding.
  \item\label{it:MX-d-CT} $\mathcal{M}_{\mathcal{X}}$ is a $d$-cluster-tilting
    subcategory of $\mmod\underline{\mathcal{A}}_{\mathcal{X}}$.
  \end{enumerate}
\end{lemma}
\begin{proof}
  \eqref{it:MX-relative} We only prove that the canonical inclusion is a
  covariantly $\mathcal{M}_{\mathcal{X}}$-relative $(d-1)$-homological
  embedding, that it is also a contravariantly
  $\mathcal{M}_{\mathcal{X}}$-relative $(d-1)$-homological embedding follows by
  duality. To prove this, we apply Proposition 3.18 in \cite{Jas16} which is a
  version of the usual `dimension shifting' argument. Let $M\in
  \mathcal{M}_{\mathcal{X}}\subseteq
  \mmod\underline{\mathcal{A}}_{\mathcal{X}}$, $N\in\mmod
  \underline{\mathcal{A}}_{\mathcal{X}}$ and
  \[
    \cdots \to P^1 \to P^0\to N\to 0
  \]
  a projective resolution of $N$ as an
  $\underline{\mathcal{A}}_{\mathcal{X}}$-module. By assumption, for all
  $i\in\mathbb{Z}$ the projective $\underline{\mathcal{A}}_{\mathcal{X}}$-module
  $P^i$ belongs to $\mathcal{M}$. Given that $\mathcal{M}$ is a
  $d$-cluster-tilting subcategory of $\mmod \mathcal{A}$, for all $i\in
  \set{1,\dots,d-1}$ the extension group $\Ext^i_{\mathcal{A}}(\mathcal{M},M)$
  vanishes. Hence, by Proposition 3.18 in \cite{Jas16} for each
  $i\in\set{1,\dots,d-1}$ there is an isomorphism between
  $\Ext^i_{\mathcal{A}}(N,M)$ and the cohomology of the induced complex
  \[
    \Hom_{\mathcal{A}}(P^0,M)\to\dots\to\Hom_{\mathcal{A}}(P^d,M)
  \]
  at $\Hom_{\mathcal{A}}(P^i,M)$ which, by definition, is isomorphic to
  $\Ext^i_{\underline{\mathcal{A}}_{\mathcal{X}}}(N,M)$. This shows that the
  canonical inclusion is a covariantly $\mathcal{M}_{\mathcal{X}}$-relative
  $(d-1)$-homological embedding.
  
  \eqref{it:MX-d-CT} By assumption $\mathcal{M}_{\mathcal{X}}$ contains all
  projective and all injective $\underline{\mathcal{A}}_{\mathcal{X}}$-modules
  whence it is a generating-cogenerating subcategory of
  $\mmod\underline{\mathcal{A}}_{\mathcal{X}}$. Let us show that
  $\mathcal{M}_{\mathcal{X}}$ is a functorially finite subcategory of
  $\mmod\underline{\mathcal{A}}_{\mathcal{X}}$. For this, let $N$ be an
  $\underline{\mathcal{A}}_{\mathcal{X}}$-module and $f\colon M\to N$ a right
  $\mathcal{M}$-approximation of $N$. Write $M=Q\oplus M'$ where $M'$ is the
  largest summand of $M$ which lies in $\mathcal{M}_{\mathcal{X}}$. Note that by
  assumption $Q$ is a projective-injective $\mathcal{A}$-module. Write
  \[
    f=[f_Q\ f_{M'}]\colon Q\oplus M'\to N
  \]
  and let $f_P\colon P\to N$ a projective cover of $N$ as an
  $\underline{\mathcal{A}}_{\mathcal{X}}$-module. We claim that
  \[
    f'=[f_P\ f_{M'}]\colon P\oplus M'\to N
  \]
  is a right $\mathcal{M}_{\mathcal{X}}$-approximation of $N$. Indeed, let
  $g\colon M''\to N$ be a morphism with $M''\in\mathcal{M}_{\mathcal{X}}$. Since
  $f$ is a right $\mathcal{M}$-approximation of $N$ there exists a
  morphism
  \[
    h=[h_0\ h_1]^\top\colon M''\to Q\oplus M'
  \]
  such that $g=f\circ h$, that is
  \[
    g=f_Q\circ h_0+f_{M'}\circ h_1.
  \]
  Moreover, given that $f_P$ is an epimorphism and $Q$ is a projective(-injective)
  $\mathcal{A}$-module, there exists a morphism $q\colon Q\to P$ such that
  $f_Q=f_P\circ q$. It readily follows that $g=f'\circ h'$ where
  \[
    h'=[q\circ h_0\ h_1]^\top\colon M''\to P\oplus M.
  \]
  This shows that $\mathcal{M}_{\mathcal{X}}$ is contravariantly finite in
  $\mmod\underline{\mathcal{A}}_{\mathcal{X}}$; that $\mathcal{M}_{\mathcal{X}}$
  is covariantly finite in $\mmod\underline{\mathcal{A}}_{\mathcal{X}}$ follows
  by duality.

  Next, let $N$ be an $\underline{\mathcal{A}}_{\mathcal{X}}$-module such that
  for each $i\in\set{1,\dots,d-1}$ we have
  \[
    \Ext_{\underline{\mathcal{A}}_{\mathcal{X}}}^i(\mathcal{M}_{\mathcal{X}},N)\cong\Ext_{\mathcal{A}}^i(\mathcal{M}_{\mathcal{X}},N)=0
  \]
  where the first isomorphism follows from part \eqref{it:MX-relative}. Since by
  assumption every indecomposable $\mathcal{A}$-module $M\in\mathcal{M}$ which
  does not lie in $\mmod\underline{\mathcal{A}}_{\mathcal{X}}$ is
  projective-injective, for each $i\in\set{1,\dots,d-1}$ we conclude that
  $\Ext_{\mathcal{A}}^i(\mathcal{M},N)=0$. Since $\mathcal{M}$ is a
  $d$-cluster-tilting subcategory of $\mmod\mathcal{A}$, we deduce that
  $N\in\mathcal{M}\cap\mmod\underline{\mathcal{A}}_{\mathcal{X}}=\mathcal{M}_{\mathcal{X}}$,
  which is what we needed to show. By duality, every
  $\underline{\mathcal{A}}_{\mathcal{X}}$-module $N$ such that for each
  $i\in\set{1,\dots,d-1}$ there is an equality
  $\Ext_{\underline{\mathcal{A}}_{\mathcal{X}}}^i(N,\mathcal{M}_{\mathcal{X}})=0$
  also belongs to $\mathcal{M}_{\mathcal{X}}$. This shows that
  $\mathcal{M}_{\mathcal{X}}$ is a $d$-cluster-tilting subcategory of
  $\mmod\underline{\mathcal{A}}_{\mathcal{X}}$.
\end{proof}

\begin{remark}
  In the setting of \th\ref{lemma:d-CT:idempotent_reduction}, it can be shown
  that the canonical inclusion
  $\mmod\underline{\mathcal{A}}_{\mathcal{X}}\to\mmod\mathcal{A}$ is moreover a
  $(d-1)$-homological embedding, see Appendix
  \ref{app:relative_homological_embeddings}.
\end{remark}

For the convenience of the reader we state
\th\ref{lemma:d-CT:idempotent_reduction} in the case of finite dimensional
algebras.

\begin{lemma}
  \th\label{lemma:d-CT:idempotent_reduction:algebras} Let $A$ be a finite
  dimensional algebra and $\mathcal{M}$ a $d$-cluster-tilting subcategory of
  $\mmod A$. Let $e\in A$ be an idempotent such that the following conditions
  are satisfied.
  \begin{itemize}
  \item All the projective and all the injective $(A/AeA)$-modules belong to
    $\mathcal{M}$.
  \item Every indecomposable $A$-module $M\in\mathcal{M}$ which does not lie in
    $\mmod(A/AeA)$ is projective-injective.
  \end{itemize}
  Define $\mathcal{M}_{e}:=\mathcal{M}\cap\mmod(A/AeA)$. Then, the following
  statements hold.
  \begin{enumerate}
  \item\label{it:MX-relative-algebras} The canonical inclusion
    $\mmod(A/AeA)\to\mmod A$ is an $\mathcal{M}_{e}$-relative
    $(d-1)$-homological embedding.
  \item\label{it:MX-d-CT-algebras} $\mathcal{M}_{e}$ is a $d$-cluster-tilting
    subcategory of $\mmod(A/AeA)$.
  \end{enumerate}
\end{lemma}

Almost split sequences are the main object of study in classical
Auslander--Reiten theory. The analogous concept in higher Auslander--Reiten
theory is that of a $d$-almost split sequence, see Definition 3.1 in
\cite{Iya07}.

\begin{definition}
  \th\label{def:d-AR-seq} Let $\mathcal{M}$ be a Krull--Schmidt additive
  category. A sequence $\delta$ in $\mathcal{M}$ of the form
  \[
    \delta\colon0\to L\to M^1\to \cdots\to M^d\to N\to 0
  \]
  is \emph{$d$-almost split} if the following conditions are satisfied.
  \begin{enumerate}
  \item All morphisms in $\delta$ belong to the Jacobson radical of
    $\mathcal{M}$.
  \item For each $X\in\mathcal{M}$ the sequences of abelian groups
    \[
      \begin{tikzcd}[column sep=tiny, row sep=tiny]
        0\rar&\mathcal{M}(X,L)\rar&\mathcal{M}(X,M^1)\rar&\cdots\rar&\mathcal{M}(X,M^d)\rar&\mathcal{M}(X,N)
      \end{tikzcd}
    \]
    and
    \[
      \begin{tikzcd}[column sep=tiny, row sep=tiny]
        0\rar&\mathcal{M}(N,X)\rar&\mathcal{M}(M^d,X)\rar&\cdots\rar&\mathcal{M}(M^1,X)\rar&\mathcal{M}(L,X)
      \end{tikzcd}
    \]
    are exact.
  \item Every non-split monomorphism $L\to X$ in $\mathcal{M}$ factors through
    $L\to M^1$.
  \item Every non-split epimorphism $X\to N$ in $\mathcal{M}$ factors through
    $M^d\to N$.
  \end{enumerate}
\end{definition}

\begin{remark}
  Note that $1$-almost split sequences are nothing but classical almost split
  sequences.
\end{remark}

\begin{definition}
  Let $\mathcal{M}$ be a Krull--Schmidt additive category. We say that
  \emph{$\mathcal{M}$ has $d$-almost split sequences} if for every
  non-projective indecomposable object $N\in\mathcal{M}$ (resp. non-injective
  indecomposable object $L\in\mathcal{M}$) there exists a $d$-almost split
  sequence $0\to L\to M^1\to \cdots\to M^d\to N\to 0$.
\end{definition}

It is shown in Theorem 2.3.1 in \cite{Iya07a} that the contravariant functor
\[
  \Tr_d:=\Tr\circ\Omega^{d-1}\colon\underline{\mmod}\,\mathcal{A}\to\underline{\mmod}\,\mathcal{A}^\op
\]
induces an adjoint pair of functors
\[
  \begin{tikzcd}[column sep=small]
    \tau_d:=D\circ\Tr_d\colon\underline{\mmod}\,\mathcal{A}\rar[shift
    left]&\overline{\mmod}\,\mathcal{A}\colon\tau_d^-:=\Tr_d\circ D\lar[shift
    left]
  \end{tikzcd}
\]
called the \emph{$d$-Auslander--Reiten
  translations}.\nomenclature[14]{$\tau_d(M)$}{the $d$-Auslander--Reiten
  translate of $M$} The subsequent result is one of the main motivations for the
study of $d$-cluster-tilting subcategories of locally bounded categories. It
combines Theorems 2.5.1 and 2.5.3 in \cite{Iya07a} and adapts them to our
setting.

\begin{theorem}
  \th\label{thm:d-CT:existence_d-AR_seq} Let $\mathcal{A}$ be a locally bounded
  category and $\mathcal{M}\subseteq\mmod\mathcal{A}$ a $d$-cluster-tilting
  subcategory. Then, the following statements hold.
    \begin{enumerate}
  \item There are mutually inverse equivalences
    \[
      \begin{tikzcd}[column sep=small]
        \tau_d\colon\underline{\mathcal{M}}\rar[shift
        left]&\overline{\mathcal{M}}\colon\tau_d^-\lar[shift left]
      \end{tikzcd}
    \]
  \item\label{it:AR-formulas} For all $M,N\in\mathcal{M}$ there are bifunctorial
    isomorphisms
    \begin{align*}
      D\underline{\Hom}_{\mathcal{A}}(N,\tau_d M)&\cong\Ext_{\mathcal{A}}^d(M,N)\intertext{and}
                                            D\overline{\Hom}_{\mathcal{A}}(\tau_d^-N,M)&\cong\Ext_{\mathcal{A}}^d(M,N).
    \end{align*}
    Moreover, if $M$ has projective dimension at most $d$, then there is a
    bifunctorial isomorphism
    \[
      D\Hom_{\mathcal{A}}(N,\tau_dM)\cong\Ext_{\mathcal{A}}^d(M,N),
    \]
    whereas if $N$ has injective dimension at most $d$, then there is a
    bifunctorial isomorphism
    \[
      D\Hom_{\mathcal{A}}(\tau_d^-N,M)\cong\Ext_{\mathcal{A}}^d(M,N).
    \]
  \item The additive category $\mathcal{M}$ has $d$-almost split sequences.
    Moreover, if
    \[
      0\to L\to M^1\to \cdots\to M^d\to N\to 0
    \]
    is a $d$-almost split sequence in $\mathcal{M}$, then there are isomorphisms
    $L\cong\tau_d N$ and $N\cong \tau_d^{-}L$.
  \end{enumerate}
\end{theorem}

In parallel to the introduction of higher Auslander--Reiten theory, Iyama
introduced in \cite{Iya07a} the class of weakly $d$-representation-finite
algebras which can be thought of as higher analogues of finite dimensional
algebras of finite representation type from the viewpoint of Auslander's
correspondence (the precise terminology was introduced in Definition 2.2 of
\cite{IO11}). We recall the definition below together with some new terminology
which we use in the sequel.

\begin{definition}
  \th\label{def:d-RF} Let $A$ be a finite dimensional algebra.
  \begin{enumerate}
  \item We say that $A$ is \emph{weakly $d$-representation-finite} if there
    exists a $d$-cluster-tilting $A$-module.
  \item We say that $A$ is \emph{$d$-representation-finite $d$-hereditary} if
    $A$ has global dimension at most $d$ and $A$ is a weakly
    $d$-representation-finite algebra.
  \item We say that $A$ is \emph{$d\mathbb{Z}$-representation-finite} if there
    exists a $d\mathbb{Z}$-cluster-tilting $A$-module.
  \end{enumerate}
\end{definition}

\begin{warning}
  In \cite{IO11} the $d$-representation-finite $d$-hereditary algebras are
  simply called `$d$-representation-finite'. Our choice of terminology is
  partially motivated by Definition 3.2 in \cite{HIO14} where the notion of
  `$d$-hereditary' algebra is introduced. Note that, from a homological point of
  view, $d\mathbbm{Z}$-representation-finite algebras are higher analogues of
  (possibly non-hereditary) algebras of finite representation type.
\end{warning}

The subsequent observation gives a restriction on the global dimension of
$d\mathbb{Z}$-representation-finite algebras.

\begin{proposition}
  \th\label{prop:dZ-CT:gldimdZ} Let $A$ be a $d\mathbb{Z}$-representation-finite
  algebra of finite global dimension. Then, $\gldim A$ is a multiple of $d$. In
  particular, if $A$ is $d$-representation-finite $d$-hereditary, then either
  $A$ has global dimension $d$ or it is semisimple.
\end{proposition}
\begin{proof}
  Since $A$ has finite global dimension, the equality
  \[
    \gldim A=\max\setP{i\in\mathbb{Z}}{\Ext_A^i(DA,A)\neq0}
  \]
  holds. Moreover, given that $A$ is $d\mathbb{Z}$-representation-finite,
  $\Ext_A^i(DA,A)\neq0$ implies that $i\in d\mathbb{Z}$. Therefore $\gldim A$ is
  a multiple of $d$. The second claim is an immediate consequence of the first
  one.
\end{proof}

We recall a result which makes manifest the analogy between
$d$-representation-finite $d$-hereditary algebras and hereditary algebras of
finite representation type.

\begin{theorem}[Proposition 1.3 in \cite{Iya11}]
  \th\label{thm:d-RF_d-H:unique_d-CT} Let $A$ be a $d$-representation-finite
  $d$-hereditary algebra. Then,
  \[
    \add\setP{\tau_d^{-i}(A)}{i\geq0}=\add\setP{\tau_d^{i}(DA)}{i\geq0}\subseteq\mmod
    A
  \]
  is the unique $d$-cluster-tilting subcategory of $\mmod A$. In particular,
  each indecomposable object $X\in \mathcal{M}$ is of the form $\tau_d^{-j} P$
  for some indecomposable projective $A$-module $P$ and some integer $0\leq
  j\leq i_P$, where $i_P<\infty$ is the maximal number $i$ such that
  $\tau_d^{-i}P\neq 0$.
\end{theorem}

\begin{notation}
  Let $A$ be a $d$-representation-finite $d$-hereditary algebra. We denote the
  unique $d$-cluster-tilting subcategory of $\mmod A$ by $\mathcal{M}(A)$, see
  \th\ref{thm:d-RF_d-H:unique_d-CT}. We denote a basic additive generator of
  $\mathcal{M}(A)$ by $M(A)$.\nomenclature[15]{$\mathcal{M}(A)$}{the unique
    $d$-cluster-tilting subcategory in $\mmod A$, where $A$ is a
    $d$-re\-pre\-sen\-ta\-tion-finite $d$-hereditary algebra}
\end{notation}

Given a finite dimensional algebra $A$ of finite global dimension we denote the
\emph{derived Nakayama functor} by
\[
  \nu:=-\otimes_A^\mathbb{L} DA\colon\Db(\mmod A)\to\Db(\mmod A).
\]
It is well known that $\nu$ is a Serre functor on $\Db(\mmod A)$ and that the
autoequivalence $\nu_1:=\nu[-1]$ is a derived version of the Auslander--Reiten
translation, see Section 1.4 in \cite{Hap88} and Theorem I.2.4 in \cite{RVdB02}.

\begin{notation}
  \th\label{not:d-RF_d-H:D} Let $A$ be a finite dimensional algebra of finite
  global dimension and define \nomenclature[16]{$\nu_d$}{the $d$-th desuspension
    of the derived Nakayama functor on $\Db(\mmod A)$, where $A$ is a finite
    dimensional algebra of finite global dimension}
  \[
    \nu_d:=\nu[-d]\colon\Db(\mmod A)\to\Db(\mmod A).
  \]
  This autoequivalence is commonly thought of as a higher analogue of the
  derived Auslander--Reiten translation $\nu_1$. Following \cite{Iya11}, we
  define the subcategory \nomenclature[17]{$\mathcal{U}(A)$}{the standard
    $d\mathbb{Z}$-cluster-tilting subcategory in $\Db(\mmod A)$, where $A$ is a
    $d$-representation-finite $d$-hereditary algebra}
  \[
    \mathcal{U}(A):=\add\setP{\nu_d^i(DA)}{i\in\mathbb{Z}}\subseteq\Db(\mmod A).
  \]
\end{notation}

When $A$ is a $d$-representation-finite $d$-hereditary algebra, the category
$\mathcal{U}(A)$ plays the role of the derived category of $\add M(A)$, as
suggested by the following theorem.

\begin{theorem}[Theorem 1.2.1 in \cite{Iya11}]
  \th\label{thm:d-RF-d-H:D} Let $A$ be a $d$-representation-finite
  $d$-heredi\-tary al\-gebra. Then,
  \[
    \mathcal{U}(A)=\add\setP{M\in\mathcal{M}(A)[di]}{i\in\mathbb{Z}}
  \]
  is a $d\mathbb{Z}$-cluster-tilting subcategory of $\Db(\mmod A)$.
\end{theorem}

\begin{remark}
  Let $A$ be a representation-finite hereditary algebra. Note that
  $\mathcal{U}(A)=\Db(\mmod A)$ in this case.
\end{remark}

\begin{remark}
  Let $A$ be a representation-finite $d$-hereditary algebra. The additive
  category $\mathcal{U}(A)$ is a $(d+2)$-angulated category in the sense of
  \cite{GKO13}. Thus, form the point of view of higher-dimensional
  Auslander--Reiten theory, $\mathcal{U}(A)$ can be thought of as a
  higher-dimensional analogue of $\Db(\mmod A)$ in the same was as $\add M(A)$
  can be thought of as a higher-dimensional analogue of $\mmod A$.
\end{remark}


\section{The higher Nakayama algebras of type $\mathbb{A}$}
\label{sec:An}

Let $d$ be a positive integer. In this section we construct the $d$-Nakayama
algebras of type $\mathbb{A}$. We show that they belong to the class of
$d\mathbb{Z}$-representation-finite algebras and establish their basic
properties.

\subsection{The higher Auslander algebras of type $\mathbb{A}$}
\label{subsec:An}

We begin by recalling Iyama's original construction of the higher Auslander
algebras of type $\mathbb{A}$ from \cite{Iya11}. From our viewpoint, these
algebras are to be thought of as higher dimensional analogues of the hereditary
Nakayama algebras. Oppermann and Thomas gave a combinatorial description of
these algebras and their cluster-tilting modules in \cite{OT12}. Here we give a
complementary but closely related description in terms of ordered sequences in a
finite linear order. Although the results we need concerning these algebras have
all appeared in \cite{OT12}, we give new proofs of most of them using the
language of representations of posets.

\begin{setting}
  We fix positive integers $d$ and $n$ until further notice.
\end{setting}

Consider the linearly ordered set \nomenclature[18]{$A_n$}{the poset
  $\set{0<1<\cdots<n-1}$}
\[
  A_n:=\set{0<1<\dots<n-1}
\]
which we view as a locally finite category (or rather as a finite dimensional
algebra) as explained in Section \ref{subsec:posets}. Therefore we identify
$A_n$ with the path algebra of the quiver $0\to1\to\cdots\to n-1$, see also
\th\ref{ex:poset:An,ex:poset:An:interlacing}.

\begin{definition}
  \th\label{def:An} The \emph{$d$-Auslander algebra of type $\mathbb{A}_n$} is
  the $d$-cone $A_n^{(d)}$.\nomenclature[19]{$A_n^{(d)}$}{the higher Auslander
    algebra of type $\mathbb{A}_n$}
\end{definition}

A presentation of $A_n^{(d)}$ as a quiver with relations was given by Iyama in
Definition 6.5 and Theorem 6.7 in \cite{Iya11}. Using the explicit description
of the $d$-cone of a poset given in \th\ref{prop:poset:d-cone:interlacing} it is
elementary to verify that Iyama's presentation in fact agrees with our definition. In
order to facilitate this comparison we give an explicit description of the
$d$-cone $A_n^{(d)}$ using generators and relations. By definition, the Gabriel
quiver $Q$ of $A_n^{(d)}$ has as vertices the
set\nomenclature[20]{$\mathbf{os}_n^d$}{the poset $\mathbf{os}^d(A_n)$ of
  ordered sequences of length $d$ in $A_n$}
\[
  \mathbf{os}_n^d:=\mathbf{os}^d(A_n)
\]
of ordered sequences of length $d$ in $A_n$, which are nothing but tuples
$\bm{\lambda}=(\lambda_1,\dots,\lambda_d)$ of integers satisfying
\[
  0\leq\lambda_1\leq\cdots\leq\lambda_d\leq n-1.
\]
For each $\bm{\lambda}\in\mathbf{os}_n^d$ and each $i\in\set{1,\dots,d}$ such
that $\bm{\lambda}+e_i$ is again an ordered sequence of length $d$ there is an
arrow in $Q$ of the form
\[
  a_i=a_i(\bm{\lambda})\colon\bm{\lambda}\to\bm{\lambda}+e_i,
\]
where $e_1,\dots,e_d$ is the standard basis of $\mathbb{Z}^d$. With the above
notation, $A_n^{(d)}$ is identified with the finite dimensional algebra
$\mathbbm{k}Q/I$ where $I$ is the two-sided ideal of $\mathbbm{k}Q$ generated by
the relations
\[
  a_j(\bm{\lambda}+e_i)a_i(\bm{\lambda})-a_i(\bm{\lambda}+e_j)a_j(\bm{\lambda})
\]
for each $\bm{\lambda}\in\mathbf{os}_n^d$ and each $i,j\in\set{1,\dots,d}$ such
that $i\neq j$. By convention, $a_i(\bm{\lambda})=0$ whenever $\bm{\lambda}$ or
$\bm{\lambda}+e_i$ are not vertices of $Q$, hence some of the above relations
are in fact zero relations. For example, $A_n^{(1)}$ is just the path algebra of
the quiver $0\to 1\to\cdots\to n-1$. The Gabriel quivers of $A_4^{(2)}$ and
$A_4^{(3)}$ are shown in Figure \ref{fig:A4}.
\begin{figure}
  \begin{center}
    \includegraphics[width=0.65\textwidth]{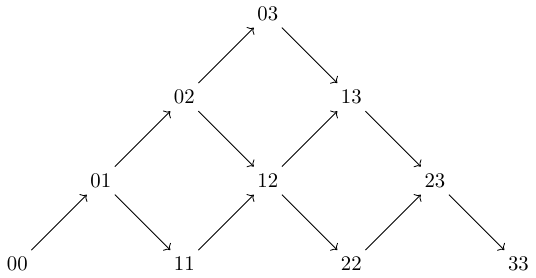}\vskip2em
    \includegraphics[width=\textwidth]{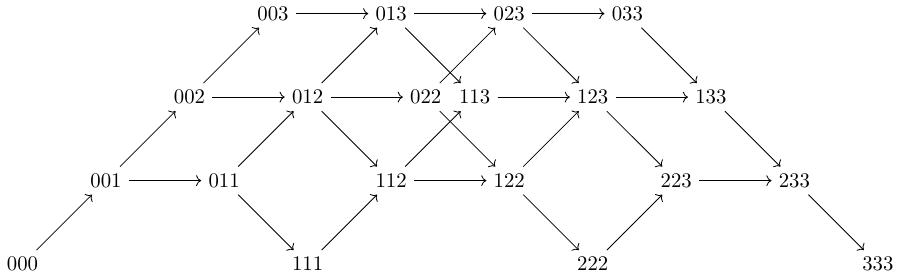}
  \end{center}
  \caption{The Gabriel quivers of $A_4^{(2)}$ (top) and $A_4^{(3)}$ (bottom).}
  \label{fig:A4}
\end{figure}

The choice of terminology in \th\ref{def:An} is justified by the following
theorem, which is a special case of Corollary 1.16 in \cite{Iya11}.

\begin{theorem}
  \th\label{thm:An} The algebra $A_n^{(d)}$ is $d$-representation-finite
  $d$-hereditary. In particular, there exists a unique basic $d$-cluster-tilting
  $A_n^{(d)}$-module $M(A_n^{(d)})$. Moreover, there is an isomorphism of
  algebras
  \[
    \End_{A_n^{(d)}}(M(A_n^{(d)}))\cong A_n^{(d+1)}.
  \]
\end{theorem}

The aim of this subsection is to give a new proof of the following statement,
which corresponds to Theorem/Construction 3.4 in \cite{OT12}. It explains our
motivation for introducing the inductive construction of the $d$-cone of a
poset, see \th\ref{prop:d-cone:interval_modules:interlacing}.

\begin{theorem}
  \th\label{thm:M_nd} There is an
  isomorphism\nomenclature[22]{$M(A_n^{(d)})$}{the unique basic
    $d$-cluster-tilting $A_n^{(d)}$-module}
  \begin{equation*}
    M(A_n^{(d)})\cong\bigoplus\setP{M(\bm{\lambda})}{\bm{\lambda}\in\mathbf{os}_n^{d+1}}
  \end{equation*}
  where $M(\bm{\lambda})$\nomenclature[23]{$M(\bm{\lambda})$}{the interval
    $A_n^{(d)}$-module
    $M[(\lambda_1,\dots,\lambda_d),(\lambda_2,\dots,\lambda_{d+1})]$ for
    $\bm{\lambda}\in\mathbf{os}_n^d$} denotes the interval module
  \[
    M[(\lambda_1,\dots,\lambda_d),(\lambda_2,\dots,\lambda_{d+1})].
  \]
  Moreover, for each $\bm{\mu}\in\mathbf{os}_n^d$ we have
  \[
    [M(\bm{\lambda}):S_{\bm{\mu}}]=\begin{cases} 1&\text{if
      }\lambda_1\leq\mu_1\leq\cdots\leq\lambda_d\leq\mu_d\leq\lambda_{d+1},\\
      0&\text{otherwise.}
    \end{cases}
  \]
\end{theorem}

In view of \th\ref{thm:An} and \th\ref{thm:d-RF_d-H:unique_d-CT}, the proof of
\th\ref{thm:M_nd} amounts to calculating the inverse higher Auslander--Reiten
translates of the indecomposable projective modules over the higher Auslander
algebras of type $\mathbb{A}$. The following result corresponds to parts (1) and
(2) of Theorem 3.6 in \cite{OT12}. It gives an explicit description of the
projective and of the injective modules over the higher Auslander algebras of
type $\mathbb{A}$.

\begin{proposition}
  \th\label{prop:An:proj_inj} Let $\bm{\lambda}\in\mathbf{os}_n^d$. Then, the
  following statements hold.
  \begin{enumerate}
  \item\label{it:A_nd-projectives} The projective $A_n^{(d)}$-module at the
    vertex $\bm{\lambda}$ is precisely
    \[
      P_{\bm{\lambda}}=M(0,\lambda_1,\dots,\lambda_d).
    \]
  \item\label{it:A_nd-injectives} The injective $A_n^{(d)}$-module at the
    vertex $\bm{\lambda}$ is precisely
    \[
      I_{\bm{\lambda}}=M(\lambda_1,\dots,\lambda_d,n-1).
    \]
  \end{enumerate}
\end{proposition}
\begin{proof}
  \eqref{it:A_nd-projectives} We prove the claim by using the explicit
  description of a basis of the projective $A_n^{(d)}$-module $P_{\bm{\lambda}}$
  in terms of paths in the Gabriel quiver, see for example Lemma III.2.4 in
  \cite{ASS06}. \th\ref{prop:poset:d-cone:interlacing} shows that, up to
  scaling, there is at most one non-zero path between any two vertices in
  $A_n^{(d)}$ and that the existence of such a non-zero path is determined by
  the interlacing relation. By the definition of the interval module
  $M(0,\lambda_1,\dots,\lambda_d)$, it is enough to show that
  $\bm{\mu}\in\mathbf{os}_n^d$ interlaces $\bm{\lambda}$ if and only if
  $\bm{\mu}$ belongs to the closed interval
  \[
    [(0,\lambda_1,\dots,\lambda_{d-1}),(\lambda_1,\dots,\lambda_d)]\subset\mathbf{os}_n^d.
  \]
  If $\bm{\mu}$ interlaces $\bm{\lambda}$, then it is clear that
  $(0,\lambda_1,\dots,\lambda_{d-1})$ interlaces $\bm{\mu}$. But this already
  implies that $\bm{\mu}$ belongs to the above closed interval. On the other
  hand, since
  \[
    (0,\lambda_1,\dots,\lambda_{d-1})\rightsquigarrow(\lambda_1,\dots,\lambda_d),
  \]
  \th\ref{prop:poset:d-cone:interlacing} implies that the closed interval is
  contained in $\mathbf{os}_n^d$, therefore every $\bm{\mu}$ in this interval
  interlaces $\bm{\lambda}$. This proves the statement. Statement
  \eqref{it:A_nd-injectives} can be shown similarly.
\end{proof}

Our next result combines Propositions 3.13 and 3.17 in \cite{OT12}. It gives an
explicit description of the higher Auslander translation in $\add M(A_n^{(d)})$.

\begin{notation}
  \th\label{not:taud} Given $\bm{\lambda}\in\mathbb{Z}^{d+1}$ we introduce the
  notation\nomenclature[24]{$\tau_d(\bm{\lambda})$}{the tuple
    $\bm{\lambda}-(1,\dots,1)$ where $\bm{\lambda}\in\mathbf{os}^d$}
  \[
    \tau_d(\bm{\lambda}):=\bm{\lambda}-(1,\dots,1)\qquad\text{and}\qquad
    \tau_d^{-}(\bm{\lambda}):=\bm{\lambda}+(1,\dots,1).
  \]
\end{notation}

\begin{proposition}
  \th\label{prop:An:taud} The following statements hold.
  \begin{enumerate}
  \item\label{it:A_nd-projective_resolutions} Let
    $\bm{\lambda}\in\mathbf{os}_n^{d+1}$ be such that $\lambda_1\neq0$ (that is
    $M(\bm{\lambda})$ is not projective) and
    \[
      0\to P^{-d}\to P^{-d+1}\to\cdots\to P^0\to M(\bm{\lambda})\to0
    \]
    a minimal projective resolution of $M(\bm{\lambda})$. Then,
    \[
      P^0\cong
      P_{\lambda_2,\dots,\lambda_{d+1}}=M(0,\lambda_2,\dots,\lambda_{d+1})
    \]
    and for each $i\in\set{1,\dots,d}$ there is an isomorphism
    \begin{align*}
      P^{-i}&\cong P_{\lambda_1-1,\dots,\lambda_i-1,\lambda_{i+2},\dots,\lambda_{d+1}}\\
            &=M(0,\lambda_1-1,\dots,\lambda_i-1,\lambda_{i+2},\dots,\lambda_{d+1}).
    \end{align*}
  \item\label{it:A_nd-injective_resolutions} Let
    $\bm{\lambda}\in\mathbf{os}_n^{d+1}$ be such that $\lambda_{d+1}\neq n-1$
    (that is $M(\bm{\lambda})$ is not injective) and
    \[
      0\to M(\bm{\lambda})\to I^0\to\cdots\to I^{d-1}\to I^d\to0
    \]
    a minimal injective coresolution of $M(\bm{\lambda})$. Then,
    \[
      I^0\cong I_{\lambda_1,\dots,\lambda_d}=M(\lambda_1,\dots,\lambda_d,n-1)
    \]
    and for each $i\in\set{1,\dots,d}$ there is an isomorphism
    \begin{align*}
      I^i&\cong I_{\lambda_1,\dots,\lambda_{d-i},\lambda_{d-i+2}+1,\dots,\lambda_{d+1}+1}\\
         &=M(\lambda_1,\dots,\lambda_{d-i},\lambda_{d-i+2}+1,\dots,\lambda_{d+1}+1,n-1).
    \end{align*}
  \item\label{it:A_nd-translate} Let $\bm{\lambda}\in\mathbf{os}_n^{d+1}$ be
    such that $\lambda_1\neq0$. Then,
    \[
      \tau_d(M(\bm{\lambda}))=M(\tau_d(\bm{\lambda})).
    \]
  \item\label{it:A_nd-inverse_translate} Let
    $\bm{\lambda}\in\mathbf{os}_n^{d+1}$ be such that $\lambda_{d+1}\neq
    n-1$. Then,
    \[
      \tau_d^-(M(\bm{\lambda}))=M(\tau_d^-(\bm{\lambda})).
    \]
  \end{enumerate}
\end{proposition}
\begin{proof}
  \eqref{it:A_nd-projective_resolutions} By
  \th\ref{prop:d-cone:interval_modules:interlacing} the non-zero morphism
  \[
    M(0,\lambda_2,\dots,\lambda_{d+1})\to
    M(\lambda_1,\dots,\lambda_d,\lambda_{d+1})
  \]
  is surjective; its kernel is the (interval) submodule of
  $M(0,\lambda_2,\dots,\lambda_{d+1})$ whose composition factors are the simple
  $A_n^{(d)}$-modules $S_{\bm{\mu}}$ such that
  \[
    (0,\lambda_2,\dots,\lambda_d)\preceq(\mu_1,\dots,\mu_d)\preceq(\lambda_2,\dots,\lambda_{d+1})
  \]
  but
  \[
    (\lambda_1,\dots,\lambda_d)\not\preceq(\mu_1,\dots,\mu_d).
  \]
  This is precisely the interval module
  \[
    M[(0,\lambda_2,\dots,\lambda_d),(\lambda_1-1,\lambda_3,\dots,\lambda_{d+1})].
  \]
  Similarly, for each $i\in\set{1,\dots,d}$ one can verify that the kernel of the
  non-zero morphism
  \[
    M(0,\lambda_1-1,\dots,\lambda_{i-1}-1,\lambda_{i+1},\dots,\lambda_{d+1})\to
    M(0,\lambda_1-1,\dots,\lambda_{i-2}-1,\lambda_{i},\dots,\lambda_{d+1})
  \]
  is precisely the interval module
  \[
    M[(0,\lambda_1-1,\dots,\lambda_{i-1}-1,\lambda_{i+1},\dots,\lambda_d),(\lambda_1-1,\dots,\lambda_i-1,\lambda_{i+2},\dots,\lambda_{d+1})]
  \]
  which is also the image of the morphism
  \[
    M(0,\lambda_1-1,\dots,\lambda_i-1,\lambda_{i+2},\dots,\lambda_{d+1})\to
    M(0,\lambda_1-1,\dots,\lambda_{i-1}-1,\lambda_{i+1},\dots,\lambda_{d+1}).
  \]
  This proves the claim. The proof of statement
  \eqref{it:A_nd-injective_resolutions} is dual.

  \eqref{it:A_nd-translate} Let
  \[
    0\to P^{-d}\to P^{-d+1}\to\cdots\to P^0\to M(\bm{\lambda})\to0
  \]
  be the minimal projective resolution of $M(\bm{\lambda})$ constructed in the
  previous step. By definition, $\tau_d(M(\bm{\lambda}))$ is isomorphic to the
  kernel of the morphism $\nu(P^{-d})\to\nu(P^{-d+1})$, that is the morphism
  \[
    M(\lambda_1-1,\dots,\lambda_d-1,n-1)\to
    M(\lambda_1-1,\dots,\lambda_{d-1}-1,\lambda_{d+1},n-1).
  \]
  By \th\ref{prop:d-cone:interval_modules:interlacing} this is precisely the
  interval module
  \[
    M(\tau_d(\bm{\lambda}))=M(\lambda_1-1,\dots,\lambda_d-1,\lambda_{d+1}-1).
  \]
  The proof of statement \eqref{it:A_nd-inverse_translate} is dual.
\end{proof}

We are ready to give our proof of \th\ref{thm:M_nd}.

\begin{proof}[Proof of \th\ref{thm:M_nd}]
  According to \th\ref{thm:d-RF_d-H:unique_d-CT} it is enough to show that
  \[
    \bigoplus\setP{\tau_d^{-i}(P_{\bm{\mu}})}{\bm{\mu}\in\mathbf{os}_n^d\text{
        and
      }i\geq0}=\bigoplus\setP{M(\bm{\lambda})}{{\bm{\lambda}\in\mathbf{os}_n^{d+1}}}.
  \]
  For this, observe that \th\ref{prop:An:taud} implies that for every
  $\bm{\lambda}\in\mathbf{os}_n^{d+1}$ there is an isomorphism
  \[
    M(\bm{\lambda})\cong\tau_d^{-\lambda_1}(P_{\lambda_2-\lambda_1,\dots,\lambda_{d+1}-\lambda_1}).
  \]
  Moreover, \th\ref{prop:An:taud} also shows that the iterated
  $\tau_d^-$-translates of the indecomposable projective $A_n^{(d)}$-modules are
  always modules of the form $M(\bm{\lambda})$ for some
  $\bm{\lambda}\in\mathbf{os}_n^{d+1}$. This proves the claim.
\end{proof}

The following result corresponds to parts (3) and (4) in Theorem 3.6
\cite{OT12}. It shows that $A_n^{(d)}$ and $\add M(A_n^{(d)})$ are, also from a
combinatorial perspective, higher dimensional analogues of $A_n$ and of its
module category.

\begin{proposition}
  \th\label{prop:An:Hom_Ext} Let $\bm{\lambda},\bm{\mu}\in\mathbf{os}_n^{d+1}$.
  Then, there are isomorphisms
  \begin{align*}
    \Hom_{A_n^{(d)}}(M(\bm{\lambda}),M(\bm{\mu}))&\cong\begin{cases}
      \mathbbm{k}&\text{if }\bm{\lambda}\rightsquigarrow\bm{\mu},\\
      0&\text{otherwise;}
    \end{cases}\intertext{and}
         \Ext_{A_n^{(d)}}^d(M(\bm{\lambda}),M(\bm{\mu}))&\cong\begin{cases}
           \mathbbm{k}&\text{if }\bm{\mu}\rightsquigarrow\tau_d(\bm{\lambda}),\\
           0&\text{otherwise.}
         \end{cases}
  \end{align*}
\end{proposition}
\begin{proof}
  The first isomorphism is proven in
  \th\ref{prop:d-cone:interval_modules:interlacing}. The second isomorphism
  follows from the higher dimensional Auslander--Reiten formulas given in
  statement \eqref{it:AR-formulas} in \th\ref{thm:d-CT:existence_d-AR_seq},
  taking into account the formulas for the higher Auslander--Reiten translate
  proven in statements \eqref{it:A_nd-translate} and
  \eqref{it:A_nd-inverse_translate} in \th\ref{prop:An:taud} and the fact that
  $M(\bm{\lambda})$ has projective dimension $d$, see
  \th\ref{prop:An:taud}\eqref{it:A_nd-projective_resolutions}.
\end{proof}

We conclude our study of the higher Auslander algebras of type $\mathbb{A}$ with
a simple observation.

\begin{lemma}
  \th\label{lemma:An:len} Let $\bm{\lambda}\in\mathbf{os}_n^{d+1}$. Then, the
  interval module $M(\bm{\lambda})$ has Loewy length
  $\lambda_{d+1}-\lambda_1+1$.
\end{lemma}
\begin{proof}
  Recall that the interval module $M(\bm{\lambda})$ has simple top
  $S_{\lambda_2,\dots,\lambda_{d+1}}$ and simple socle
  $S_{\lambda_1,\dots,\lambda_d}$. The Loewy length of $M(\bm{\lambda})$ is then
  1 plus the length of the longest path from $(\lambda_1,\dots,\lambda_d)$ to
  $(\lambda_2,\dots,\lambda_{d+1})$ in $\mathbf{os}_n^d$ which can easily seen
  to have length
  \[
    (\lambda_{d+1}-\lambda_d)+(\lambda_d-\lambda_{d-1})+\cdots+(\lambda_2-\lambda_1)=\lambda_{d+1}-\lambda_1.\qedhere
  \]
\end{proof}

Motivated by \th\ref{lemma:An:len}, we make the following definition.

\begin{definition}
  \th\label{def:An:len} Let $\bm{\lambda}\in\mathbf{os}_n^{d+1}$. We define the
  \emph{Loewy length of $\bm{\lambda}$} to
  be\nomenclature[25]{$\len(\bm{\lambda})$}{the Loewy length
    $\lambda_{d+1}-\lambda_1+1$ of $\bm{\lambda}\in\mathbf{os}^{d+1}$}
  \[
    \len(\bm{\lambda}):=\len(M(\bm{\lambda}))=\lambda_{d+1}-\lambda_1+1\in\set{1,\dots,n}.
  \]
\end{definition}

\subsection{The higher Nakayama algebras of type $\mathbb{A}$}

We are ready to introduce the higher Nakayama algebras of type $\mathbb{A}$. In
the classical case, the Nakayama algebras of type $\mathbb{A}$ are
\emph{admissible quotients} of path algebras of the linearly oriented quivers of
type $\mathbb{A}$. In contrast, the higher Nakayama algebras of type
$\mathbb{A}$ are \emph{idempotent quotients} of the higher Auslander algebras of
type $\mathbb{A}$.

\begin{setting}
  We fix positive integers $d$ and $n$ until further notice.
\end{setting}

We begin by recalling the notion of a Kupisch series of type $\mathbb{A}$,
originally introduced in \cite{Kup59}.

\begin{definition}
  \th\label{def:KS:An} Let
  $\underline{\ell}=(\ell_0,\ell_1,\dots,\ell_{n-1})$\nomenclature[26]{$\underline{\ell}$}{a
    Kupisch series} be a tuple of positive integers. We say that
  $\underline{\ell}$ is a \emph{(connected) Kupisch series of type
    $\mathbb{A}_n$} if $\ell_0=1$ and for all $i\neq 0$ there are inequalities
  \[
    2\leq\ell_{i}\leq \ell_{i-1}+1.
  \]
  We denote the set of Kupisch series of type $\mathbb{A}_n$ by $\KSA{n}$.
\end{definition}

\begin{remark}
  Suppose that the ground field is algebraically closed. In this case, it is
  well known that there is a bijective correspondence between Morita equivalence
  classes of (connected) Nakayama algebras of type $\mathbb{A}$ and Kupisch
  series of type $\mathbb{A}$. In one direction the correspondence is given by
  associating the tuple
  \[
    \tuple{\len(e_0A),\len(e_1A),\dots,\len(e_{n-1}A)}
  \]
  to a Nakayama algebra $A$ with Gabriel quiver $0\to1\to\cdots\to n-1$.
\end{remark}

\begin{notation}
  Let $\underline{\ell}$ be a Kupisch series of type $\mathbb{A}_n$. We denote
  the basic Nakayama algebra with Kupisch series $\underline{\ell}$ by
  $A_{\underline{\ell}}^{(1)}$. Note that, by definition, $A_{\underline{\ell}}^{(1)}$ is an
  admissible quotient of $A_n$.
\end{notation}

We need some elementary observations concerning Kupisch series which follow by
an easy induction from the inequalities $\ell_i\leq\ell_{i-1}+1$.

\begin{lemma}
  \th\label{lemma:KS} Let $\underline{\ell}$ be a Kupisch series of type
  $\mathbb{A}_n$. Then, the following statements hold.
  \begin{enumerate}
  \item\label{lemma:KS:An:it:2} For all $i\leq j$ the inequality
    $i-j\leq\ell_i-\ell_j$ is satisfied.
  \item\label{lemma:KS:An:it:1} For all $i\in\set{0,1,\dots,n-1}$ the inequality
    $\ell_i\leq i+1$ is satisfied.
  \end{enumerate}
\end{lemma}

The product order on $\mathbb{Z}^n$ endows the set $\KSA{n}$ of Kupisch series
of type $\mathbb{A}_n$ with the structure of a poset. We need the following
simple observations, the proofs of which are also left to the reader.

\begin{lemma}
  \th\label{lemma:KS:An:poset} The following statements hold.
  \begin{enumerate}
  \item The Kupisch series $(1,2,\dots,n)$ is the unique maximal element in
    $\KSA{n}$.
  \item The Kupisch series $(1,2,\dots,2)$ is the unique minimal element in
    $\KSA{n}$.
  \item The Hasse quiver of $\KSA{n}$ is connected.
  \item Two Kupisch series $\underline{\ell}\leq\underline{\ell}'$ are
    neighbours in the Hasse quiver of $\KSA{n}$ then they are also neighbours in
    $\mathbb{Z}^n$.
  \end{enumerate}
\end{lemma}

Let $\underline{\ell}$ be a Kupisch series of type $\mathbb{A}_n$ and
$\bm{\lambda}=(\lambda_1,\lambda_2)\in\mathbf{os}_n^2$. It is well known and
elementary to verify that the $A_n^{(1)}$-module $M(\bm{\lambda})$ is an
$A_{\underline{\ell}}^{(1)}$-module if and only if
\[
  \len(\bm{\lambda})=\len(M(\bm{\lambda}))\leq
  \len(P_{\lambda_2})=\ell_{\lambda_2}
\]
where $P_{\lambda_2}$ is the indecomposable projective
$A_{\underline{\ell}}^{(1)}$-module with top $S_{\lambda_2}$. This observation
motivates our definition of the higher Nakayama algebras of type $\mathbb{A}$.

\begin{definition}
  \th\label{def:An:l} Let $\underline{\ell}$ be a Kupisch series of type
  $\mathbb{A}_n$.
  \begin{enumerate}
  \item\label{def:An:l:it:l-restriction} The
    \emph{$\underline{\ell}$-restriction of
      $\mathbf{os}_n^{d+1}$}\nomenclature[27]{$\mathbf{os}_{\underline{\ell}}^d$}{the
      $\underline{\ell}$-restriction of $\mathbf{os}_n^d$} is the subset
    \begin{equation*}
      \mathbf{os}_{\underline{\ell}}^{d+1}:=\setP{\bm{\lambda}\in\mathbf{os}_n^{d+1}}{\len(\bm{\lambda})\leq\ell_{\lambda_{d+1}}}.
    \end{equation*}
  \item The \emph{$(d+1)$-Nakayama algebra with Kupisch series
      $\underline{\ell}$}\nomenclature[28]{$A_{\underline{\ell}}^{(d)}$}{the
      $d$-Nakayama algebra with Kupisch series $\underline{\ell}$ of type
      $\mathbb{A}_n$} is the finite dimensional algebra
    \[
      A_{\underline{\ell}}^{(d+1)}:=A_n^{(d+1)}/[\mathbf{os}_n^{d+1}\setminus\mathbf{os}_{\underline{\ell}}^{d+1}].
    \]
  \item The $A_n^{(d)}$-module $M_{\underline{\ell}}^{(d)}$ is by
    definition\nomenclature[29]{$M_{\underline{\ell}}^{(d)}$}{the distinguished
      $d\mathbb{Z}$-cluster-tilting $A_{\underline{\ell}}^{(d)}$-module}
    \begin{equation*}
      M_{\underline{\ell}}^{(d)}:=\bigoplus\setP{M(\bm{\lambda})}{\bm{\lambda}\in\mathbf{os}_{\underline{\ell}}^{d+1}}.
    \end{equation*}
  \end{enumerate}
\end{definition}

\begin{figure}
  \begin{center}
    \includegraphics[width=0.65\textwidth]{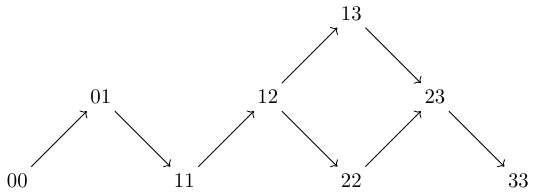}\vskip2em
    \includegraphics[width=\textwidth]{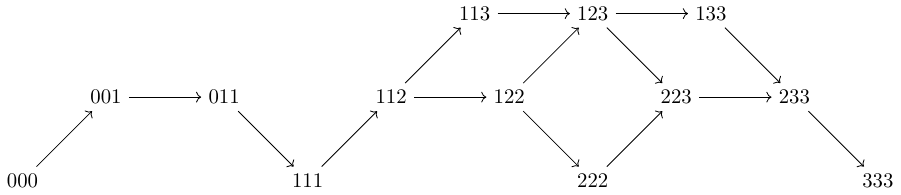}
  \end{center}
  \caption{The Gabriel quivers of $A_{\underline{\ell}}^{(2)}$ (top) and
    $A_{\underline{\ell}}^{(3)}$ (bottom) for $\underline{\ell}=(1,2,2,3)$.}
  \label{fig:An:l:1223}
\end{figure}

Let $\underline{\ell}$ be a Kupisch series of type $\mathbb{A}_n$. Note that the
condition for $\bm{\lambda}\in\mathbf{os}_n^{d+1}$ to belong to
$\mathbf{os}_{\underline{\ell}}^{d+1}$ only makes use of the pair
$(\lambda_1,\lambda_{d+1})$.

The purpose of this section is to prove the following theorem, which establishes
the higher Auslander--Reiten theoretical nature of the higher Nakayama algebras
of type $\mathbb{A}$.

\begin{theorem}
  \th\label{thm:An:l} Let $\underline{\ell}$ be a Kupisch series of type
  $\mathbb{A}_n$.
  \begin{enumerate}
  \item\label{thm:An:l:it:KS} For each $i\in\set{0,1,\dots,n-1}$ the
    indecomposable projective $A_{\underline{\ell}}^{(d)}$-module at the vertex
    $(i,\dots,i)$ has length $\ell_i$.
  \item\label{thm:An:l:it:dZ-RF} $M_{\underline{\ell}}^{(d)}$ is a
    $d\mathbb{Z}$-cluster-tilting $A_{\underline{\ell}}^{(d)}$-module. In
    particular, $A_{\underline{\ell}}^{(d)}$ is a
    $d\mathbb{Z}$-representation-finite algebra.
  \item\label{thm:An:l:it:taud} For every simple
    $A_{\underline{\ell}}^{(d)}$-module $S$ which is a direct summand of
    $M_{\underline{\ell}}^{(d)}$ the $A_{\underline{\ell}}^{(d)}$-module
    $\tau_d^i(S)$ is either simple or zero for all $i\in\mathbb{Z}$.
  \end{enumerate}
\end{theorem}

\begin{remark}
  Statement \eqref{thm:An:l:it:KS} in \th\ref{thm:An:l} can be interpreted as
  saying that $A_{\underline{\ell}}^{(d)}$ `has Kupisch series
  $\underline{\ell}$'. Statement \eqref{thm:An:l:it:dZ-RF} in \th\ref{thm:An:l}
  shows that $M_{\underline{\ell}}^{(d)}$ is a higher dimensional analogue of
  $M_{\underline{\ell}}^{(1)}$, the standard additive generator of $\mmod
  A_{\underline{\ell}}^{(1)}$. Finally, statement \eqref{thm:An:l:it:taud} in
  \th\ref{thm:An:l} should be compared with Theorem IV.2.10 in \cite{ARS97}
  which states that an Artin algebra $A$ is a Nakayama algebra if and only if
  the $\tau$-orbit of each simple $A$-module consists entirely of simple
  $A$-modules.
\end{remark}

We divide the proof of \th\ref{thm:An:l} into several parts. Since the case
$d=1$ is classical we restrict our attention to the case $d\geq2$.

\begin{setting}
  We fix integers $d\geq2$ and $n\geq1$ and a Kupisch series $\underline{\ell}$
  of type $\mathbb{A}_n$ until the end of this section.
\end{setting}

\begin{proposition}
  \th\label{prop:An:l:M} The $A_n^{(d)}$-module $M_{\underline{\ell}}^{(d)}$ is
  in fact an $A_{\underline{\ell}}^{(d)}$-module.
\end{proposition}
\begin{proof}
  Let $\bm{\lambda}\in\mathbf{os}_{\underline{\ell}}^{d+1}$ and
  $\bm{\mu}\in\mathbf{os}_n^d$ be such that
  $[M(\bm{\lambda}):S_{\bm{\mu}}]\neq0$; we need to prove that
  $\bm{\mu}\in\mathbf{os}_{\underline{\ell}}^d$. Let
  \[
    x:=\lambda_{d+1}+1-\ell_{\lambda_{d+1}}.
  \]
  Note that the inequality $\len(\bm{\lambda})\leq\ell_{\lambda_{d+1}}$, which
  holds by assumption, is equivalent to
  \[
    \lambda_1-x=\lambda_1-(\lambda_{d+1}+1-\ell_{\lambda_{d+1}})\geq0.
  \]
  Therefore $x\leq\lambda_1\leq\mu_1\leq\mu_d\leq\lambda_{d+1}$. We claim that
  there are inequalities
  \[
    \len(\bm{\mu})=\mu_d-\mu_1+1\leq\mu_d-x+1\leq\ell_{\mu_d}.
  \]
  The inequality on the left is clear. The inequality on the right is equivalent
  to
  \[
    \mu_d-\lambda_{d+1}\leq\ell_{\mu_d}-\ell_{\lambda_{d+1}}
  \]
  and therefore follows from \th\ref{lemma:KS}\eqref{lemma:KS:An:it:2}. This
  shows that $\bm{\mu}\in\mathbf{os}_{\underline{\ell}}^d$.
\end{proof}

The key ingredient in the proof of \th\ref{thm:An:l} is
\th\ref{lemma:d-CT:idempotent_reduction:algebras}. To use it we need detailed
knowledge of the projective and of the injective modules over the higher
Nakayama algebras of type $\mathbb{A}$.

\begin{proposition}
  \th\label{prop:An:l:proj_inj} Let
  $\bm{\lambda}\in\mathbf{os}_{\underline{\ell}}^d$. Then, the following
  statements hold.
  \begin{enumerate}
  \item\label{prop:An:l:proj_inj:it:P} The indecomposable projective
    $A_{\underline{\ell}}^{(d)}$-module with top $S_{\bm{\lambda}}$ is precisely
    \[
      P_{\bm{\lambda}}=M(x,\lambda_1,\dots,\lambda_d)\in\add
      M_{\underline{\ell}}^{(d)}
    \]
    where $x=\lambda_d+1-\ell_{\lambda_d}$. Moreover,
    \[
      x=\min\setP{0\leq{x'}\leq\lambda_1}{(x',\lambda_1,\dots,\lambda_d)\in\mathbf{os}_{\underline{\ell}}^{d+1}}.
    \]
  \item\label{prop:An:l:proj_inj:it:I} Let
    $\bm{\lambda}\in\mathbf{os}_{\underline{\ell}}^d$. The indecomposable
    injective $A_{\underline{\ell}}^{(d)}$-module with socle $S_{\bm{\lambda}}$
    is precisely
    \[
      I_{\bm{\lambda}}=M(\lambda_1,\dots,\lambda_d,y)\in\add
      M_{\underline{\ell}}^{(d)},
    \]
    where
    \[
      y:=\max\setP{\lambda_d\leq y'\leq
        n-1}{(\lambda_1,\dots,\lambda_d,y')\in\mathbf{os}_{\underline{\ell}}^{d+1}}.
    \]
  \end{enumerate}
\end{proposition}
\begin{proof}
  \eqref{prop:An:l:proj_inj:it:P} Note that statement \eqref{lemma:KS:An:it:1}
  in \th\ref{lemma:KS} implies that $x\geq0$. Moreover, the inequality
  \[
    \lambda_1-x=\lambda_1-\lambda_d-1+\ell_{\lambda_d}=-\len(\bm{\lambda})+\ell_{\lambda_d}\geq0
  \]
  holds by assumption. It follows that
  $(x,\lambda_1,\dots,\lambda_d)\in\mathbf{os}_{\underline{\ell}}^{d+1}$ and
  therefore
  \[
    M(x,\lambda_1,\dots,\lambda_d)\in\add M_{\underline{\ell}}^{(d)}.
  \]
  Suppose now that $\bm{\mu}\in\mathbf{os}_{\underline{\ell}}^d$ is such that
  there exists a non-zero path $\bm{\mu}\rightsquigarrow\bm{\lambda}$. By
  \th\ref{prop:poset:d-cone:interlacing} and the definition of
  $A_{\underline{\ell}}^{(d)}$ this is equivalent to the inequalities
  \[
    \mu_1\leq\lambda_1\leq\dots\leq\mu_d\leq\lambda_d
  \]
  being satisfied together with
  $[\bm{\mu},\bm{\lambda}]\subset\mathbf{os}_{\underline{\ell}}^d$. We need to
  prove that $x\leq \mu_1$ or, equivalently
  \[
    \len(\mu_1,\lambda_2,\dots,\lambda_d)=\lambda_d-\mu_1+1\leq\ell_{\lambda_d}.
  \]
  But this follows from the fact that
  $(\mu_1,\lambda_2,\dots,\lambda_d)\in[\bm{\mu},\bm{\lambda}]\subset\mathbf{os}_{\underline{\ell}}^d$.
  This shows that $P_{\bm{\lambda}}=M(x,\lambda_1,\dots,\lambda_d)$.

  \eqref{prop:An:l:proj_inj:it:I} By construction
  $(\lambda_1,\dots,\lambda_d,y)\in\mathbf{os}_{\underline{\ell}}^{d+1}$ whence
  \[
    M(\lambda_1,\dots,\lambda_d,y)\in\add M_{\underline{\ell}}^{(d)}.
  \]
  Suppose now that $\bm{\mu}\in\mathbf{os}_{\underline{\ell}}^d$ is such that
  there exists a non-zero path $\bm{\lambda}\rightsquigarrow\bm{\mu}$. By
  \th\ref{prop:poset:d-cone:interlacing} and the definition of
  $A_{\underline{\ell}}^{(d)}$ this is equivalent to the inequalities
  \[
    \lambda_1\leq\mu_1\leq\cdots\leq\lambda_d\leq\mu_d
  \]
  being satisfied together with
  $[\bm{\lambda},\bm{\mu}]\subset\mathbf{os}_{\underline{\ell}}^d$. We need to
  prove that $\mu_d\leq y$. For this, note that
  $(\lambda_1,\dots,\lambda_{d-1},\mu_d)\in[\bm{\lambda},\bm{\mu}]\subset\mathbf{os}_{\underline{\ell}}^d$
  which immediately implies that
  $(\lambda_1,\mu_d)\in\mathbf{os}_{\underline{\ell}}^2$. The claim now follows
  from the maximality of $y$.
\end{proof}

The following statement settles part \eqref{thm:An:l:it:KS} in
\th\ref{thm:An:l}.

\begin{corollary}
  Let $\bm{\lambda}\in\mathbf{os}_{\underline{\ell}}^{d}$. Then,
  $\len(P_{\bm{\lambda}})=\ell_{\lambda_d}$. In particular, for every
  $i\in\set{0,1,\dots,n-1}$ the projective cover of the simple
  $A_{\underline{\ell}}^{(d)}$-module at the vertex $(i,\dots,i)$ has Loewy
  length $\ell_i$.
\end{corollary}
\begin{proof}
  By \th\ref{prop:An:l:proj_inj}\eqref{prop:An:l:proj_inj:it:P}, the projective
  cover of the simple $A_{\underline{\ell}}^{(d)}$-module at the vertex
  $\bm{\lambda}$ is
  \[
    P_{\bm{\lambda}}=M(\lambda_d+1-\ell_{\lambda_d},\lambda_1,\dots,\lambda_d)
  \]
  whence
  \[
    \len(P_{\bm{\lambda}})=\lambda_d-(\lambda_d+1-\ell_{\lambda_d})+1=\ell_{\lambda_d}
  \]
  as required.
\end{proof}

We continue the proof of \th\ref{thm:An:l} by showing that the higher Nakayama
algebras are weakly $d$-representation-finite.

\begin{proposition}
  \th\label{prop:An:l:M:d-CT} The $A_{\underline{\ell}}^{(d)}$-module
  $M_{\underline{\ell}}^{(d)}$ is $d$-cluster-tilting.
\end{proposition}
\begin{proof}
  We use the partial order on the set of Kupisch series of type $\mathbb{A}_n$.
  Suppose that the claim holds for some Kupisch series
  $\underline{\ell}'\geq\underline{\ell}$ which is a neighbour of
  $\underline{\ell}$ in the Hasse quiver of the poset of Kupisch series of type
  $\mathbb{A}_n$. We claim that the statement holds for $\underline{\ell}$ in
  this case. To prove this, it is enough to verify the hypotheses of
  \th\ref{lemma:d-CT:idempotent_reduction:algebras}.

  \emph{Step 1:
    $\mathbf{os}_{\underline{\ell}}^{d+1}\subset\mathbf{os}_{\underline{\ell}'}^{d+1}$.}
  Let $\bm{\lambda}\in\mathbf{os}_{\underline{\ell}}^{d+1}$. Then
  \[
    \lambda_{d+1}-\lambda_1+1\leq\ell_{\lambda_{d+1}}\leq\ell_{\lambda_{d+1}}'
  \]
  since $\underline{\ell}\leq\underline{\ell}'$. Note that this fact, together
  with \th\ref{prop:An:l:proj_inj}, implies that the projective and the
  injective $A_{\underline{\ell}}^{(d)}$-modules are direct summands of
  $M_{\underline{\ell}'}^{(d)}$.
  
  \emph{Step 2: If
    $\bm{\lambda}\in\mathbf{os}_{\underline{\ell}'}^{d+1}\setminus\mathbf{os}_{\underline{\ell}}^{d+1}$,
    then $M(\bm{\lambda})$ is projective-injective as an
    $A_{\underline{\ell}'}^{(d)}$-module.} By assumption,
  \[
    \ell_{\lambda_{d+1}}<\len({\bm{\lambda}})\leq\ell_{\lambda_{d+1}}'.
  \]
  Since $\underline{\ell}$ and $\underline{\ell}'$ are neighbours, we conclude
  that $\ell_{\lambda_{d+1}}+1=\ell_{\lambda_{d+1}}'$ and, since
  $\len(\bm{\lambda})$ is an integer, that
  $\len(\bm{\lambda})=\ell_{\lambda_{d+1}}'$. The last equality is equivalent to
  \[
    \lambda_1=\lambda_{d+1}+1-\ell_{\lambda_{d+1}}'=\lambda_{d+1}+1-(\ell_{\lambda_{d+1}}+1)=\lambda_{d+1}-\ell_{\lambda_{d+1}}.
  \]
  In view of the leftmost equality, the module $M(\bm{\lambda})$ is a
  projective $A_{\underline{\ell}'}^{(d)}$-module by
  \th\ref{prop:An:l:proj_inj}.

  Now we prove that $M(\bm{\lambda})$ is an injective
  $A_{\underline{\ell}'}^{(d)}$-module. By
  \th\ref{prop:An:l:proj_inj}\eqref{prop:An:l:proj_inj:it:I} it is enough to
  prove that
  \[
    \lambda_{d+1}=y:=\max\setP{\lambda_d\leq y'\leq
      n-1}{(\lambda_1,y')\in\mathbf{os}_{\underline{\ell}'}^{2}}.
  \]
  Suppose that $\lambda_{d+1}+1\leq y$. In view of \th\ref{thm:M_nd} and the
  inequalities
  \[
    \lambda_1=\lambda_1\leq\lambda_2=\lambda_2\leq\cdots\leq\lambda_{d-1}=\lambda_{d-1}\leq\lambda_d\leq\lambda_{d+1}+1\leq
    y
  \]
  the simple $A_n^{(d)}$-module at the vertex
  $(\lambda_1,\dots,\lambda_{d-1},\lambda_{d+1}+1)$ is a composition factor of
  the $A_{\underline{\ell}'}^{(d)}$-module $M(\lambda_1,\dots,\lambda_d,y)$. In
  particular
  $(\lambda_1,\dots,\lambda_{d-1},\lambda_{d+1}+1)\in\mathbf{os}_{\underline{\ell}'}^d$.
  We claim that this is impossible. To show this, we observe that the fact that
  $\underline{\ell}$ is a Kupisch series implies
  \[
    \ell_{\lambda_{d+1}+1}'=\ell_{\lambda_{d+1}+1}\leq\ell_{\lambda_{d+1}}+1=\ell_{\lambda_{d+1}}'=\len(\bm{\lambda}).
  \]
  But this inequality already implies
  $(\lambda_1,\dots,\lambda_{d-1},\lambda_{d+1}+1)\notin\mathbf{os}_{\underline{\ell}'}^d$.
  Indeed,
  \[
    \ell_{\lambda_{d+1}+1}'\leq\len(\bm{\lambda})<\len(\bm{\lambda})+1=\lambda_{d+1}+2-\lambda_1.
  \]
  This proves our claim. We conclude that $\lambda_{d+1}=y$, as required.

  Since we have verified the hypotheses in
  \th\ref{lemma:d-CT:idempotent_reduction:algebras}, we have proven that
  $M_{\underline{\ell}}^{(d)}$ is a $d$-cluster-tilting
  $A_{\underline{\ell}}^{(d)}$-module in this case. The general case now follows
  since the claim holds for the Kupisch series $(1,2,\dots,n)$ by
  \th\ref{thm:An,thm:M_nd} and the Hasse quiver of the (finite) poset of Kupisch series
  of type $\mathbb{A}_n$ is connected, since it has $(1,2,\dots,n)$ as a maximum element.
\end{proof}

The following result implies that the higher Nakayama algebras of type
$\mathbb{A}$ are moreover $d\mathbb{Z}$-representation-finite. It settles
\th\ref{thm:An:l}\eqref{thm:An:l:it:dZ-RF}.

\begin{proposition}
  \th\label{prop:An:l:M:dZ-CT} Suppose that $d\geq2$. Then, if $M(\bm{\lambda})$ is not projective as an $A_{\underline{\ell}}^{(d)}$-module, then $\Omega^d(M(\bm{\lambda}))\cong M(\lambda_{d+1}+1-\ell_{\lambda_{d+1}}, \lambda_1-1,\dots,\lambda_d-1)$.  
In particular, $\add
  M_{\underline{\ell}}^{(d)}$ is closed under taking $d$-th syzygies. In
  particular, $M_{\underline{\ell}}^{(d)}$ is a $d\mathbb{Z}$-cluster-tilting
  $A_{\underline{\ell}}^{(d)}$-module.
\end{proposition}
\begin{proof}
  Let $\bm{\lambda}\in\mathbf{os}_{\underline{\ell}}^{d+1}$ be such that
  $M(\bm{\lambda})$ is not projective as an $A_{\underline{\ell}}^{(d)}$-module;
  by \th\ref{prop:An:l:proj_inj}\eqref{prop:An:l:proj_inj:it:P} this means that
  $x<\lambda_1$ where $x=\lambda_{d+1}+1-\ell_{\lambda_{d+1}}$. Let
  \[
    0\to\Omega^d(M(\bm{\lambda}))\to P^{-d+1}\to\cdots\to P^0\to
    M(\bm{\lambda})\to0
  \]
  be part of a minimal projective resolution of $M(\bm{\lambda})$. A
  straightforward calculation analogous to that in the proof of
  \th\ref{prop:An:taud}\eqref{it:A_nd-projective_resolutions} shows that
  \[
    P^0\cong
    P_{\lambda_2,\dots,\lambda_{d+1}}=M(x,\lambda_2,\dots,\lambda_{d+1}),
  \]
  and that for each $i\in\set{1,\dots,d-1}$ there is an isomorphism
  \begin{align*}
    P^{-i}&\cong P_{\lambda_1-1,\dots,\lambda_i-1,\lambda_{i+2},\dots,\lambda_{d+1}}\\
          &=M(x,\lambda_1-1,\dots,\lambda_i-1,\lambda_{i+2},\dots,\lambda_{d+1}),
  \end{align*}
  and finally that
  \[
    \Omega^d(M(\bm{\lambda}))\cong M(x,\lambda_1-1,\dots,\lambda_d-1).
  \]
  Note that we use \th\ref{prop:An:l:proj_inj}\eqref{prop:An:l:proj_inj:it:P} to
  identify the indecomposable projective $A_{\underline{\ell}}^{(d)}$-modules
  appearing in the exact sequence above. We need to show that
  \[
    (x,\lambda_1-1,\dots,\lambda_d-1)\in\mathbf{os}_{\underline{\ell}}^{d+1}.
  \]
  Indeed, by \th\ref{lemma:KS}\eqref{lemma:KS:An:it:2} there is an inequality
  \[
    (\lambda_d-1)-\lambda_{d+1}\leq\ell_{\lambda_d-1}-\ell_{\lambda_{d+1}}.
  \]
  This inequality is equivalent to
  \[
    \len(x,\lambda_1-1,\dots,\lambda_d-1)=\lambda_d-x=\lambda_d-(\lambda_{d+1}+1-\ell_{\lambda_{d+1}})\leq\ell_{\lambda_d-1},
  \]
  which is what we needed to prove. The second claim in the proposition follows
  from De\-fi\-ni\-tion-Proposition 2.15 in \cite{IJ17} which states that a
  $d$-cluster-tilting subcategory of a module category is
  $d\mathbb{Z}$-cluster-tilting if and only if it is closed under taking $d$-th
  syzygies.
\end{proof}

Finally, \th\ref{thm:An:l}\eqref{thm:An:l:it:taud} follows from the following
more general result.

\begin{proposition}
  \th\label{prop:An:l:taud} Let
  $\bm{\lambda}\in\mathbf{os}_{\underline{\ell}}^{d+1}$ be such that
  $M(\bm{\lambda})$ is not projective as an $A_{\underline{\ell}}^{(d)}$-module.
  Then, there is an isomorphism $\tau_d(M(\bm{\lambda}))\cong
  M(\tau_d(\bm{\lambda}))$.
\end{proposition}
\begin{proof}
  Since by assumption $M(\bm{\lambda})$ is not projective as an
  $A_{\underline{\ell}}^{(d)}$-module,
  \th\ref{prop:An:l:proj_inj}\eqref{prop:An:l:proj_inj:it:P} shows that
  \[
    \len(\bm{\lambda})=\lambda_{d+1}-\lambda_1+1\leq\ell_{\lambda_{d+1}}-1.
  \]
  In particular $\lambda_1\neq0$ (see \th\ref{lemma:KS}). By
  \th\ref{prop:An:proj_inj}\eqref{it:A_nd-projectives} $M(\bm{\lambda})$ is not
  projective as an $A_n^{(d)}$-module either. Therefore, since the claim holds
  for $\underline{\ell}=(1,2,\dots,n)$, by
  \th\ref{prop:An:taud}\eqref{it:A_nd-translate} there is a $d$-almost split
  sequence
  \[
    0\to M(\tau_d(\bm{\lambda}))\to M^1\to\cdots\to M^d\to M(\bm{\lambda})\to0
  \]
  in $\add M_n^{(d)}$. In view of \th\ref{thm:d-CT:existence_d-AR_seq} it is enough
  to show that this sequence is contained in $\add M_{\underline{\ell}}^{(d)}$.
  We recall from Proposition 3.19 in \cite{OT12} that
  \[
    M(\bm{\lambda})\oplus M^1\oplus\cdots\oplus M^d\oplus
    M(\tau_d(\bm{\lambda}))=\bigoplus\setP{M(\bm{\mu})}{\bm{\mu}\in[\tau_d(\bm{\lambda}),\bm{\lambda}]\subset\mathbf{os}_n^{d+1}}.
  \]
  Hence we need to show that
  $[\tau_d(\bm{\lambda}),\bm{\lambda}]\subset\mathbf{os}_{\underline{\ell}}^{d+1}$.
  Note that $\bm{\mu}\in[\tau_d(\bm{\lambda}),\bm{\lambda}]$ must satisfy
  \[
    (\mu_1,\mu_{d+1})\in\set{(\lambda_1-1,\lambda_{d+1}-1),\
      (\lambda_1,\lambda_{d+1}-1),\ (\lambda_1-1,\lambda_{d+1}),\
      (\lambda_1,\lambda_{d+1})}.
  \]
  Let $(\mu_1,\mu_{d+1})=(\lambda_1-1,\lambda_{d+1}-1)$. Then
  \[
    \len(\tau_d(\bm{\lambda}))=\lambda_{d+1}-1-(\lambda_1-1)+1=\len(\bm{\lambda})\leq\ell_{\lambda_{d+1}}-1\leq\ell_{\lambda_{d+1}-1}
  \]
  where the rightmost inequality holds since $\underline{\ell}$ is a Kupisch
  series. This shows that $\mu\in\mathbf{os}_{\underline{\ell}}^{d+1}$ in this
  case. It is straightforward to verify that
  $\bm{\mu}\in\mathbf{os}_{\underline{\ell}}^{d+1}$ in the remaining cases.
\end{proof}

We conclude this section with an immediate consequence of
\th\ref{prop:An:l:taud}, which should be compared with Corollary 2.9 in
\cite{ARS97}.

\begin{corollary}
  \th\label{coro:Al:Loewy_length} Let $M$ be an indecomposable direct summand of
  $M_{\underline{\ell}}^{(d)}$ of Loewy length $\ell$. Then, every non-zero
  $A_{\underline{\ell}}^{(d)}$-module in the $\tau_d$-orbit of $M$ also has
  Loewy length $\ell$.
\end{corollary}
\begin{proof}
  Let $\bm{\lambda}\in\mathbf{os}_{\underline{\ell}}^{(d)}$ be such that
  $M(\bm{\lambda})$ is non-projective as an $A_{\underline{\ell}}^{(d)}$-module.
  By \th\ref{lemma:An:len,prop:An:l:taud} we have
  \[
    \len(\tau_d(M(\bm{\lambda})))=\len(M(\tau_d(\bm{\lambda})))=\lambda_{d+1}-1-(\lambda_1-1)+1=\len(M(\bm{\lambda})),
  \]
  as required.
\end{proof}


\section{The higher Nakayama algebras of type $\mathbb{A}_\infty^\infty$}
\label{sec:A-infty}

In this section we introduce a family of categories which are to be thought of
as higher dimensional analogues of the mesh category of type
$\mathbb{ZA}_\infty$. We also introduce infinite analogues of the Nakayama
algebras of type $\mathbb{A}$. The latter categories are essential for our
construction of the higher Nakayama algebras of type $\widetilde{\mathbb{A}}$
and the higher dimensional analogues of the tubes in Section \ref{sec:A-tilde}.

\subsection{The mesh category of type $\mathbb{ZA}_\infty^{(d-1)}$}

We begin this section by introducing the higher dimensional analogues of the
mesh categories of type $\mathbb{ZA}_\infty$.

\begin{setting}
  We fix a positive integer $d$ until further notice.
\end{setting}

Consider the poset of integer numbers\nomenclature[30]{$A_\infty$}{the poset of
  intger numbers $\set{\cdots<-1<0<1<\cdots}$}
\[
  A_\infty=\set{\cdots<-1<0<1<\cdots}
\]
and denote the set of ordered sequences of length $d$ in $A_\infty$
by\nomenclature[31]{$\mathbf{os}^d$}{the poset of ordered sequences of length
  $d$ in the poset $A_\infty$}
\[
  \mathbf{os}^d:=\mathbf{os}^d(A_\infty).
\]

\begin{definition}
  \th\label{def:mesh} The \emph{mesh category of type
    $\mathbb{ZA}_\infty^{(d-1)}$} is the $d$-cone
  $A_\infty^{(d)}$.\nomenclature[32]{$A_\infty^{(d)}$}{the mesh category of type
    $\mathbb{ZA}_\infty^{(d-1)}$} We also define the
  subcategory\nomenclature[33]{$\mathcal{M}_\infty^{(d)}$}{the distinguished
    $d\mathbb{Z}$-cluster-tilting subcategory of $\mmod A_\infty^{(d)}$}
  \begin{equation*}
    \mathcal{M}_\infty^{(d)}:=\add\setP{M(\bm{\lambda})\in\mmod A_\infty^{(d)}}{\bm{\lambda}\in\mathbf{os}^{d+1}}\subseteq\mmod A_\infty^{(d)}.
  \end{equation*}
\end{definition}

A presentation of $A_\infty^{(d)}$ by generators and relations can be given as
follows. By definition, the Gabriel quiver $Q$ of $A_\infty^{(d)}$ has as vertices
the set $\mathbf{os}^d$, that is the set of tuples
$\bm{\lambda}=(\lambda_1,\dots,\lambda_d)$ of integers satisfying
\[
  \lambda_1\leq\cdots\leq\lambda_d.
\]
For each $\bm{\lambda}\in\mathbf{os}^d$ and each $i\in\set{1,\dots,d}$ such that
$\bm{\lambda}+e_i$ is again an ordered sequence there is an arrow in $Q$ of the
form
\[
  a_i=a_i(\bm{\lambda})\colon\bm{\lambda}\to\bm{\lambda}+e_i.
\]
With the above notation, $A_\infty^{(d)}$ is identified with the (non-unital)
algebra $\mathbbm{k}Q/I$ where $I$ is the two-sided ideal of $\mathbbm{k}Q$
generated by the relations
\[
  a_j(\bm{\lambda}+e_i)a_i(\bm{\lambda})-a_i(\bm{\lambda}+e_j)a_j(\bm{\lambda})
\]
for each $\bm{\lambda}\in\mathbf{os}^d$ and each $i,j\in\set{1,\dots,d}$ such
that $i\neq j$. By convention, $a_i(\bm{\lambda})=0$ whenever $\bm{\lambda}$ or
$\bm{\lambda}+e_i$ are not vertices of $Q$, hence some of the above relations
are in fact zero relations. For example, $A_\infty^{(1)}$ is just the path
category of the infinite quiver $\cdots\to-1\to0\to 1\to\cdots$. The Gabriel
quivers of $A_\infty^{(2)}$ and $A_\infty^{(3)}$ are shown in Figure
\ref{fig:A-infty} (note that the Gabriel quiver of $A_{\infty}^{(3)}$
is canonically embedded in 3-dimensional space). In particular, note that
$A_\infty^{(2)}$ is isomorphic to the mesh category of type
$\mathbb{ZA}_\infty$.
\begin{figure}
  \begin{center}
  \includegraphics[width=0.65\textwidth]{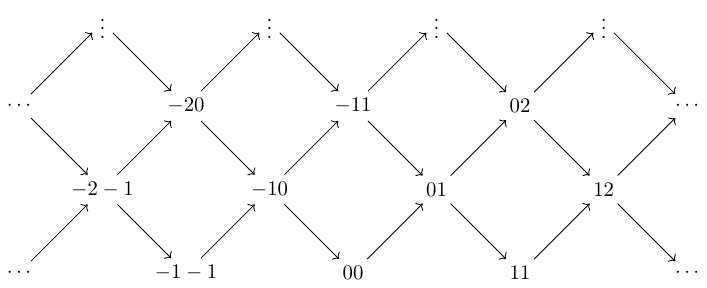}\vskip2em
  \includegraphics[width=\textwidth]{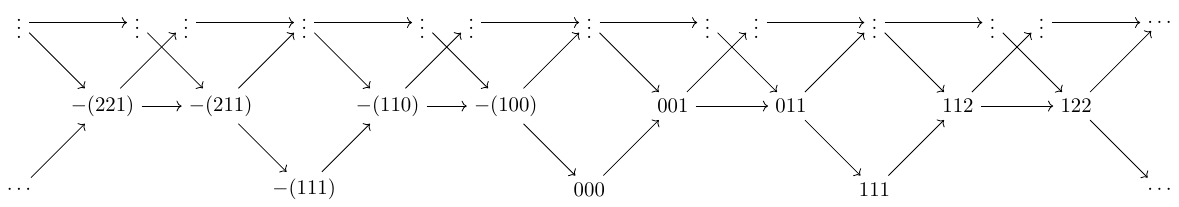}
  \end{center}
  \caption{The Gabriel quivers of $A_\infty^{(2)}$ (top) and $A_\infty^{(3)}$
    (bottom).}
  \label{fig:A-infty}
\end{figure}

\begin{remark}
  \th\label{rmk:A-infty:lZ:mesh} The choice of terminology in \th\ref{def:mesh}
  can be justified as follows. Let
  \[
    \mathbf{os}_{[0,\infty)}^{d}:=\setP{\bm{\lambda}\in\mathbf{os}^{d+1}}{0\leq \lambda_1}.
  \]
  Moreover, for $s\in\mathbb{Z}$ define
  \[
    \mathbf{os}^{d+1}(s):=\setP{\bm{\lambda}\in\mathbf{os}^{d+1}}{\lambda_1=s}.
  \]
  It is clear that there is a bijection
  \[
    \textstyle\dot\bigcup_{s\in\mathbb{Z}}\mathbf{os}^{d+1}(s)=\mathbf{os}^{d+1}\to\mathbf{os}_{[0,\infty)}^d\times\mathbb{Z}
  \]
  given by $
  \bm{\lambda}\mapsto((\lambda_2-\lambda_1,\dots,\lambda_{d+1}-\lambda_1),\lambda_1)$.
  Note that the subcategory of $A_\infty^{(d+1)}$ spanned by
  $\mathbf{os}^{d+1}(s)$ can be thought of as a higher dimensional analogue of
  the path category of the quiver
  \[
    (s,s)\to (s,s+1)\to(s,s+2)\to\cdots,
  \]
  which corresponds to the case $d=1$. Also, note that for each
  $\bm{\lambda}\in\mathbf{os}^{d+1}$ there is a unique arrow
  \[
    \bm{\lambda}\to\bm{\lambda}+e_1
  \]
  precisely when $\lambda_1+1\leq\lambda_2$. Therefore the Gabriel quiver of
  $A_\infty^{(d+1)}$ can equivalently be described as the quiver with vertex set
  $\mathbf{os}_{[0,\infty)}^d\times\mathbb{Z}$ and arrows
  \[
    b_i=b_i(\bm{\lambda},s)\colon(\bm{\lambda},s)\to(\bm{\lambda}+e_i,s)
  \]
  for each $i\in\set{1,\dots,d}$, whenever $\bm{\lambda}+e_i$ belongs to
  $\mathbf{os}_{[0,\infty)}^d$, (these arrows correspond to arrows
  \[
    a_{i+1}=a_{i+1}(\bm{\lambda})\colon\bm{\lambda}\to\bm{\lambda}+e_{i+1}
  \]
  for $i=1,\dots,d$ in our original description of $A_\infty^{(d+1)}$) as well
  as `connecting arrows'
  \[
    b_0=b_0(\bm{\lambda},s)\colon(\bm{\lambda},s)\to(\tau_d(\bm{\lambda}),s+1)
  \]
  whenever $\lambda_1>0$, where $\tau_d(\bm{\lambda})=\bm{\lambda}-(1,\dots,1)$
  as in the case of $\mathbf{os}^d_n$ (these arrows correspond to arrows
  \[
    a_1=a_1(\bm{\lambda})\colon\bm{\lambda}\to\bm{\lambda}+e_1
  \]
  in our original description of $A_\infty^{(d+1)}$). Under this bijection, the
  defining relations of $A_\infty^{(d+1)}$ are naturally divided into two
  types:
  \begin{itemize}
  \item[(I)] For each $(\bm{\lambda},s)\in\mathbf{os}_{[0,\infty)}^d\times\mathbb{Z}$
    and each $1\leq i<j\leq d$ there is a relation
    \[
      b_j(\bm{\lambda}+e_i,s)b_i(\bm{\lambda},s)-b_i(\bm{\lambda}+e_j,s)b_j(\bm{\lambda},s).
    \]
  \item[(II)] For each $(\bm{\lambda},s)\in\mathbf{os}_{[0,\infty)}^d\times\mathbb{Z}$
    and each $1\leq j\leq d$ there is a relation
    \[
      b_j(\tau_d(\bm{\lambda}),s+1)b_0(\bm{\lambda},s)-b_0(\bm{\lambda}+e_j,s)b_j(\bm{\lambda},s).
    \]
  \end{itemize}
  As usual, some of the above commutativity relations are in fact zero
  relations, depending on whether the arrows involved are present in the quiver
  or not. Note that the Gabriel quiver of $A_\infty^{(2)}$ is just the
  repetitive quiver $\mathbb{ZA}_\infty$ and the above relations reduce to the
  mesh relations (there are no relations of type (I)).
\end{remark}

\begin{remark}
  \th\label{rmk:A-infty:locally_bounded} The category $A_\infty^{(d)}$ is
  \emph{not} locally bounded. Indeed, given $\bm{\lambda}\in\mathbf{os}^d$, for
  each $i>0$ there are obvious interlacings
  $\bm{\lambda}\rightsquigarrow(\bm{\lambda}+ie_d)$ and
  $(\bm{\lambda}-ie_1)\rightsquigarrow\bm{\lambda}$. In view of
  \th\ref{prop:poset:d-cone:interlacing}, these interlacings correspond to
  non-zero morphisms $\bm{\lambda}\to(\bm{\lambda}+ie_{d})$ and
  $(\bm{\lambda}-ie_{1})\to\bm{\lambda}$ in $A_\infty^{(d)}$.
\end{remark}

Although the category $A_\infty^{(d)}$ is not locally bounded, it can be
`approximated' by locally bounded categories. Let us make this statement
precise.

\begin{notation}
  \th\label{not:A-infty:ab} Let $a\leq b$ be integers and define
  \[
    \mathbf{os}_{[a,b]}^{d}:=\setP{\bm{\lambda}\in\mathbf{os}^{d}}{a\leq\lambda_1\leq\cdots\leq\lambda_d\leq
      b}.
  \]
  By construction, the idempotent quotient
  \[
    A_{[a,b]}^{(d)}:=A_\infty^{(d)}/[\mathbf{os}^d\setminus\mathbf{os}_{[a,b]}^d]
  \]
  is isomorphic to $A_n^{(d)}$ where $n=b-a+1$. The unique $d$-cluster-tilting
  subcategory of $\mmod A_{[a,b]}^{(d)}$ is precisely
  \[
    \mathcal{M}_{[a,b]}^{(d)}:=\add\setP{M(\bm{\lambda})\in\mmod
      A_{[a,b]}^{(d)}}{\bm{\lambda}\in\mathbf{os}_{[a,b]}^{d+1}}
  \]
  (compare with \th\ref{thm:M_nd}). As explained in Subsection
  \ref{subsec:locally_bounded_categories}, the canonical functor
  $A_\infty^{(d)}\to A_{[a,b]}^{(d)}$ induces a fully faithful exact functor
  \[
    \mmod A_{[a,b]}^{(d)}\hookrightarrow\mmod A_\infty^{(d)}
  \]
  which clearly restricts to a fully faithful functor
  \[
    \mathcal{M}_{[a,b]}^{(d)}\hookrightarrow\mathcal{M}_\infty^{(d)}.
  \]
  Finally, since we deal with finite dimensional modules and the difference
  $b-a$ can be arbitrarily large, we conclude that
  \[
    \textstyle\mmod A_\infty^{(d)}=\bigcup_{a\leq b}\mmod
    A_{[a,b]}^{(d)}\qquad\text{and}\qquad\mathcal{M}_\infty^{(d)}=\bigcup_{a\leq
      b}\mathcal{M}_{[a,b]}^{(d)}.
  \]
\end{notation}

We begin our study of the category $A_\infty^{(d)}$ with a simple observation.

\begin{proposition}
  \th\label{prop:A-infty:NP} The abelian category $\mmod A_\infty^{(d)}$
  contains no non-zero projective objects and no non-zero injective objects.
\end{proposition}
\begin{proof}
  Let $M\in\mmod A_\infty^{(d)}$ be non-zero. Then, there exist integers $a\leq
  b$ such that $M$ is non-projective as an $A_{[a,b]}^{(d)}$-module. Therefore
  there exists a projective $A_{[a,b]}^{(d)}$-module $P$ and a non-split
  epimorphism $f\colon P\to M$ in $\mmod A_{[a,b]}^{(d)}$. Since the canonical
  inclusion $\mmod A_{[a,b]}^{(d)}\hookrightarrow\mmod A_\infty^{(d)}$ is exact,
  $f$ is also an epimorphism in $\mmod A_\infty^{(d)}$ whence $M$ is not
  projective. A dual argument shows that $\mmod A_\infty^{(d)}$ contains no
  non-zero injective objects.
\end{proof}

Let $\mathcal{A}$ and $\mathcal{B}$ be abelian categories. Following
\cite{Psa14}, a fully faithful functor
$F\colon\mathcal{A}\hookrightarrow\mathcal{B}$ is called a \emph{homological
  embedding} if for all $X,Y\in\mathcal{A}$ and for all $i\geq1$ the induced
homomorphism
\[
  F\colon\Ext_{\mathcal{A}}^i(X,Y)\to\Ext_{\mathcal{B}}^i(FX,FY)
\]
is an isomorphism, see \cite{Psa14} and \cite{GL91} for further information on
this notion.

\begin{proposition}
  \th\label{prop:A-infty:Ext} The following statements hold.
  \begin{enumerate}
  \item\label{prop:A-infty:Ext:it:ab-cd} Let $\alpha\leq a\leq b\leq \beta$ be integers.
    Then, the canonical inclusion $\mmod A_{[a,b]}^{(d)}\hookrightarrow\mmod
    A_{[\alpha,\beta]}^{(d)}$ is a homological embedding. In particular, the canonical
    exact functor
    \[
      \Db(\mmod A_{[a,b]}^{(d)})\to\Db(\mmod A_{[\alpha,\beta]}^{(d)})
    \]
    is fully faithful.
  \item\label{prop:A-infty:Ext:it:ab-infty} Let $a\leq b$ be integers. Then, the
    canonical inclusion $\mmod A_{[a,b]}^{(d)}\hookrightarrow\mmod
    A_\infty^{(d)}$ is a homological embedding. In particular, the canonical
    exact functor
    \[
      \Db(\mmod A_{[a,b]}^{(d)})\to\Db(\mmod A_\infty^{(d)})
    \]
    is fully faithful.
  \end{enumerate}
\end{proposition}
\begin{proof}
  \eqref{prop:A-infty:Ext:it:ab-cd} By Proposition 4.9 in \cite{GL91}, it is
  enough to prove that
  \[
    \Ext_{A_{[\alpha,\beta]}^{(d)}}^i(A_{[a,b]}^{(d)},A_{[a,b]}^{(d)})=0
  \]
  for each $i>0$. But this follows immediately from the fact that
  $A_{[a,b]}^{(d)}\in\mathcal{M}_{[\alpha,\beta]}^{(d)}$ and the formulae for the
  extension spaces given in \th\ref{prop:An:Hom_Ext}.

  \eqref{prop:A-infty:Ext:it:ab-infty} Let $M$ and $N$ be finite dimensional
  $A_{[a,b]}^{(d)}$-modules and $i\geq0$. Suppose that an exact sequence
  $\delta\in\Ext_{A_{[a,b]}^{(d)}}^i(M,N)$ is trivial in $\mmod A_\infty^{(d)}$
  in the sense of Yoneda. Thus, there exists a finite zig-zag of equivalences
  from $\delta$ to the trivial exact sequence. Let $\alpha\leq a\leq b\leq\beta$ be integers such that all of the exact sequences appearing in this zig-zag
  are contained in $\mmod A_{[\alpha,\beta]}^{(d)}$. Thus, $\delta$ is trivial as an
  element of $\Ext_{A_{[\alpha,\beta]}^{(d)}}^i(M,N)$. Since, by the previous statement,
  the canonical embedding $\mmod A_{[a,b]}^{(d)}\hookrightarrow\mmod
  A_{[\alpha,\beta]}^{(d)}$ is homological, $\delta$ is also trivial as an element of
  $\Ext_{A_{[a,b]}^{(d)}}^i(M,N)$. This shows that the induced homomorphism
  \[
    \Ext_{A_{[a,b]}^{(d)}}^i(M,N)\to\Ext_{A_\infty^{(d)}}^i(M,N)
  \]
  is injective. A similar argument can be used to show that it is also
  surjective. We leave the details to the reader. The fact that the canonical
  exact functor
  \[
    \Db(\mmod A_{[a,b]}^{(d)})\to\Db(\mmod A_\infty^{(d)})
  \]
  is fully faithful follows from Theorem 2.1 in \cite{Yao96} (note that in this
  case we cannot use Proposition 4.9 in \cite{GL91} since $\mmod A_\infty^{(d)}$
  does not coincide with the category of finitely presented
  $A_\infty^{(d)}$-modules).
\end{proof}

As a consequence of \th\ref{prop:A-infty:Ext} we show that the higher
Auslander--Reiten formulae hold in $\mathcal{M}_\infty^{(d)}$ and thus obtain a
combinatorial description of the spaces of degree $d$ extensions in
$\mathcal{M}_\infty^{(d)}$ analogous to that in \th\ref{prop:An:Hom_Ext}.

\begin{proposition}
  \th\label{prop:A-infty:Hom_Ext} Let
  $\bm{\lambda},\bm{\mu}\in\mathbf{os}^{d+1}$. The following statements hold.
  \begin{enumerate}
  \item\label{prop:A-infty:Hom_Ext:it:AR_formulas} There are bifunctorial
    isomorphisms
    \begin{align*}
      D\Hom_{A_\infty^{(d)}}(M(\bm{\mu}),M(\tau_d(\bm{\lambda})))\cong\Ext_{A_\infty^{(d)}}^d(M(\bm{\lambda}),M(\bm{\mu}))\intertext{and}D\Hom_{A_\infty^{(d)}}(M(\tau_d^-(\bm{\mu})),M(\bm{\lambda}))\cong\Ext_{A_\infty^{(d)}}^d(M(\bm{\lambda}),M(\bm{\mu})).
    \end{align*}
  \item\label{prop:A-infty:Hom_Ext:it:Hom_Ext} There are isomorphisms
    \begin{align*}
      \Hom_{A_\infty^{(d)}}(M(\bm{\lambda}),M(\bm{\mu}))&\cong\begin{cases}
        \mathbbm{k}&\text{if }\bm{\lambda}\rightsquigarrow\bm{\mu},\\
        0&\text{otherwise;}
      \end{cases}\intertext{and}
           \Ext_{A_\infty^{(d)}}^d(M(\bm{\lambda}),M(\bm{\mu}))&\cong\begin{cases}
             \mathbbm{k}&\text{if }\bm{\mu}\rightsquigarrow\tau_d(\bm{\lambda}),\\
             0&\text{otherwise.}
           \end{cases}
    \end{align*}
  \end{enumerate}
\end{proposition}
\begin{proof}
  \eqref{prop:A-infty:Hom_Ext:it:AR_formulas} Let $a\leq b$ be integers such
  that $M(\bm{\lambda})$ and $M(\bm{\mu})$ are $\mmod A_{[a,b]}^{(d)}$-modules.
  By \th\ref{prop:A-infty:Ext} there is an isomorphism of vector spaces
  \[
    \Ext_{A_\infty^{(d)}}^d(M(\bm{\lambda}),M(\bm{\mu}))\cong\Ext_{A_{[a,b]}^{(d)}}^d(M(\bm{\lambda}),M(\bm{\mu})).
  \]
  Moreover, since the canonical inclusion $\mmod
  A_{[a,b]}^{(d)}\hookrightarrow\mmod A_\infty^{(d)}$ is fully faithful there is
  an isomorphism
  \[
    \Hom_{A_\infty^{(d)}}(M(\bm{\lambda}),M(\bm{\mu}))\cong\Hom_{A_{[a,b]}^{(d)}}(M(\bm{\lambda}),M(\bm{\mu}))
  \]
  The required isomorphisms follow from
  \th\ref{thm:d-CT:existence_d-AR_seq}\eqref{it:AR-formulas} applied to the
  finite dimensional algebra $A_{[a,b]}^{(d)}$, taking into account that every
  non-zero module in $\mathcal{M}_{[a,b]}^{(d)}$ has projective dimension either
  $0$ or $d$, see \th\ref{prop:An:taud}\eqref{it:A_nd-projective_resolutions}.
  
  \eqref{prop:A-infty:Hom_Ext:it:Hom_Ext} The first isomorphism is proven in
  \th\ref{prop:d-cone:interval_modules:interlacing} while the second one follows
  immediately from the first one together with statement
  \eqref{prop:A-infty:Hom_Ext:it:AR_formulas}.
\end{proof}

The next theorem describes basic representation-theoretic properties of
$A_\infty^{(d)}$ from the viewpoint of higher Auslander--Reiten theory. Its
proof is postponed to Subsection \ref{subsec:proof:thm:mesh} as it relies on the
content of the next subsection.

\begin{theorem}
  \th\label{thm:mesh} The following statements hold.
  \begin{enumerate}
  \item\label{thm:mesh:it:gldimd} The abelian category $\mmod A_\infty^{(d)}$
    has global dimension $d$.
  \item\label{thm:mesh:it:d-CT} $\mathcal{M}_\infty^{(d)}$ is a
    $d$-cluster-tilting subcategory of $\mmod A_\infty^{(d)}$.
  \item\label{thm:mesh:it:d-almost-split} The category
    $\mathcal{M}_\infty^{(d)}$ has $d$-almost split sequences. Moreover, for
    each $\bm{\lambda}\in\mathbf{os}^{d+1}$ there are isomorphisms
    \[
      \tau_d(M(\bm{\lambda}))\cong M(\tau_d(\bm{\lambda}))\qquad\text{and}\qquad
      \tau_d^-(M(\bm{\lambda}))\cong M(\tau_d^-(\bm{\lambda})).
    \]
  \item\label{thm:mesh:it:taud:isos} For every indecomposable
    $A_\infty^{(d)}$-module $M\in\mathcal{M}_\infty^{(d)}$ and every
    $i,j\in\mathbb{Z}$ there is an isomorphism $\tau_d^i(M)\cong\tau_d^j(M)$ if
    and only if $i=j$.
  \item\label{thm:mesh:it:taud} For every simple $A_\infty^{(d)}$-module
    $S\in\mathcal{M}_\infty^{(d)}$ and for every $i\in\mathbb{Z}$ the
    $A_\infty^{(d)}$-module $\tau_d^i(S)$ is simple.
  \end{enumerate}
\end{theorem}

\subsection{The higher Nakayama categories of type $\mathbb{A}_\infty^\infty$}

Let $d$ be a positive integer. Our proof of \th\ref{thm:mesh} relies on a
detailed study of certain idempotent quotients of $A_\infty^{(d)}$, which we now
introduce. These quotients, also essential in the construction of the higher
Nakayama algebras of type $\widetilde{\mathbb{A}}$ in Section \ref{sec:A-tilde},
should be thought of as higher dimensional analogues of the admissible quotients
of the path category of the infinite quiver
\[
  \cdots\to-1\to0\to1\to\cdots.
\]
We begin by extending the definition of a Kupisch series (\th\ref{def:KS:An}) to
infinite tuples of positive integers.

\begin{definition}
  \th\label{def:KS:A-infty} Let
  $\underline{\ell}=(\dots,\ell_{-1},\ell_0,\ell_1,\dots)$ be an infinite tuple
  of positive integers. We say that $\underline{\ell}$ is a \emph{(connected)
    Kupisch series of type $\mathbb{A}_\infty^\infty$} if for all
  $i\in\mathbb{Z}$ there are inequalities
  \[
    2\leq\ell_{i}\leq \ell_{i-1}+1.
  \]
  We denote the set of Kupisch series of type $\mathbb{A}_\infty^\infty$ by
  $\KSAi$.
\end{definition}

We use the following class of Kupisch series for constructing the higher
Nakayama algebras of type $\widetilde{\mathbb{A}}$ in Section \ref{sec:A-tilde}.

\begin{definition}
  \th\label{def:KS:A-infty:l-bounded} Let $\ell\geq2$ be an integer. A Kupisch
  series $\underline{\ell}$ of type $\mathbb{A}_\infty^\infty$ is
  \emph{$\ell$-bounded} if for all $i\in\mathbb{Z}$ there is an inequality
  $\ell_i\leq\ell$.
\end{definition}

\begin{setting}
  We fix positive integers $d$ and $\ell\geq2$ until further notice.
\end{setting}

The next definition is an infinite analogue of \th\ref{def:An:l}.

\begin{definition}
  \th\label{def:A-infty:l} Let $\underline{\ell}$ be an $\ell$-bounded Kupisch
  series of type $\mathbb{A}_\infty^\infty$.
  \begin{enumerate}
  \item The \emph{$\underline{\ell}$-restriction of $\mathbf{os}^{d+1}$} is the
    subset
    \begin{equation*}
      \mathbf{os}_{\underline{\ell}}^{d+1}:=\setP{\bm{\lambda}\in\mathbf{os}^{d+1}}{\len(\bm{\lambda})\leq\ell_{\lambda_{d+1}}}.
    \end{equation*}
  \item The \emph{$(d+1)$-Nakayama algebra with Kupisch series
      $\underline{\ell}$} is the idempotent quotient
    \[
      A_{\underline{\ell}}^{(d+1)}:=A_\infty^{(d+1)}/[\mathbf{os}^{d+1}\setminus\mathbf{os}_{\underline{\ell}}^{d+1}].
    \]
  \item We define the subcategory
    \begin{equation*}
      \mathcal{M}_{\underline{\ell}}^{(d)}:=\add\setP{M(\bm{\lambda})\in\mmod A_\infty^{(d)}}{\bm{\lambda}\in\mathbf{os}_{\underline{\ell}}^{d+1}}.
    \end{equation*}
  \end{enumerate}
\end{definition}

We make a couple of elementary but important observations.

\begin{proposition}
  \th\label{prop:A-infty:l:locally_bounded} Let $d,\ell\geq2$ be integers and
  $\underline{\ell}$ an $\ell$-bounded Kupisch series of type
  $\mathbb{A}_\infty^\infty$. Then, the category $A_{\underline{\ell}}^{(d)}$ is
  locally bounded.
\end{proposition}
\begin{proof}
  Let $\bm{\lambda},\bm{\mu}\in\mathbf{os}_{\underline{\ell}}^d$ be such that
  $\bm{\lambda}\rightsquigarrow\bm{\mu}$ and
  $[\bm{\lambda},\bm{\mu}]\subset\mathbf{os}_{\underline{\ell}}^d$. By
  definition, there are inequalities
  \[
    \lambda_1\leq\mu_1\leq\cdots\leq\lambda_d\leq\mu_d\leq\ell+\mu_1-1\leq\ell+\lambda_2-1
  \]
  where the penultimate inequality follows from the inequalities
  \[
    \len(\bm{\mu})=\mu_d-\mu_1+1\leq\ell_{\mu_d}\leq\ell.
  \]
  which hold by assumption. Therefore for fixed
  $\bm{\lambda}\in\mathbf{os}_{\underline{\ell}}^d$ the set
  \[
    \setP{\bm{\mu}\in\mathbf{os}_{\underline{\ell}}^d}{\bm{\lambda}\rightsquigarrow\bm{\mu}\text{
        and }[\bm{\lambda},\bm{\mu}]\subset\mathbf{os}_{\underline{\ell}}^d}
  \]
  is finite. A similar argument shows that for fixed
  $\bm{\mu}\in\mathbf{os}_{\underline{\ell}}^d$ the set
  \[
    \setP{\bm{\lambda}\in\mathbf{os}_{\underline{\ell}}^d}{\bm{\lambda}\rightsquigarrow\bm{\mu}\text{
        and }[\bm{\lambda},\bm{\mu}]\subset\mathbf{os}_{\underline{\ell}}^d}
  \]
  is finite. The claim now follows from \th\ref{prop:poset:d-cone:interlacing}
  and the definition of $A_{\underline{\ell}}^{(d)}$.
\end{proof}

\begin{proposition}
  \th\label{prop:A-infinity:l:1} Let $d$ and $\ell\geq2$ be integers and
  $\underline{\ell}$ an $\ell$-bounded Kupisch series of type
  $\mathbb{A}_\infty^\infty$. Then,
  $\mathcal{M}_{\underline{\ell}}^{(d)}\subseteq\mmod
  A_{\underline{\ell}}^{(d)}$.
\end{proposition}
\begin{proof}
  The proof of \th\ref{prop:An:l:M} carries over. We leave the details to the
  reader.
\end{proof}

The following theorem is an infinite analogue of \th\ref{thm:An:l}.

\begin{theorem}
  \th\label{thm:Ail} Let $d$ be a positive integer. The following statements
  hold.
  \begin{enumerate}
  \item\label{it:thm:Ail:KS} For each $i\in\mathbb{Z}$ the indecomposable
    projective $A_{\underline{\ell}}^{(d)}$-module at the vertex $(i,\dots,i)$
    has Loewy length $\ell_i$.
  \item\label{it:thm:Ail:dZ-RF} $\mathcal{M}_{\underline{\ell}}^{(d)}$ is a
    $d\mathbb{Z}$-cluster-tilting subcategory of $\mmod
    A_{\underline{\ell}}^{(d)}$.
  \item\label{it:thm:Ail:taud-formula} Let
    $\bm{\lambda}\in\mathbf{os}_{\underline{\ell}}^{d+1}$ be such that
    $M(\bm{\lambda})$ is non-projective as an
    $A_{\underline{\ell}}^{(d)}$-module. Then, there is an isomorphism of
    $A_{\underline{\ell}}^{(d)}$-modules
    \[
      \tau_d(M(\bm{\lambda}))\cong M(\tau_d(\bm{\lambda})).
    \]
  \item\label{it:thm:Ail:taud-formula:inj} Let
    $\bm{\lambda}\in\mathbf{os}_{\underline{\ell}}^{d+1}$ be such that
    $M(\bm{\lambda})$ is non-injective as an
    $A_{\underline{\ell}}^{(d)}$-module. Then, there is an isomorphism of
    $A_{\underline{\ell}}^{(d)}$-modules
    \[
      \tau_d^-(M(\bm{\lambda}))\cong M(\tau_d^-(\bm{\lambda})).
    \]
  \item\label{it:thm:Ail:taud} For every simple
    $A_{\underline{\ell}}^{(d)}$-module
    $S\in\mathcal{M}_{\underline{\ell}}^{(d)}$ and for every $i\in\mathbb{Z}$
    the $A_{\underline{\ell}}^{(d)}$-module $\tau_d^i(S)$ is simple. Moreover,
    for every pair of integers $i$ and $j$ there is an isomorphism
    $\tau_d^i(S)\cong\tau_d^j(S)$ if and only if $i=j$.
  \end{enumerate}
\end{theorem}

The proof of \th\ref{thm:Ail} is completely analogous to that of
\th\ref{thm:An:l}. The main difference is that we need to consider a suitable
replacement of the finite dimensional algebra $A_n^{(d)}$ for which the claim is
essentially known by results of Iyama and Opperman \cite{IO13} (some arguments
are needed in order to translate their results to our specific combinatorial
framework). The general case then follows by iterated application of
\th\ref{lemma:d-CT:idempotent_reduction}.

\begin{notation}
  To simplify the notation, we write\nomenclature[35]{$\mathbb{Z}\ell$}{the
    $\ell$-bounded Kupisch series $(\dots,\ell,\ell,\ell,\dots)$ of type
    $\mathbb{A}_\infty^\infty$}
  \[
    A_{\mathbb{Z}\ell}^{(d)}:=A_{(\dots,\ell,\ell,\ell,\dots)}^{(d)}.
  \]
\end{notation}

\begin{remark}
  \th\label{rmk:A-infty:lZ:mesh2} The category $A_{\mathbb{Z}\ell}^{(d+1)}$
  should be thought of as a higher dimensional analogue of the mesh category of
  type $\mathbb{ZA}_\ell$, see \th\ref{rmk:A-infty:lZ:mesh}.
\end{remark}

\begin{remark}
  The following observations should be compared with the discussion in
  \th\ref{not:A-infty:ab}. Let $\ell\geq2$ be an integer. As explained in
  Subsection \ref{subsec:locally_bounded_categories}, the canonical functor
  $A_\infty^{(d)}\to A_{\mathbb{Z}\ell}^{(d)}$ induces a fully faithful exact
  functor
  \[
    \mmod A_{\mathbb{Z}\ell}^{(d)}\hookrightarrow\mmod A_\infty^{(d)}
  \]
  which clearly restricts to a fully faithful functor
  \[
    \mathcal{M}_{\mathbb{Z}\ell}^{(d)}\hookrightarrow\mathcal{M}_\infty^{(d)}.
  \]
  Moreover,
  \[
    \textstyle\mmod A_\infty^{(d)}=\bigcup_{{\ell\geq2}}\mmod
    A_{\mathbb{Z}\ell}^{(d)}\qquad\text{and}\qquad\mathcal{M}_\infty^{(d)}=\bigcup_{\ell\geq2}\mathcal{M}_{\mathbb{Z}\ell}^{(d)}.
  \]
\end{remark}

For a finite dimensional algebra $A$, we denote its repetitive category by
$\widehat{A}$, see Chapter II in \cite{Hap88} for the
definition.\nomenclature[38]{$\widehat{A}$}{the repetitive category of a finite
  dimensional algebra $A$} The next result combines Theorem 6.12 in \cite{Iya11}
and Theorem 4.7 in \cite{IO13} applied to our specific setting.

\begin{proposition}\th\label{prop:A-infty:lZ:equivalences}
  There are equivalences of categories
  \[
    A_{\mathbb{Z}\ell}^{(d+1)}\cong\widehat{A_{\ell-1}^{(d+1)}}\qquad\text{and}\qquad\add
    A_{\mathbb{Z}\ell}^{(d+1)}\cong\mathcal{U}(A_{\ell}^{(d)}).
  \]
\end{proposition}
\begin{proof}
  We first observe that \th\ref{prop:An:Hom_Ext} readily implies that there are
  isomorphisms of algebras
  \[
    \underline{\End}_{A_{\ell}^{(d)}}(M_{\ell}^{(d)})\cong\End_{A_\ell^{(d)}}(\bigoplus\setP{M(\bm{\lambda})}{\bm{\lambda}\in\mathbf{os}^{d+1}:\lambda_1>0})\cong A_{\ell-1}^{(d+1)}.
  \]
  Indeed, it suffices to note that if
  $\bm{\lambda},\bm{\mu}\in\mathbf{os}^{d+1}$ are such that $\lambda_1=0$ and
  $\mu_1>0$, then $\bm{\mu}\not\rightsquigarrow\bm{\lambda}$. Then, according to
  Theorem 4.7 in \cite{IO13}, there is an equivalence of additive categories
  \[
    \add\widehat{A_{\ell-1}^{(d+1)}}\cong\mathcal{U}(A_{\ell}^{(d)}).
  \]
  Therefore we only need to prove that there is an equivalence of categories
  \[
    A_{\mathbb{Z}\ell}^{(d+1)}\cong\widehat{A_{\ell-1}^{(d+1)}}.
  \]
  Although this is a straightforward verification, we provide some details for
  the convenience of the reader. We use the main theorem in \cite{Sch99}, which
  describes the quiver with relations of the repetitive category of a finite
  dimensional algebra defined by a quiver with relations which are zero
  relations or commutativity relations.

  Motivated by Lemma 4.9 in \cite{IO13}, for
  $\bm{\lambda}\in\mathbf{os}_{\mathbb{Z}\ell}^{d+1}$ we define
  \[
    \mathbb{S}(\bm{\lambda}):=(\lambda_2,\dots,\lambda_{d+1},\lambda_1+\ell-1).
  \]
  Note that
  \[
    (\lambda_1+\ell-1)-\lambda_{d+1}=\ell-\len(\bm{\lambda})\geq0
  \]
  and
  \[
    \len(\mathbb{S}(\bm{\lambda}))=(\lambda_1+\ell-1)-\lambda_2+1=\lambda_1+\ell-\lambda_2\leq\ell,
  \]
  therefore $\mathbb{S}(\bm{\lambda})\in\mathbf{os}_{\mathbb{Z}\ell}^{d+1}$. It
  is easy to verify that this defines a bijection
  \[
    \mathbb{S}\colon\mathbf{os}_{\mathbb{Z}\ell}^{d+1}\to\mathbf{os}_{\mathbb{Z}\ell}^{d+1}
  \]
  which induces an automorphism
  \[
    \mathbb{S}\colon A_{\mathbb{Z}\ell}^{(d+1)}\to A_{\mathbb{Z}\ell}^{(d+1)}.
  \]
  It is also straightforward to verify that $\mathbf{os}_{\ell-1}^{d+1}$ is a
  fundamental domain for the action of $\mathbb{S}$ on
  $\mathbf{os}_{\mathbb{Z}\ell}^{d+1}$, by which we mean that the function
  \begin{align*}
    \mathbf{os}_{\ell-1}^{d+1}\times\mathbb{Z}&\to\mathbf{os}_{\mathbb{Z}\ell}^{d+1}\\
    (\bm{\lambda},i)&\mapsto\mathbb{S}^i(\bm{\lambda})
  \end{align*}
  is bijective. We also note that if $\bm{\lambda}\in\mathbf{os}_{\ell-1}^{d+1}$
  is such that $\lambda_{d+1}=\ell-2$, then
  \[
    \bm{\lambda}+e_{d+1}=(\lambda_1,\dots,\lambda_d,\ell-1)=\mathbb{S}(0,\lambda_1,\dots,\lambda_d).
  \]

  The above discussion shows that the Gabriel quiver of
  $A_{\mathbb{Z}\ell}^{(d+1)}$ can be equivalently described as the quiver
  obtained by gluing together a $\mathbb{Z}$-indexed collection of copies of the
  Gabriel quiver of $A_{\ell-1}^{(d+1)}$, the $i$-th copy joined to the
  $(i+1)$-th by `connecting arrows'
  \[
    \mathbb{S}^i(\bm{\lambda})\to\mathbb{S}^i(\bm{\lambda})+e_{d+1}=\mathbb{S}^i(0,\lambda_1,\dots,\lambda_d)
  \]
  for each $\bm{\lambda}\in\mathbf{os}_{\ell-1}^{d+1}$ such that
  $\lambda_{d+1}=\ell-2$. Moreover, the proof of
  \th\ref{prop:An:proj_inj}\eqref{it:A_nd-projectives} shows that
  $(0,\lambda_1,\dots,\lambda_d)$ is the initial vertex of any longest path in
  $A_{\ell-1}^{(d)}$ ending at $\bm{\lambda}$.

  \begin{figure}
    \begin{center}
    \includegraphics[width=\textwidth]{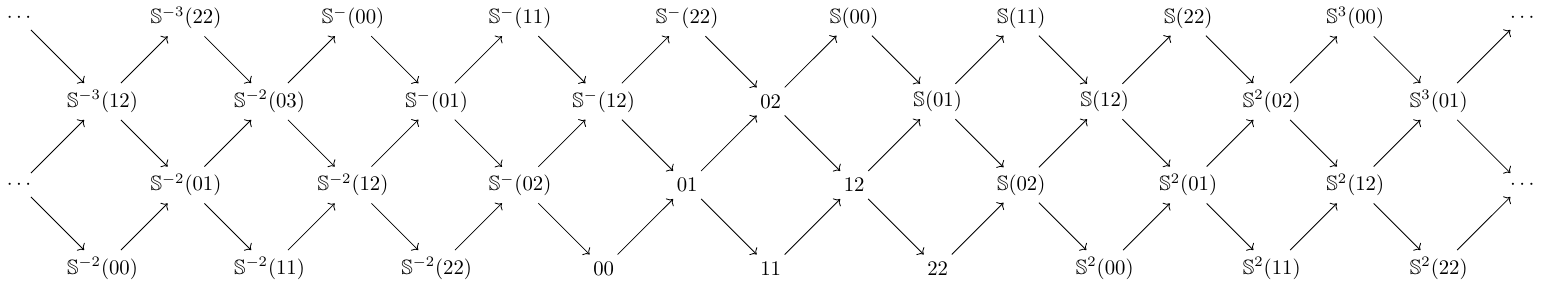}\vskip2em
    \includegraphics[width=\textwidth]{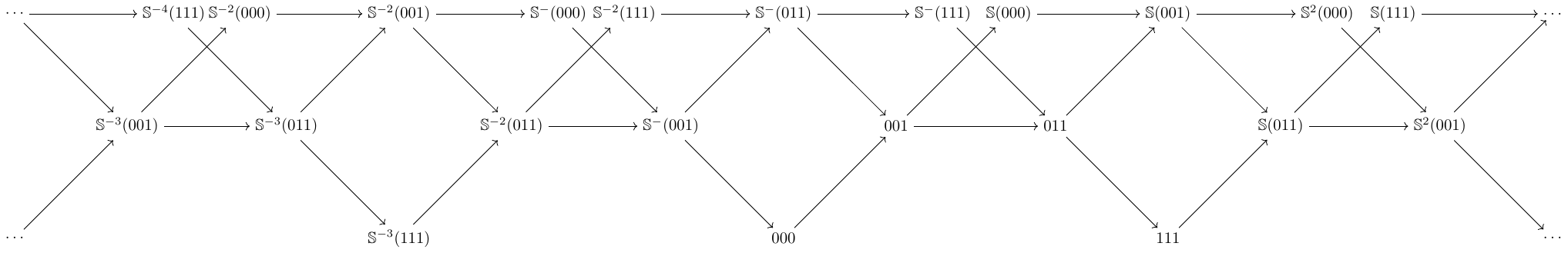}
    \end{center}
    \caption{The isomorphisms $A_{\mathbb{Z}4}^{(2)}\cong\widehat{A_3^{(2)}}$
      (top) and $A_{\mathbb{Z}3}^{(3)}\cong\widehat{A_2^{(3)}}$ (bottom), see
      \th\ref{prop:A-infty:lZ:equivalences}.}
    \label{fig:repetitive}
  \end{figure}
  
  It is now straightforward to verify that the quiver with relations
  $\widehat{A_{\ell-1}^{(d+1)}}$ obtained by applying the main theorem in
  \cite{Sch99} to the finite dimensional algebra $A_{\ell-1}^{(d+1)}$ agrees
  with the above description of $A_{\mathbb{Z}\ell}^{(d+1)}$ (see Figure
  \ref{fig:repetitive} for a couple of examples). We leave the remaining details
  to the reader.
\end{proof}

\begin{remark}
  If $d=1$ and $\ell\ge2$ is arbitrary, then the equivalence
  \[
    \add A_{\mathbb{Z}\ell}^{(2)}\cong\mathcal{U}(A_\ell^{(1)})=\Db(\mmod
    A_{\ell}^{(1)})
  \]
  in \th\ref{prop:A-infty:lZ:equivalences} is a special case of a result of
  Happel, see Proposition 5.6 in \cite{Hap88}.
\end{remark}

The next theorem is essentially a consequence of
\th\ref{prop:A-infty:lZ:equivalences} together with a special case of Corollary
4.10 in \cite{IO13}. It settles \th\ref{thm:Ail} in the case
$\underline{\ell}=\mathbb{Z}\ell$.

\begin{theorem}
  \th\label{thm:A-infty:lZ} The following statements hold.
  \begin{enumerate}
  \item\label{thm:A-infty:lZ:it:selfinjective} The locally bounded category
    $A_{\mathbb{Z}\ell}^{(d)}$ is selfinjective. Moreover,
    \[
      \proj A_{\mathbb{Z}\ell}^{(d)}=\add\setP{M(\bm{\lambda})\in\mmod
        A_{\mathbb{Z}\ell}^{(d)}}{\bm{\lambda}\in\mathbf{os}_{\mathbb{Z}\ell}^{d+1}\text{
          and }\len(\bm{\lambda})=\ell}.
    \]
  \item\label{thm:A-infty:lZ:it:nakayama} Let
    $\bm{\lambda}\in\mathbf{os}_{\mathbb{Z}\ell}^{d+1}$ be such that
    $\len(\bm{\lambda})=\ell$ and $M(\bm{\lambda})$ the associated
    indecomposable projective $A_{\mathbb{Z}\ell}^{(d)}$-module. Then, there is
    an isomorphism
    \[
      \nu(M(\bm{\lambda}))\cong M(\lambda_2,\dots,\lambda_{d+1},\lambda_2+\ell-1).
    \]
  \item\label{thm:A-infty:lZ:it:taud} Let
    $\bm{\lambda}\in\mathbf{os}_{\mathbb{Z}\ell}^{d+1}$ be such that
    $\len(\bm{\lambda})<\ell$. Then, there are isomorphisms of
    $A_{\mathbb{Z}\ell}^{(d)}$-modules
    \[
      \tau_d(M(\bm{\lambda}))\cong
      M(\tau_d(\bm{\lambda}))\quad\text{and}\quad\tau_d^-(M(\bm{\lambda}))\cong
      M(\tau_d^-(\bm{\lambda})).
    \]
  \item\label{thm:A-infty:lZ:it:equivalence} There is an equivalence of
    triangulated categories
    \[
      \underline{\mmod}\,A_{\mathbb{Z}\ell}^{(d)}\xrightarrow{\sim}\Db(\mmod
      A_{\ell-1}^{(d)})
    \]
    which restricts to an equivalence between the subcategories
    \[
      \underline{\mathcal{M}}_{\mathbb{Z}\ell}^{(d)}\xrightarrow{\sim}\mathcal{U}(A_{\ell-1}^{(d)}).
    \]
    In particular, $\mathcal{M}_{\mathbb{Z}\ell}^{(d)}$ is a
    $d\mathbb{Z}$-cluster-tilting subcategory of $\mmod
    A_{\mathbb{Z}\ell}^{(d)}$.
  \item\label{thm:A-infty:lZ:it:taud2} For every simple
    $A_{\mathbb{Z}\ell}^{(d)}$-module $S\in\mathcal{M}_{\mathbb{Z}\ell}^{(d)}$
    the $A_{\mathbb{Z}\ell}^{(d)}$-module $\tau_d(S)$ is simple. Moreover, for
    every pair of integers $i$ and $j$ there is an isomorphism
    $\tau_d^i(S)\cong\tau_d^j(S)$ if and only if $i=j$.
  \end{enumerate}
\end{theorem}
\begin{proof}
  \eqref{thm:A-infty:lZ:it:selfinjective} By
  \th\ref{prop:A-infty:lZ:equivalences} there is an equivalence of categories
  \[
    A_{\mathbb{Z}\ell}^{(d)}\cong\widehat{A_{\ell-1}^{(d)}}.
  \]
  It is a general fact that the repetitive category of a finite dimensional
  algebra is a locally bounded selfinjective category, see Lemma 2.2 in
  \cite{Hap88}. This proves the first claim.
  
  Now we prove the second claim, concerning the classification of the
  projective(-injective) $A_{\mathbb{Z}\ell}^{(d)}$-modules. Let
  $\bm{\lambda}\in\mathbf{os}_{\mathbb{Z}\ell}^d$. We claim that the
  indecomposable projective $A_{\mathbb{Z}\ell}^{(d)}$-module at the vertex
  $\bm{\lambda}$ is precisely
  \[
    M(\lambda_d-\ell+1,\lambda_1,\dots,\lambda_d),
  \]
  which clearly has Loewy length $\ell$. First, observe that
  \[
    \lambda_d-\ell+1\leq\lambda_1
  \]
  if and only if
  \[
    \lambda_1-(\lambda_d-\ell+1)=\ell-\len(\bm{\lambda})\geq0,
  \]
  which holds by assumption. Thus
  $M(\lambda_d-\ell+1,\lambda_1,\dots,\lambda_d)$ is a well defined
  $A_\infty^{(d)}$-module. In order to show that it is moreover an
  $A_{\mathbb{Z}\ell}^{(d)}$-module we need to prove that each
  $\bm{\mu}\in\mathbf{os}^d$ such that
  \[
    \lambda_d-\ell+1\leq\mu_1\leq\lambda_1\leq\cdots\leq\mu_d\leq\lambda_d
  \]
  also satisfies $\len(\bm{\mu})\leq\ell$. Indeed, given such $\bm{\mu}$ we have
  \[
    \len(\bm{\mu})=\mu_d-\mu_1+1\leq\lambda_d-(\lambda_d-\ell+1)+1=\ell.
  \]
  To prove that $M(\bm{\lambda})$ is projective as an
  $A_{\mathbb{Z}\ell}^{(d)}$-module it suffices to show that each
  $\bm{\mu}\in\mathbf{os}_{\mathbb{Z}\ell}^d$ such that
  $\bm{\mu}\rightsquigarrow\bm{\lambda}$ and
  $[\bm{\mu},\bm{\lambda}]\subset\mathbf{os}_{\mathbb{Z}\ell}^d$ also satisfies
  \[
    \lambda_d-\ell+1\leq\mu_1\leq\lambda_1\leq\cdots\leq\mu_d\leq\lambda_d.
  \]
  Such $\bm{\mu}$ by definition satisfies
  \[
    \mu_1\leq\lambda_1\leq\cdots\leq\mu_d\leq\lambda_d.
  \]
  Suppose that $\lambda_d-\ell+1>\mu_1$ so that the tuple
  \[
    \bm{\rho}:=(\lambda_d-\ell,\lambda_2,\dots,\lambda_d)\in\mathbf{os}^d
  \]
  satisfies $\bm{\rho}\in[\bm{\mu},\bm{\lambda}]$. Then, the strict inequality
  \[
    \len(\bm{\rho})=\lambda_d-(\lambda_d-\ell)+1=\ell+1>\ell
  \]
  contradicts the assumption
  $[\bm{\mu},\bm{\lambda}]\subset\mathbf{os}_{\mathbb{Z}\ell}^d$. Therefore
  $\lambda_d-\ell+1\leq\mu_1$, as required.

  \eqref{thm:A-infty:lZ:it:nakayama} Let
  $\bm{\lambda}\in\mathbf{os}_{\mathbb{Z}\ell}^{d+1}$ be such that
  $\len(\bm{\lambda})=\ell$ and $M(\bm{\lambda})$ the associated projective
  $A_{\mathbb{Z}\ell}^{(d)}$-module. Then $\nu M(\bm{\lambda})$ is the
  indecomposable injective $A_{\mathbb{Z}\ell}^{(d)}$-module whose simple socle
  is the simple top $S(\lambda_2,\dots,\lambda_{d+1})$ of $M(\bm{\lambda})$.
  Thus, in view of the proof of \eqref{thm:A-infty:lZ:it:selfinjective}, we obtain the
  required isomorphism
        \[
      \nu(M(\bm{\lambda}))\cong M(\lambda_2,\dots,\lambda_{d+1},\lambda_2+\ell-1).
    \]
  
    \eqref{thm:A-infty:lZ:it:taud} Let
    $\bm{\lambda}\in\mathbf{os}_{\mathbb{Z}\ell}^{d+1}$ be such that
    $\len(\bm{\lambda})<\ell$. A direct calculation (which we leave to the
    reader) shows that the morphism $P^d\to P^{d-1}$ between the terms of degree
    $d$ and $d-1$ in the minimal projective resolution of $M(\bm{\lambda})$ is
    given by the (essentially unique) non-zero morphism
    \[
      \begin{tikzcd}
        M(\lambda_d-\ell,\lambda_1-1,\dots,\lambda_{d-1}-1,\lambda_d-1)\rar&
        M(\lambda_{d+1}-\ell+1,\lambda_1-1,\dots,\lambda_{d-1}-1,\lambda_{d+1});
      \end{tikzcd}
    \]
    note that these modules are well defined since $\len(\bm{\lambda})<\ell$ and
    that they are indeed projective by \eqref{thm:A-infty:lZ:it:selfinjective}.
    Since $P^d\to P^{d-1}$ is the minimal projective presentation of
    $\Omega^{d-1}(M(\bm{\lambda}))$, by definition $\tau_d M(\bm{\lambda})$ is
    the kernel of the morphism $\nu P^d\to \nu P^{d-1}$ which, in view of
    \eqref{thm:A-infty:lZ:it:nakayama}, is readily seen to be isomorphic to the
    interval $A_{\mathbb{Z}\ell}^{(d)}$-module
    \[
      M[(\lambda_1-1,\dots,\lambda_d-1),(\lambda_2-1,\dots,\lambda_{d+1}-1)]
    \]
    which is is precisely $M(\tau_d(\bm{\lambda}))$.
  
  \eqref{thm:A-infty:lZ:it:equivalence} We follow the proof of Corollary 4.10 in
  \cite{IO13}. The general theory of repetitive categories shows that there is
  an equivalence of triangulated categories
  \[
    \underline{\mmod}\,A_{\mathbb{Z}\ell}^{(d)}=\underline{\mmod}\,\widehat{A_{\ell-1}^{(d)}}\xrightarrow{\sim}\Db(\mmod
    A_{\ell-1}^{(d)})
  \]
  induced by the tilting object
  \[
    T:=\bigoplus\setP{{M(0,\lambda_1,\dots,\lambda_d)}}{\bm{\lambda}\in\mathbf{os}_{\ell-1}^d}\in\underline{\mmod}\,
    A_{\mathbb{Z}\ell}^{(d)},
  \]
  where we use the explicit isomorphism
  \[
    A_{\mathbb{Z}\ell}^{(d)}\cong\widehat{A_{\ell-1}^{(d)}}
  \]
  given in the proof of \th\ref{prop:A-infty:lZ:equivalences} to identify
  $A_{\ell-1}^{(d)}$ with $T$, see for example Section II.4 in \cite{Hap88}.
  Given that the canonical functor $\mmod
  A_{\mathbb{Z}\ell}^{(d)}\to\underline{\mmod}A_{\mathbb{Z}\ell}^{(d)}$ induces
  a bijective correspondence between $d$-cluster-tilting subcategories of $\mmod
  A_{\mathbb{Z}\ell}^{(d)}$ and those of
  $\underline{\mmod}A_{\mathbb{Z}\ell}^{(d)}$, it is enough to show that
  \[
    \underline{\mathcal{M}}_{\mathbb{Z}\ell}^{(d)}=\add\setP{\mathbb{S}_d^i(T)\in\underline{\mmod}\,A_{\mathbb{Z}\ell}^{(d)}}{i\in\mathbb{Z}}
  \]
  where $\mathbb{S}_d=\mathbb{S}\circ\Omega^{d}$ and $\mathbb{S}$ is the Serre
  functor on $\underline{\mmod}\,A_{\mathbb{Z}\ell}^{(d)}$, see
  \th\ref{not:d-RF_d-H:D}, \th\ref{thm:d-RF-d-H:D} and recall that $\Omega$ is
  the inverse of the suspension functor in $\underline{\mmod}\,A_{\mathbb{Z}\ell}^{(d)}$.
  Therefore, in view of statement \eqref{thm:A-infty:lZ:it:taud}, it is enough
  to prove that there are isomorphisms of functors $\mathbb{S}_d\cong\tau_d$.
  But this is clear since $\mathbb{S}=\tau\circ\Omega^-$ implies
  \[
    \mathbb{S}_d=\mathbb{S}\circ\Omega^d\cong(\tau\circ\Omega^-)\circ\Omega^d\cong
    \tau\circ\Omega^{d-1}=\tau_d.
  \]
  This proves the claim.
  
  \eqref{thm:A-infty:lZ:it:taud2} The claim follows from the explicit formula
  for the higher Auslander--Reiten translate given in statement
  \eqref{thm:A-infty:lZ:it:taud}.
\end{proof}

We conclude this subsection with a sketch of the proof of \th\ref{thm:Ail}. The
argumentation is essentially the same as in the proof of \th\ref{thm:An}.

\begin{proof}[Sketch of proof of \th\ref{thm:Ail}]
  Let $d$ be a positive integer, $\ell\geq2$, and $\underline{\ell}$ an
  $\ell$-bounded Kupisch series of type $\mathbb{A}_\infty^\infty$. The proof of
  \th\ref{thm:Ail} is obtained using \th\ref{lemma:d-CT:idempotent_reduction},
  by constructing $A_{\underline{\ell}}^{(d)}$ as an iterated idempotent
  quotient of $A_{\mathbb{Z}\ell}^{(d)}$ for which the claim is known by
  \th\ref{thm:A-infty:lZ}. We leave the details to the reader.
\end{proof}

\subsection{Proof of {\th\ref{thm:mesh}}}
\label{subsec:proof:thm:mesh}

We are now ready to prove \th\ref{thm:mesh}.

\begin{proof}[Proof of \th\ref{thm:mesh}]
  \eqref{thm:mesh:it:gldimd} Let $i>d$, $M$ and $N$ two
  $A_\infty^{(d)}$-modules, and
  \[
    \delta\colon 0\to N\to L^1\to\cdots\to L^i\to M\to0
  \]
  an exact sequence in $\mmod A_\infty^{(d)}$. We need to prove that this
  sequence is trivial in the sense of Yoneda. For this, choose integers $a\leq
  b$ such that $\delta$ is an exact sequence of $A_{[a,b]}^{(d)}$-modules (this
  is possible since $\delta$ involves only finitely many finite dimensional
  $A_\infty^{(d)}$-modules). Since $A_{[a,b]}^{(d)}\cong A_n^{(d)}$ has global
  dimension at most $d$ (see \th\ref{thm:An}), the exact sequence $\delta$ is
  trivial as an exact sequence in $\mmod A_{[a,b]}^{(d)}$ whence also as an
  exact sequence in $\mmod A_\infty^{(d)}$. This proves the claim.

  \eqref{thm:mesh:it:d-CT} We divide the proof into three steps.

  \emph{Step 1: $\mathcal{M}$ is a generating-cogenerating functorially finite
    subcategory of $\mmod A_\infty^{(d)}$.} Let $N$ be a finite dimensional
  $A_\infty^{(d)}$-module and $\ell\geq2$ an integer such that $N\in\mmod
  A_{\mathbb{Z}\ell}^{(d)}$. Since $\mathcal{M}_{\mathbb{Z}\ell}^{(d)}$ is a
  generating(-cogenerating) functorially finite subcategory of $\mmod
  A_{\mathbb{Z}\ell}^{(d)}$, there exists a surjective right
  $\mathcal{M}_{\mathbb{Z}\ell}^{(d)}$-approximation $f\colon M\to N$ of $N$. We
  claim that $f$ is a right $\mathcal{M}_\infty^{(d)}$-approximation of $N$.
  Indeed, let $\bm{\lambda}\in\mathbf{os}^{d+1}$ and $g\colon M(\bm{\lambda})\to
  N$ a morphism of $A_\infty^{(d)}$-modules. If
  $\bm{\lambda}\in\mathbf{os}_{\mathbb{Z}\ell}^{d+1}$, then we already know that
  $g$ factors through $f$. Suppose that
  $\bm{\lambda}\not\in\mathbf{os}_{\mathbb{Z}\ell}^{d+1}$, that is
  $\ell':=\len(\bm{\lambda})>\ell$. According to
  \th\ref{thm:A-infty:lZ}\eqref{thm:A-infty:lZ:it:selfinjective},
  $M(\bm{\lambda})$ is a projective-injective
  $A_{\mathbb{Z}\ell'}^{(d)}$-module, whence $g$ factors through $f$, as the
  latter is an epimorphism between $A_{\mathbb{Z}\ell'}^{(d)}$-modules. This
  proves that $f$ is a right $\mathcal{M}_\infty^{(d)}$-approximation of $N$.
  One can prove the existence of injective left
  $\mathcal{M}_\infty^{(d)}$-approximations dually.

  \emph{Step 2: $\mathcal{M}_\infty^{(d)}$ is a $d$-rigid subcategory of $\mmod
    A_{\infty}^{(d)}$.} Let $\bm{\lambda},\bm{\mu}\in\mathbf{os}^{d+1}$. We need
  to prove that $\Ext_{A_\infty^{(d)}}^i(M(\bm{\lambda}),M(\bm{\mu}))$ vanishes
  for all $i\in\set{1,\dots,d-1}$. For this, let $i\in\set{1,\dots,d-1}$ and
  \[
    \delta\colon0\to M(\bm{\mu})\to L^1\to\cdots\to L^i\to M(\bm{\lambda})\to0
  \]
  be an exact sequence in $\mmod A_\infty^{(d)}$. Let $a\leq b$ be integers such
  that all of the modules involved in $\delta$ are $A_{[a,b]}^{(d)}$-modules.
  Since $\mathcal{M}_{[a,b]}^{(d)}$ is a $d$-rigid subcategory of $\mmod
  A_{[a,b]}^{(d)}$, the above sequence is trivial in
  $\Ext_{A_{[a,b]}^{(d)}}^i(M(\bm{\lambda}),M(\bm{\mu}))$ in the sense of Yoneda
  whence it is also trivial as an element of
  $\Ext_{A_\infty^{(d)}}^i(M(\bm{\lambda}),M(\bm{\mu}))$, as required.
  
  \emph{Step 3: $\mathcal{M}_\infty^{(d)}$ is a $d$-cluster-tilting subcategory
    of $\mmod A_\infty^{(d)}$.} Let $N$ be an $A_\infty^{(d)}$-module such that
  for each $i\in\set{1,\dots,d-1}$ there is an isomorphism
  \[
    \Ext_{A_\infty^{(d)}}^i(\mathcal{M}_\infty^{(d)},N)\cong0.
  \]
  We need to prove that $N\in\mathcal{M}_\infty^{(d)}$. Let $a\leq b$ be
  integers such that $N$ is an $A_{[a,b]}^{(d)}$-module. Since
  $\mathcal{M}_{[a,b]}^{(d)}$ is a $d$-cluster-tilting subcategory of $\mmod
  A_{[a,b]}^{(d)}$, it is enough to note that, if
  $\bm{\lambda}\in\mathbf{os}_{[a,b]}^{d+1}$, then for each
  $i\in\set{1,\dots,d-1}$ there is an isomorphism
  \[
    \Ext_{A_{[a,b]}^{(d)}}^i(M(\bm{\lambda}),N)\cong\Ext_{A_\infty^{(d)}}^i(M(\bm{\lambda}),N)=0,
  \]
  (see \th\ref{prop:A-infty:Ext}). This shows that
  $N\in\mathcal{M}_{[a,b]}\subset\mathcal{M}_\infty^{(d)}$. Dually, one can show
  that if for each $i\in\set{1,\dots,d-1}$ there is an isomorphism
  \[
    \Ext_{A_\infty^{(d)}}^i(N,\mathcal{M}_\infty^{(d)})\cong0,
  \]
  then $N\in\mathcal{M}_\infty^{(d)}$.
  
  \eqref{thm:mesh:it:d-almost-split}, \eqref{thm:mesh:it:taud:isos}, and
  \eqref{thm:mesh:it:taud} These claims are all consequences of the
  corresponding statements in \th\ref{thm:A-infty:lZ}. Indeed, for
  $\bm{\lambda}\in\mathbf{os}^{d+1}$ let $\ell:=\len(\bm{\lambda})+1$ so that, by
  \th\ref{thm:A-infty:lZ}\eqref{thm:A-infty:lZ:it:selfinjective}, the module
  $M(\bm{\lambda})$ is neither projective nor injective as an
  $A_{\mathbb{Z}\ell}^{(d)}$-module. We claim that the $d$-almost split sequence
  \[
    \delta\colon0\to \tau_d(M(\bm{\lambda}))\to L^1\to\cdots\to
    L^d\xrightarrow{f} M(\bm{\lambda})\to0
  \]
  in $\mathcal{M}_{\mathbb{Z}\ell}^{(d)}$ is still $d$-almost split in
  $\mathcal{M}_\infty^{(d)}$. It is enough to prove that $f$ is right almost
  split in $\mathcal{M}_\infty^{(d)}$. Let $\bm{\mu}\in\mathbf{os}^{d+1}$ and
  $g\colon M(\bm{\mu})\to M(\bm{\lambda})$ a non-zero morphism which is not a
  retraction. If $\len(\bm{\mu})\leq\ell$, then we already know that $g$ factors
  through $f$. If $\len(\bm{\mu})>\ell$, then
  \th\ref{thm:A-infty:lZ}\eqref{thm:A-infty:lZ:it:selfinjective} shows that
  $M(\bm{\mu})$ is projective as $A_{\mathbb{Z}\ell'}^{(d)}$-module where
  $\ell':=\len(\bm{\mu})$. Therefore $g$ factors through $f$, since the latter
  is an epimorphism between $A_{\mathbb{Z}\ell'}^{(d)}$-modules. This shows that
  $\delta$ is a $d$-almost split sequence in $\mathcal{M}_\infty^{(d)}$ and also
  that
  \[
    \tau_d(M(\bm{\lambda}))=M(\tau_d(\bm{\lambda}))
  \]
  is the $d$-Auslander--Reiten translate of $M(\bm{\lambda})$ in
  $\mathcal{M}_\infty^{(d)}$. The remaining statements follow immediately from
  this identity. This finishes the proof of the theorem.
\end{proof}


\section{The higher Nakayama algebras of type $\widetilde{\mathbb{A}}$}
\label{sec:A-tilde}

Let $d$ be a positive integer. In this section we construct the $d$-Nakayama
algebras of type $\widetilde{\mathbb{A}}$. We show that they belong to the class
of $d\mathbb{Z}$-representation-finite algebras and establish their basic
properties.

\subsection{The selfinjective higher Nakayama algebras}

We begin by constructing certain selfinjective
$d\mathbb{Z}$-representation-finite algebras which should be thought of as
higher dimensional analogues of the selfinjective Nakayama algebras. Our
construction uses results of Darp\"o and Iyama \cite{DI17}.

We begin by recalling the definition of the orbit category of a category with a
free $\mathbb{Z}$-action, see for example page 1136 in \cite{MOS09}.

\begin{definition}
  \th\label{def:orbit_category} Let $\mathcal{A}$ be a small category and
  $\varphi\colon\mathcal{A}\to\mathcal{A}$ an automorphism acting freely on the
  set of objects of $\mathcal{A}$ (we say that $\mathcal{A}$ \emph{has a free
    $\mathbb{Z}$-action}). Given an object $x\in\mathcal{A}$ we denote its
  $\varphi$-orbit by
  \[
    \varphi^\mathbb{Z}(x):=\setP{\varphi^i(x)\in\mathcal{A}}{i\in\mathbb{Z}}.
  \]
  The \emph{orbit category $\mathcal{A}/\varphi$} has objects the set of
  $\varphi$-orbits of objects in $\mathcal{A}$. The vector space of morphisms
  $\varphi^{\mathbb{Z}}(x)\to\varphi^{\mathbb{Z}}(y)$ is defined by the short
  exact sequence
  \[
    [ f-\varphi^k(f)\ |\
    k\in\mathbb{Z}]\hookrightarrow\bigoplus_{i,j\in\mathbb{Z}}\mathcal{A}(\varphi^{i}(x),\varphi^{j}(y))\twoheadrightarrow(\mathcal{A}/\varphi)(\varphi^{\mathbb{Z}}(x),\varphi^{\mathbb{Z}}(y)).
  \]
\end{definition}

\begin{remark}
  \th\label{rmk:orbit_category} Let $\mathcal{A}$ be a category and
  $\varphi\colon\mathcal{A}\to\mathcal{A}$ an automorphism acting freely on the
  set of objects of $\mathcal{A}$. Given $x,y\in\mathcal{A}$, there is a
  canonical isomorphism of vector spaces
  \[
    (\mathcal{A}/\varphi)(\varphi^{\mathbb{Z}}(x),\varphi^{\mathbb{Z}}(y))\cong\bigoplus_{i\in\mathbb{Z}}\mathcal{A}(x,\varphi^i(y)).
  \]
\end{remark}

\begin{example}
  Consider the category $A_\infty^{(1)}$, that is the path category of the
  infinite quiver
  \[
    \cdots\to-1\to0\to1\to\cdots.
  \]
  The quiver automorphism $\tau_0(i\to j)= (i-1)\to(j-1)$ extends to an
  automorphism of $A_\infty^{(1)}$ acting freely on the set of objects of
  $A_\infty^{(1)}$. Therefore for every positive integer $n$ we can form the
  orbit category
  \[
    \widetilde{A}_{n-1}^{(1)}:=A_\infty^{(1)}/\tau_0^n
  \]
  which can be identified with the path algebra of the
  $\widetilde{\mathbb{A}}_{n-1}$ quiver
  \begin{center}
    \begin{tikzpicture}
      \node (0) at (180:1) {$0$}; \node (1) at (90:1) {$1$}; \node (2) at (0:1)
      {$\vdots$}; \node (n1) at (270:1) {$n-1$};
      \path[commutative diagrams/.cd, every arrow] (0) edge[bend left=20] (1)
      (1) edge[bend left=20] (2) (2) edge[bend left=20] (n1) (n1) edge[bend
      left=20] (0);
    \end{tikzpicture}
  \end{center}
  By definition, the basic Nakayama algebras of type
  $\widetilde{\mathbb{A}}_{n-1}$ are the admissible quotients of
  $\widetilde{A}_{n-1}^{(1)}$. Note that such admissible quotients correspond to those
  admissible quotients of $A_\infty^{(1)}$ which are stable under the action of
  $\tau_0^n$. These observations are the motivation for our construction of the
  higher Nakayama algebras of type $\widetilde{\mathbb{A}}$.
\end{example}

\begin{notation}
  Let $\mathcal{A}$ be a locally bounded category and
  $\varphi\colon\mathcal{A}\to\mathcal{A}$ an automorphism acting freely on the
  set of objects of $\mathcal{A}$. The canonical functor
  $F\colon\mathcal{A}\to\mathcal{A}/\varphi$ induces an exact functor
  \[
    F^*\colon\mmod(\mathcal{A}/\varphi)\to\mmod\mathcal{A}
  \]
  which admits an exact left adjoint
  \[
    F_*\colon\mmod\mathcal{A}\to\mmod(\mathcal{A}/\varphi),
  \]
  see for example Section 14.3 in \cite{GR97}.
\end{notation}

We introduce the following modification of \th\ref{def:KS:An}.

\begin{definition}
  \th\label{def:KS:A-tilde} Let $n$ be a positive integer and
  $\underline{\ell}=(\ell_0,\ell_1,\dots,\ell_{n-1})$ be a tuple of positive
  integer numbers. We say that $\underline{\ell}$ is a \emph{(connected) Kupisch
    series of type $\widetilde{\mathbb{A}}_{n-1}$} if for all
  $i\in\mathbb{Z}/n\mathbb{Z}$ there are inequalities $2\leq\ell_{i}\leq
  \ell_{i-1}+1$. Two Kupisch series $\underline{\ell}$ and $\underline{\ell}'$
  of type $\widetilde{\mathbb{A}}_{n-1}$ are \emph{equivalent} if there exists a
  cyclic permutation $\sigma\in\mathfrak{S}_n$ which is either trivial or does
  not have fixed points such that
  \[
    \underline{\ell}'=(\ell_{\sigma(0)},\ell_{\sigma(1)},\cdots,\ell_{\sigma(n-1)}).
  \]
  We denote the set of equivalence classes of Kupisch series of type
  $\widetilde{\mathbb{A}}_{n-1}$ by $\KSAt{n-1}$.
\end{definition}

\begin{remark}
  Suppose that the ground field is algebraically closed. It is well known that
  there is a bijective correspondence between Morita equivalence classes of
  (connected) Nakayama algebras of type $\widetilde{\mathbb{A}}_{n-1}$ and
  equivalence classes of Kupisch series of type $\widetilde{\mathbb{A}}_{n-1}$.
  In one direction the correspondence is given by associating the tuple
  \[
    \tuple{\len(e_0A),\len(e_1A),\dots,\len(e_{n-1}A)}
  \]
  to a Nakayama algebra $A$ with Gabriel quiver $\widetilde{\mathbb{A}}_{n-1}$.
  Moreover, a non-semisimple (connected) Nakayama algebra of type
  $\widetilde{\mathbb{A}}_{n-1}$ is selfinjective if and only if it has Kupisch
  series $(\ell,\dots,\ell)$ for some integer $\ell\geq2$, see for example
  Proposition V.3.8 in \cite{ASS06}.
\end{remark}

\begin{notation}
  Let $\underline{\ell}$ be a Kupisch series of type
  $\widetilde{\mathbb{A}}_{n-1}$. We denote the basic Nakayama algebra with
  Kupisch series $\underline{\ell}$ by $\widetilde{A}_{\underline{\ell}}^{(1)}$.
  By definition, $\widetilde{A}_{\underline{\ell}}^{(1)}$ is an admissible
  quotient of $\widetilde{A}_{n-1}^{(1)}$.
\end{notation}

\begin{setting}
  We fix positive integers $d,n$ and $\ell\geq2$ until further notice.
\end{setting}

We are ready give the definition of the selfinjective higher Nakayama algebras.

\begin{definition}
  \th\label{def:A-tilde:selfinjective} The \emph{selfinjective $(d+1)$-Nakayama
    algebra with Kupisch series $(\ell,\dots,\ell)$} is the orbit
  category\nomenclature[40]{$\widetilde{A}_{n-1,\ell}^{(d+1)}$}{the
    selfinjective $(d+1)$-Nakayama algebra with Kupisch series
    $(\ell,\dots,\ell)$}
  \[
    \widetilde{A}_{n-1,\ell}^{(d+1)}:=A_{\mathbb{Z}\ell}^{(d+1)}/\tau_d^n.
  \]
  We also define the subcategory
  \[
    \widetilde{\mathcal{M}}_{n-1,\ell}^{(d)}:=F_*(\mathcal{M}_{\mathbb{Z}\ell}^{(d)})\subseteq\mmod\widetilde{A}_{n-1,\ell}^{(d)}.
  \]
\end{definition}

The following theorem is the starting point of our construction of the higher
Nakayama algebras of type $\widetilde{\mathbb{A}}_{n-1}$ which we give in the
subsequent section. For the most part, it is a direct application of Theorem 2.3
in \cite{DI17}.

\begin{theorem}
  \th\label{thm:A-tilde:selfinjective} The following statements hold.
  \begin{enumerate}
  \item\label{thm:A-tilde:selfinjective:it:selfinjective} The algebra
    $\widetilde{A}_{n-1,\ell}^{(d)}$ is finite dimensional and selfinjective.
  \item\label{thm:A-tilde:selfinjective:it:LL} Every indecomposable
    projective-injective $\widetilde{A}_{n-1,\ell}^{(d)}$-module has Loewy
    length $\ell$.
  \item\label{thm:A-tilde:selfinjective:it:addM} There exists an
    $\widetilde{A}_{n-1,\ell}^{(d)}$-module $\widetilde{M}_{n-1,\ell}^{(d)}$
    such that
    \[
      \widetilde{\mathcal{M}}_{n-1,\ell}^{(d)}=\add\widetilde{M}_{n-1,\ell}^{(d)}.
    \]
  \item\label{thm:A-tilde:selfinjective:it:d-CT}
    $\widetilde{\mathcal{M}}_{n-1,\ell}^{(d)}$ is a
    $d\mathbb{Z}$-cluster-tilting subcategory of
    $\mmod\widetilde{A}_{n-1,\ell}^{(d)}$.
  \item\label{thm:A-tilde:selfinjective:it:tau_d} For every indecomposable
    non-projective $\widetilde{A}_{n-1,\ell}^{(d)}$-module
    $M\in\widetilde{\mathcal{M}}_{n-1,\ell}^{(d)}$ and every $i,j\in\mathbb{Z}$
    there is an isomorphism $\tau_d^i(M)\cong\tau_d^j(M)$ if and only if $i-j\in
    n\mathbb{Z}$.
  \item\label{thm:A-tilde:selfinjective:it:tau_d-simples} For every simple
    $\widetilde{A}_{n-1,\ell}^{(d)}$-module
    $S\in\widetilde{\mathcal{M}}_{n-1,\ell}^{(d)}$ and for each integer $i$ the
    $\widetilde{A}_{n-1,\ell}^{(d)}$-module $\tau_d^i(S)$ is simple.
  \end{enumerate}
\end{theorem}
\begin{proof}
  All claims are well known in the case $d=1$, hence we may assume that
  $d\geq2$. We note that \th\ref{prop:A-infty:lZ:equivalences,thm:A-infty:lZ}
  show that we are in the setting of Theorem 2.3 in \cite{DI17} and therefore
  statements \eqref{thm:A-tilde:selfinjective:it:selfinjective},
  \eqref{thm:A-tilde:selfinjective:it:addM}, and
  \eqref{thm:A-tilde:selfinjective:it:d-CT} hold. Note also that Lemma 3.9(b) in
  \cite{DI17} implies that the diagram
  \[
    \begin{tikzcd}
      \underline{\mathcal{M}}_{\mathbb{Z}\ell}^{(d)}\rar{F_*}\dar{\tau_d}&\underline{\widetilde{\mathcal{M}}}_{n-1,\ell}^{(d)}\dar{\tau_d}\\
      \underline{\mathcal{M}}_{\mathbb{Z}\ell}^{(d)}\rar{F_*}&\underline{\widetilde{\mathcal{M}}}_{n-1,\ell}^{(d)}
    \end{tikzcd}
  \]
  commutes. The remaining statements follow immediately from the explicit
  description of the indecomposable projective-injective
  $A_{\mathbb{Z}\ell}^{(d)}$-modules and the formula for the
  $d$-Auslander--Reiten translation given in \th\ref{thm:A-infty:lZ}, taking
  into account the existence of the above commutative diagram.
\end{proof}

\subsection{The higher Nakayama algebras of type $\widetilde{\mathbb{A}}$}

We are ready to give the construction of the higher Nakayama algebras of type
$\widetilde{\mathbb{A}}$. As their name suggests, these are to be thought of as
higher dimensional analogues of the admissible quotients of basic selfinjective
Nakayama algebras.

\begin{definition}
  \th\label{def:A-tilde} Let $d$ and $n$ be positive integers,
  $\underline{\ell}$ a Kupisch series of type $\widetilde{\mathbb{A}}_{n-1}$,
  and $\ell=\max\underline{\ell}$. We make the following definitions.
  \begin{enumerate}
  \item The \emph{$\underline{\ell}$-restriction of
      $\mathbf{os}_{n-1,\ell}^{d+1}$} is the subset
    \begin{equation*}
      \mathbf{os}_{\underline{\ell}}^{d+1}:=\setP{\bm{\lambda}\in\mathbf{os}_{n-1,\ell}^{d+1}}{\len(\bm{\lambda})\leq\ell_{\lambda_{d+1}}}.
    \end{equation*}
  \item The \emph{$(d+1)$-Nakayama algebra with Kupisch series
      $\underline{\ell}$} is the finite dimensional algebra
    \[
      \widetilde{A}_{\underline{\ell}}^{(d+1)}:=A_{n-1,\ell}^{(d+1)}/[\mathbf{os}_n^{d+1}\setminus\mathbf{os}_{\underline{\ell}}^{d+1}].
    \]
  \item The $\widetilde{A}_{n-1,\ell}^{(d)}$-module $M_{\underline{\ell}}^{(d)}$
    is by definition
    \begin{equation*}
      M_{\underline{\ell}}^{(d)}:=\bigoplus_{\bm{\lambda}\in\mathbf{os}_{\underline{\ell}}^{d+1}}M(\bm{\lambda}).
    \end{equation*}
  \end{enumerate}
\end{definition}

The following theorem and its proof are analogous to \th\ref{thm:An:l}.

\begin{theorem}
  \th\label{thm:A-tilde} The following statements hold.
  \begin{enumerate}
  \item\label{it:thm:A-tilde:KS} For each $i\in\set{0,1,\dots,n-1}$ the
    indecomposable projective $\widetilde{A}_{\underline{\ell}}^{(d)}$-module at
    the vertex $(i,\dots,i)$ has Loewy length $\ell_i$.
  \item\label{it:thm:A-tilde:dZ-RF} $M_{\underline{\ell}}^{(d)}$ is a
    $d\mathbb{Z}$-cluster-tilting
    $\widetilde{A}_{\underline{\ell}}^{(d)}$-module. In particular,
    $\widetilde{A}_{\underline{\ell}}^{(d)}$ is a
    $d\mathbb{Z}$-representation-finite algebra.
  \item\label{it:thm:d-Nakayama-tilde:A:taud} For every simple
    $\widetilde{A}_{\underline{\ell}}^{(d)}$-module $S$ which is a direct
    summand of $M_{\underline{\ell}}^{(d)}$ the
    $\widetilde{A}_{\underline{\ell}}^{(d)}$-module $\tau_d(S)$ is simple.
  \end{enumerate}
\end{theorem}
\begin{proof}
  The proof is completely analogous to that of \th\ref{thm:An:l}. By
  \th\ref{thm:A-tilde:selfinjective} the claims are known for the finite
  dimensional algebra $\widetilde{A}_{n-1,\ell}^{(d)}$ were
  $\ell=\max\underline{\ell}$. In the general case, the theorem can be proven by
  iterated application of \th\ref{lemma:d-CT:idempotent_reduction:algebras}. To
  see how the lemma applies, note that $\underline{\ell}$ induces a
  $\ell$-bounded Kupisch series of type $\mathbb{A}_\infty^\infty$
  \[
    \underline{\ell}':=(\cdots,\ell_{n-1},\ell_0,\ell_1,\dots,\ell_{n-1},\ell_0,\dots).
  \]
  which is $n$-periodic in the obvious sense. Moreover, it is clear that there
  is a finite sequence of inequalities
  \[
    \underline{\ell}'=\underline{\ell}(0)\leq\cdots\leq\underline{\ell}(1)\leq\cdots\leq\underline{\ell}(t)=\mathbb{Z}\ell
  \]
  in $\KSAi$ in which every Kupisch series $\underline{\ell}(i)$ is also
  $n$-periodic. By the proof of \th\ref{thm:Ail} we know that
  \th\ref{lemma:d-CT:idempotent_reduction} can be applied inductively to the
  categories $\widetilde{A}_{\underline{\ell}(i)}^{(d)}$. It is now
  straightforward to verify that this implies that
  \th\ref{lemma:d-CT:idempotent_reduction:algebras} can be inductively applied
  to the sequence of finite dimensional algebras
  \[
    \widetilde{A}_{\underline{\ell}}^{(d)}=A_{\underline{\ell}(0)}^{(d)}/\tau_{d-1}^n,\
    A_{\underline{\ell}(1)}^{(d)}/\tau_{d-1}^n,\ \dots,\
    A_{\underline{\ell}(t)}^{(d)}/\tau_{d-1}^n=\widetilde{A}_{n-1,\ell}^{(d)}.
  \]
  We leave the details to the reader.
\end{proof}

\subsection{The higher dimensional analogues of the tubes}

We conclude this article by introducing the higher dimensional analogues of the
tubes. We need to introduce further terminology.

Let $\mathcal{A}$ be a small category. A finite dimensional $\mathcal{A}$-module
is \emph{nilpotent} if there exists a positive integer $n$ such that $M$ is an
$(\mathcal{A}/\mathcal{J}^n)$-module, where $\mathcal{J}$ denotes the Jacobson
radical of $\mathcal{A}$, see \cite{Kel64a} for the definition of the Jacobson
radical of a category. We denote the category of finite dimensional nilpotent
$\mathcal{A}$-modules by $\nil\mathcal{A}$. It is an abelian subcategory of
$\mmod\mathcal{A}$.
\begin{example}
  Let $n$ be a positive integer. It is well known that the abelian category
  $\nil\widetilde{A}_{n-1}^{(1)}$ has almost split sequences, see for example
  Lemma X.1.4(f) in \cite{SS07}. Moreover, for every indecomposable module
  $M\in\nil\widetilde{A}_{n-1}^{(1)}$ and for each $i,j\in\mathbb{Z}$ there is
  an isomorphism $\tau^i(M)\cong\tau^j(M)$ if and only if $i-j\in n\mathbb{Z}$.
  Thus, the Auslander--Reiten quiver of $\nil \widetilde{A}_{n-1}^{(1)}$ forms a tube of
  period $n$. This observation is the motivation for our construction of the
  higher dimensional analogues of the tubes.
\end{example}

\begin{setting}
  We fix positive integers $d$ and $n$.
\end{setting}

\begin{definition}
  \th\label{def:tubes} We define the \emph{$(d-1)$-tube of rank $n$} to be the
  category $\widetilde{A}_{n-1}^{(d)}:=A_\infty^{(d)}/\tau_{d-1}^n$. We also
  define the subcategory
  \begin{equation*}
    \mathcal{T}_n^{(d)}:=F_*(\mathcal{M}_\infty^{(d)})\subset\nil \widetilde{A}_{n-1}^{(d)}.
  \end{equation*}
\end{definition}

\begin{remark}
  Note that the $0$-tube of rank $n$ is just the path algebra of the circular
  quiver with $n$ vertices and $n$ arrows. Whence $\widetilde{A}_{n-1}^{(1)}$ is
  not locally finite dimensional.
\end{remark}

\begin{remark}
  Let $d\geq2$. Then, the category $\widetilde{A}_{n-1}^{(d)}$ is not locally
  bounded. Indeed, for each $\lambda\in\mathbf{os}^d$ and for each
  $i\in\mathbb{Z}$ the non-zero morphisms $\bm{\lambda}\to\bm{\lambda}+i
  e_d$ and $\bm{\lambda}-ie_1\to\bm{\lambda}$ in $A_\infty^{(d)}$ induce
  non-zero morphisms in $\widetilde{A}_{n-1}^{(d)}$ and for all $i\neq j$ we
  have
  \[
    \bm{\lambda}+i e_d\neq \bm{\lambda}+j
    e_d\qquad\text{and}\qquad\bm{\lambda}-ie_1\neq\bm{\lambda}-je_1
  \]
  in $\widetilde{A}_{n-1}^{(d)}$.
\end{remark}

\begin{theorem}
  \th\label{thm:tubes} Let $d$ and $n$ be positive integers. Then, the following
  statements hold.
  \begin{enumerate}
  \item The abelian category $\nil \widetilde{A}_{n-1}^{(d)}$ has global
    dimension $d$.
  \item\label{it:thm:tubes:d-CT} $\mathcal{T}_n^{(d)}$ is a $d$-cluster-tilting
    subcategory of $\nil\widetilde{A}_{n-1}^{(d)}$.
  \item\label{it:thm:tubes:taud} For every simple
    $\widetilde{A}_{n-1}^{(d)}$-module $S\in\mathcal{T}_n^{(d)}$ the
    $\widetilde{A}_{n-1}^{(d)}$-module $\tau_d(S)$ is simple.
  \item For every indecomposable $\widetilde{A}_{n-1}^{(d)}$-module
    $\mathcal{T}_n^{(d)}$ and every pair of integers $i$ and $j$ there is an
    isomorphism $\tau_d^i(M)\cong\tau_d^j(M)$ if and only if $i-j\in
    n\mathbb{Z}$.
  \end{enumerate}
\end{theorem}
\begin{proof}
  The proof is completely analogous to that of \th\ref{thm:mesh}. We leave the
  details to the reader.
\end{proof}



\appendix

\section{Relative $(d-1)$-homological embeddings\protect\footnote{by Julian K\"ulshammer and Chrysostomos Psaroudakis}}
\label{app:relative_homological_embeddings}

In this short appendix we prove that the canonical embedding $\mmod
\underline{\mathcal{A}}_{\mathcal{X}}\to\mmod \mathcal{A}$ in
\th\ref{lemma:d-CT:idempotent_reduction} is in fact a $(d-1)$-homological
embedding (and not only a relative one). We recall from Definition 3.6 in
\cite{Psa14} that an exact functor $i_*\colon\mathcal{B}\to\mathcal{A}$ is an
\emph{$m$-homological embedding} if for all $j\in\set{0,\dots,m}$ and for all
$X,Y\in\mathcal{B}$ the induced morphism
\[
  \Ext_\mathcal{B}^j(X,Y)\to\Ext_\mathcal{A}^j(i_*(X),i_*(Y))
\]
is an isomorphism.

\begin{proposition}
  \th\label{lemma:d-CT:idempotent_reduction:homological_embedding} Consider the
  recollement situation of abelian categories
  \[
    \begin{tikzcd}
      \mathcal{B}\arrow{r}[description]{i_*}
      &\mathcal{A}\arrow{r}[description]{j^*}\lar[bend right,
      swap]{i^*}\lar[bend left]{i^!} &\mathcal{C}\lar[bend right,
      swap]{j_!}\lar[bend left]{j_*}
    \end{tikzcd}
  \]
  where $\mathcal{A}$ and $\mathcal{C}$ have enough projectives. Then, the following statements
  hold.
  \begin{enumerate}
  \item\label{it:enough} $\mathcal{B}$ has enough projectives.
  \item\label{it:after_enough} Denote by $\mathcal{P}_{\mathcal{A}}\subseteq
    \mathcal{A}$ and $\mathcal{P}_{\mathcal{B}}\subseteq \mathcal{B}$ the
    corresponding subcategories of projectives. Let $\mathcal{N}\subseteq
    \mathcal{B}$ be a subcategory containing $i^*(P)$ for all $P\in
    \mathcal{P}_{\mathcal{A}}$. Assume that $i_*\colon \mathcal{B}\to
    \mathcal{A}$ is a contravariantly $\mathcal{N}$-relative $m$-homological
    embedding. Then, $i_*$ is an $m$-homological embedding.
  \end{enumerate}

\end{proposition}
\begin{proof}
  (i) Let $Y\in\mathcal{B}$. Since $\mathcal{A}$ has enough projectives, there
  exists an epimorphism $P\twoheadrightarrow i_*(Y)$ in $\mathcal{A}$. Since
  $i^*$ is right exact and sends projective objects in $\mathcal{A}$ to
  projective objects in $\mathcal{B}$, this yields an epimorphism
  $i^*(P)\twoheadrightarrow i^*i_*(Y)\cong Y$. This shows that $\mathcal{B}$ has
  enough projectives.

  (ii) According to Theorem 3.9 and Proposition 3.3 in \cite{Psa14}, for the
  morphism $i_*\colon \mathcal{B}\to \mathcal{A}$ to be an $m$-homological
  embedding is equivalent to $\Ext_{\mathcal{A}}^j(F(P),i_*(B))=0$ for all $B\in
  \mathcal{B}, P\in \mathcal{P}_{\mathcal{A}}, 0\leq j\leq m-1$ where
  $F(P):=\operatorname{Im}(j_!j^*(P)\to P)$ is induced by the counit of the
  adjunction.

  Applying $\Hom_{\mathcal{A}}(-,i_*(B))$ to the short exact sequence
  \[
    0\to F(P)\to P\to i_*i^*(P)\to 0
  \]
  yields the long exact sequence
  \[
    \begin{tikzcd}[column sep=tiny, row sep=tiny]
      0\rar&\Hom_{\mathcal{A}}(i_*i^*(P),i_*(B)) \arrow{r}{}
      &\Hom_{\mathcal{A}}(P,i_*(B))\rar&\Hom_{\mathcal{A}}(F(P),i_*(B))&{}\\
      {}\rar&\Ext_{\mathcal{A}}^1(i_*i^*(P),i_*(B))\rar&0\rar&\Ext_{\mathcal{A}}^1(F(P),i_*(B))\rar&{}\\
      &&\cdots\\
      \rar&\Ext_{\mathcal{A}}^{m}(i_*i^*(P),i_*(B))\rar&0
    \end{tikzcd}
  \]
  As $i_*$ is fully-faithful and $(i^*,i_*)$ is an adjoint pair, it follows that
  the first map in this long exact sequence is an isomorphism. Hence, the second
  map is $0$. Thus, $\Ext_\mathcal{A}^{j+1}(i_*i^*(P),i_*(B))\cong
  \Ext_\mathcal{A}^j(F(P),i_*(B))$ for all $j\geq 0$.

  Since $P\in \mathcal{P}_{\mathcal{A}}$, it follows that $i^*(P)\in
  \mathcal{P}_{\mathcal{B}}$. As, by assumption, $i^*(P)\in \mathcal{N}$, it
  follows that $\Ext^{j+1}_{\mathcal{A}}(i_*i^*(P),i_*(B))\cong
  \Ext^{j+1}_{\mathcal{B}}(i^*(P),B)=0$ for all $0\leq j\leq m-1$. We conclude
  that $\Ext_\mathcal{A}^j(F(P),i_*(B))=0$ for all $0\leq j\leq m-1$.
\end{proof}

\begin{corollary}
  \th\label{the-acor} Under the assumptions of
  \th\ref{lemma:d-CT:idempotent_reduction}, the embedding
  \[
    \begin{tikzcd}[column sep=small]
      \mmod \underline{\mathcal{A}}_{\mathcal{X}}\rar[hook]&\mmod \mathcal{A}
    \end{tikzcd}
  \]
  is a $(d-1)$-homological embedding.
\end{corollary}

\begin{remark}
  In the setting of \th\ref{the-acor}, we note that in general the embedding
  \[
    \begin{tikzcd}[column sep=small]
      \mmod \underline{\mathcal{A}}_{\mathcal{X}}\rar[hook]&\mmod \mathcal{A}
    \end{tikzcd}
  \]
  is not a $d$-homological embedding.
\end{remark}

\section{Nakayama algebras associated to unbounded Kupisch series\protect\footnote{by Sondre Kvamme}}

In this appendix, the restriction on Kupisch series to be bounded is removed. 

\begin{definition}
A \emph{Kupisch series of type} $\mathbb{A}^{\infty}_{\infty}$ is an infinite tuple $\underline{\ell}=(\cdots, \ell_{-1},\ell_0,\ell_1,\cdots )$ where $\ell_i$ is a nonnegative integer or $\ell_i=\infty$, satisfying the following inequalities
\[
\ell_i\leq \ell_{i-1}+1
\] 
for all $i\in \mathbb{Z}$ (with the obvious interpretation for $\ell_i=\infty$ or $\ell_{i-1}=\infty$ ).
\end{definition} 

\begin{definition} Let $\underline{\ell}$ be a Kupisch series of type $\mathbb{A}^{\infty}_{\infty}$. Define
\begin{align*}\mathbf{os}^{d+1}_{\underline{\ell}}:= \{\mathbf{\lambda}\in \mathbf{os}^{d+1}\mid \len(\mathbf{\lambda})\leq \ell_{\lambda_{d+1}}\}
\end{align*}
\begin{align*}
A^{(d+1)}_{\underline{\ell}}= A^{(d+1)}_{\infty}/[\mathbf{os}^{d+1}\setminus\mathbf{os}^{d+1}_{\underline{\ell}}]
\end{align*}
\begin{align*}
\mathcal{M}^{(d)}_{\underline{\ell}}:= \add \{M(\mathbf{\lambda})\mid \mathbf{\lambda}\in \mathbf{os}^{d+1}_{\underline{\ell}}\}
\end{align*}
With the obvious interpretation for $\ell_i=\infty$.
\end{definition}
Then as in Proposition \ref{prop:An:l:M}  one obtains that 
\[
\mathcal{M}^{(d)}_{\underline{\ell}}\subseteq \mmod A^{(d)}_{\underline{\ell}}.
\]
Note that the set $\operatorname{KS}(\mathbb{A}^{\infty}_{\infty})$ of Kupisch series of type $\mathbb{A}^{\infty}_{\infty}$ has a partial order coming from the product order. If $\underline{\ell}\leq \underline{\ell'}$, then we get an inclusion
\[
\mmod A^{(d)}_{\underline{\ell}} \to \mmod A^{(d)}_{\underline{\ell'}}
\]
and it is easy to see that 
\[
\mathcal{M}^{(d)}_{\underline{\ell}} = \mathcal{M}^{(d)}_{\underline{\ell'}}\cap \mmod A^{(d)}_{\underline{\ell}}
\]
Furthermore, if $\underline{\ell}^1\leq\underline{\ell}^2\leq\cdots$ is an
increasing sequence in $\operatorname{KS}(\mathbb{A}^{\infty}_{\infty})$ which converges to
$\underline{\ell}$ (in the obvious way), then we have that
\begin{align*}
& \mmod A^{(d)}_{\underline{\ell}} = \bigcup_{i\geq 0} \mmod A^{(d)}_{\underline{\ell^i}} \\
& \mathcal{M}^{(d)}_{\underline{\ell}} = \bigcup_{i\geq 0} \mathcal{M}^{(d)}_{\underline{\ell^i}}
\end{align*}
We show that convergence preserve $d$-cluster tilting subcategories.

\begin{lemma}
If $\mathcal{M}^{(d)}_{\underline{\ell^i}}\subseteq \mmod A^{(d)}_{\underline{\ell^i}}$ is $d\mathbb{Z}$-cluster tilting for all $i>0$, then $\mathcal{M}^{(d)}_{\underline{\ell}}\subseteq \mmod A^{(d)}_{\underline{\ell}}$ is $d\mathbb{Z}$-cluster tilting.
\end{lemma}

\begin{proof}
Obviously $\mathcal{M}^{(d)}_{\underline{\ell}}\subseteq \mmod A^{(d)}_{\underline{\ell}}$ is generating and cogenerating. We show rigidity. Let $X,Y\in \mathcal{M}^{(d)}_{\underline{\ell}}$, and let $\delta \in \Ext^j_{A^{(d)}_{\underline{\ell}}}(X,Y)$ be an arbitrary element where $j\notin d\mathbb{Z}$. Then there exists $i\geq 0$ such that $X$, $Y$, and all the middle terms in $\delta$ are in $\mmod A^{(d)}_{\underline{\ell^i}}$. Since $X,Y\in \mathcal{M}^{(d)}_{\underline{\ell}}\cap \mmod A^{(d)}_{\underline{\ell^i}}= \mathcal{M}^{(d)}_{\underline{\ell^i}}$, it follows that $\delta$ is trivial in the sense of Yoneda in $\mmod A^{(d)}_{\underline{\ell^i}}$, and therefore also in $\mmod A^{(d)}_{\underline{\ell}}$. This proves rigidity.

Now assume $X\in \mmod A^{(d)}_{\underline{\ell}}$ and $\Ext^j_{A^{(d)}_{\underline{\ell}}}(X,\mathcal{M}^{(d)}_{\underline{\ell}})=0$ for all $1\leq j\leq d-1$. Choose $i$ such that $X\in \mmod A^{(d)}_{\underline{\ell^i}}$, and choose an exact sequence
\[
M_d\to M_{d-1}\to \cdots \to M_1\to X\to 0
\]
with $M_j\in \mathcal{M}^{(d)}_{\underline{\ell^i}}\subset \mathcal{M}^{(d)}_{\underline{\ell}}$ for all $1\leq j\leq d$. Then by assumption on rigidity, we know that the sequence
\begin{multline*}
0\to \Hom_{A^{(d)}_{\underline{\ell}}}(X,Y)\to \Hom_{A^{(d)}_{\underline{\ell}}}(M_1,Y)\to \cdots \to \Hom_{A^{(d)}_{\underline{\ell}}}(M_d,Y)
\end{multline*}
is exact for $Y\in \mathcal{M}^{(d)}_{\underline{\ell^i}}$. Since $\Hom_{A^{(d)}_{\underline{\ell}}}(M,N)= \Hom_{A^{(d)}_{\underline{\ell^i}}}(M,N)$ for modules $M,N\in \mmod A^{(d)}_{\underline{\ell^i}}$, it follows that the sequence 
\begin{multline*}
0\to \Hom_{A^{(d)}_{\underline{\ell^i}}}(X,Y)\to \Hom_{A^{(d)}_{\underline{\ell^i}}}(M_1,Y)\to \cdots \to \Hom_{A^{(d)}_{\underline{\ell^i}}}(M_d,Y)
\end{multline*}
is exact. This implies that $\Ext^j_{A^{(d)}_{\underline{\ell^i}}}(X,Y)=0$ for all $1\leq j \leq d-1$, and since $Y\in \mathcal{M}^{(d)}_{\underline{\ell^i}}$ was arbitrary, we get that $X\in \mathcal{M}^{(d)}_{\underline{\ell^i}}\subseteq \mathcal{M}^{(d)}_{\underline{\ell}}$. The proof that $\Ext^j_{A^{(d)}_{\underline{\ell}}}(\mathcal{M}^{(d)}_{\underline{l}},X)=0$ for all $1\leq j\leq d-1$ implies $X\in \mathcal{M}^{(d)}_{\underline{\ell}}$ is dual.

Finally, we prove that $\mathcal{M}^{(d)}_{\underline{\ell}}$ is functorially finite. Let $X\in \mmod A^{(d)}_{\underline{\ell}}$ be arbitrary. Choose $i$ such that $X\in \mmod A^{(d)}_{\underline{\ell^i}}$, and choose an exact sequence
\[
0\to M_{d}\to \cdots\to M_1\to X\to 0
\]
 with $M_j\in \mathcal{M}^{(d)}_{\underline{\ell^i}}\subseteq \mathcal{M}^{(d)}_{\underline{\ell}}$ for all $1\leq j\leq d$ (this is possible since $\mathcal{M}^{(d)}_{\underline{\ell^i}}$ is $d$-cluster tilting). Since $\mathcal{M}^{(d)}_{\underline{\ell}}$ is rigid in degrees $1$ to $d-1$, it follows that the map $M_1\to X$ is a right $\mathcal{M}^{(d)}_{\underline{\ell}}$-approximation. Hence, $\mathcal{M}^{(d)}_{\underline{\ell}}$ is contravariantly finite. Covariantly finiteness is proved dually. 
\end{proof}

We say that a Kupisch series $\underline{\ell}$ is finite if $\ell_i\neq 0$ for
only finitely many integers $i$. Note that in this case
$A^{(d)}_{\underline{\ell}}$ is just a product of higher Nakayama algebras of
type $\mathbb{A}$, and it is known that
$\mathcal{M}^{(d)}_{\underline{\ell}}\subseteq \mmod A^{(d)}_{\underline{\ell}}$
is $d\mathbb{Z}$-cluster tilting, see \th\ref{thm:An:l}.

\begin{theorem}
  The subcategory $\mathcal{M}_{\underline{\ell}}^{(d)} \subseteq \mmod
  A_{\underline{\ell}}^{(d)}$ is $d\mathbb{Z}$-cluster tilting for every
  Kupisch series $\underline{\ell}$ of type $\mathbb{A}_{\infty}^{\infty}$.
\end{theorem}
\begin{proof}
  Since every Kupisch series of type $\mathbb{A}_{\infty}^{\infty}$ can be written as an
  increasing convergent sequence of finite Kupisch series, the claim follows
  from the previous lemma.
\end{proof}


\bibliographystyle{amsalpha}%
\bibliography{library}%


\end{document}